\RequirePackage{fix-cm}
\documentclass[smallextended]{svjour3}       % onecolumn (second format)
\smartqed  % flush right qed marks, e.g. at end of proof
\usepackage{amssymb,graphics,amsmath,amsopn,amstext,color,fullpage,fancybox,hyperref,bm,url}
\usepackage{subfig,placeins} % make it possible to include more than one captioned figure/table in a single float
\usepackage{graphicx}
\usepackage{verbatim}
\usepackage{algorithm}
\usepackage{algpseudocode}
\usepackage{longtable}
\usepackage{multirow}

\newcommand{\gap}{\vspace{0.1in}}
\newcommand{\epc}{\hspace{1pc}}

\newtheorem{assumption}{Assumption}

\title{Nonsmooth Composite  Matrix Optimization: Strong Regularity, Constraint Nondegeneracy and Beyond%\thanks{Grants or other notes
}
\author{Ying Cui \and Chao Ding}

\institute{Y. Cui \at
              Department of Industrial and Systems Engineering, University of Southern California, Los Angeles,
		CA, U.S.A.  \\
\email{yingcui@usc.edu} \and
C. Ding \at
              Institute of Applied Mathematics,  Chinese Academy of Sciences, Beijing, P.R. China. \\
              \email{dingchao@amss.ac.cn}           %  \\
}

\date{}
\begin{document}
\maketitle

\begin{abstract}
The nonsmooth composite matrix optimization   problem (CMatOP), in particular, the matrix norm minimization problem,  is a generalization of the matrix conic programming problem with wide applications in numerical linear algebra, computational statistics and engineering. 
This paper is devoted to the characterization of the strong regularity  for the CMatOP via the generalized  strong second-order sufficient condition and constraint nondegeneracy for problems with nonsmooth objective functions. The derived result supplements the existing characterization of the strong regularity for the constrained optimization problems with twice continuously differentiable data.
\end{abstract}

\gap

\noindent
{\bf AMS Class:}
	 65K05, 90C25, 90C31

%\end{document}

\section{Introduction}\label{sec: introduction}

Matrix conic programming 
is a class of optimization problems with  matrix cone constraints, in particular, the positive semidefinite constraint.
Being an extension of the classical nonlinear programming, this subject has now grown into a fruitful discipline in optimization,  including deep and rich mathematical theory, a bunch of efficient and robust solvers \cite{Sturm99,TohToddTutuncu99,zhao2010newton,YangSunToh15} and a wide range of important applications in combinatorial optimization \cite{Boyd94} and control theory \cite{Alizadeh95}. 

A natural generalization of the matrix conic programming problem is the matrix norm  minimization problem.
Starting from the nuclear norm formulation of the low rank matrix completion problem \cite{CandesTao09,RechtFazelParrilo10}, there is a growing list of algorithms and applications in such nonsmooth matrix optimization problems that also involve the spectral norm or the general matrix Ky Fan $k$-norm function  \cite{Watson1993,TohTrefethen98,BoydDiaconisParriloXiao09} in the objective. Denote $\mathbb{R}^n$ as the real $n$-dimensional space and $\mathbb{S}^n$ as the set of all $n\times n$ symmetric matrices.
A general form of the nonsmooth composite matrix optimization problems (CMatOPs) can be written as
\begin{equation}\label{opt}
\begin{array}{cl}
\displaystyle\operatornamewithlimits{minimize}_{{\bf x}\in \mathbb{X}} &\; \Phi({\bf x})\,\triangleq \, f({\bf x})  + \phi\circ \lambda(g({\bf x}))\\[0.1in] 
\mbox{subject to} &\; h({\bf x})=0,
\end{array}
\end{equation}
where $\mathbb{X}$ and $\mathbb{Y}$ are two given finite dimensional Euclidean spaces, $f:\mathbb{X}\to \mathbb{R}$ is a twice continuously differentiable function,  $g:\mathbb{X}\to \mathbb{S}^{n}$ and $h:\mathbb{X}\to \mathbb{Y}$ are twice continuously differentiable mappings,   % and containing $\{\mathcal{F}x\mid b-\mathcal{A}x \in\mathcal{Q}^\circ,\, x\in\mathbb{X}\}$,
and $\phi:\mathbb{R}^n\to (-\infty, +\infty]$ is a  symmetric function (i.e., for any $u\in \mathbb{R}^n$, $\phi(Pu) = \phi(u)$ for any $n\times n$ permutation matrix $P$). Here $\lambda(\,\bullet\,)$ denotes the vector of eigenvalues for a symmetric matrix with the components being arranged in the non-increasing order. Obviously the function $\phi\circ \lambda$ only depends on the spectrum of a given matrix, and is thus called the spectral function in the literature. There is a one-to-one correspondence between the spectral function and the so-called orthogonal-invariant matrix function, i.e., the matrix function that is invariant under orthogonal similarity transformations \cite{Lewis1995}.
% which obviously includes any unitarity invariant matrix norm.
Notice that if the function $\phi$ is taken to be the indicator function over the nonnegative orthant, the problem \eqref{opt} reduces to the nonlinear semidefinite programming problem.

One fundamental concept in the sensitivity analysis and perturbation theory is the so-called {\it strong regularity}, which is originally introduced by Robinson \cite{Robinson80} for generalized equations;
% as a criterion for the solution set to be a singleton depending Lipschitz continuously on the parameters; 
see Section \ref{sec: main results} for its definition. The  Karush-Kuhn-Tucker (KKT) optimality condition of the conventional nonlinear programming problem with twice continuously differentiable data can be formulated as a special generalized equation, whose strong regularity at a KKT solution is known to be equivalent to the strong second-order sufficient condition and the constraint nondegeneracy (when restricted to the nonlinear programming problem, the constraint nondegeneracy reduces to the linear independence constraint qualification) \cite{Robinson80,dontchev1996}; see also \cite[Proposition 5.38]{BShapiro00}. This result has been further generalized to the nonlinear semidefinite programming problems in \cite{Sun06}.
For a general class of $C^2$-reducible problems where
the constraint nondegeneracy holds, the strong regularity is further proved to be  equivalent to the Lipschitzian full stability  \cite{MNR2014}.

The non-polyhedrality of the matrix cone 
distinguishes the nature of the matrix conic programming from the classical nonlinear programming, where the constraints of the latter problems are given by finitely representable equalities and inequalities. As one can expect, such a distinction is carried to the sensitivity analysis of the CMatOPs and makes it a worthwhile effort for a deep investigation. In contrast to the relative long history of the theoretical research of the matrix conic programming, the sensitivity analysis of the nonsmooth composite matrix optimization program still stays at an early stage. It turns out that one can rewrite \eqref{opt} via its epigraphical formulation 
\begin{equation}\label{eq:epi}
\begin{array}{cl}
\displaystyle\operatornamewithlimits{minimize}_{X\in\mathbb{X},\, t\in \mathbb{R}} & \;f(X)  + t \\[0.1in] 
\mbox{subject to}  &\; h(X) =0, \epc (g(X), t)\in {\rm epi}\, (\phi\circ \lambda),
\end{array}
\end{equation}
which transforms the original problem into a matrix conic programming problem. However, such a transformation itself is inadequate for drawing the whole picture of the sensitivity results of the CMatOPs, with the following two reasons. One,  besides that of the semidefinite programming, the characterization of the strong regularity for other matrix conic programming problems, such as the one involves the epigraph of the Ky Fan $k$-norm cone, is in fact unknown. Two, the reformulation in \eqref{eq:epi} lifts the original problem from $\mathbb{X}$ to $\mathbb{X}\times \mathbb{R}$. Even if given the answer raised by the first point, it still needs the effort to bring those characterization back to the space $\mathbb{X}$.

In this paper, we characterize the strong regularity of the solution to the KKT system of  \eqref{opt} for the case where $\phi$ is piecewise affine. To accomplish this task, we  study  nonsmooth counterparts of the second-order sufficient condition and the constraint nondegeneracy via the second-order variational analysis of the spectral functions.
The adopted approach is a departure from \cite{MNR2014} that based on the second-order subdifferential of the extended value function $\Phi(X) + \delta_{Q}(h(X))$, where $\delta_{Q}(h(\,\bullet\,))$ is the indicator function of $h(\,\bullet\,)$ over $Q$, i.e., $\delta_{Q}(h(X))$ equals to $0$ if $h(X)\in Q$ and $+\infty$ otherwise. With a main focus on the characterization of the full stability in the above mentioned reference, the resulting equivalent conditions involves the limiting coderivative of  ${\rm epi}\, (\phi\circ \lambda)$, whose calculation itself might be complicated.

The rest of the paper is organized as follows. Section \ref{sec: preliminaries} summarizes some useful variational properties of eigenvalues and piecewise affine functions. In Section \ref{sec: spectral functions}, we investigate the properties of  proximal mappings associate with spectral functions that are important to the subsequent analysis.
The main result of this paper on the characterization of the strong regularity   for the CMatOPs is presented in Section \ref{sec: main results}. An example of CMatOPs involving the largest eigenvalue of a symmetric matrix is used to illustrate the derived results in Section \ref{sec:example}. We conclude our paper in the final section.

Unless otherwise specified in the paper, we use plain small Latin letters (e.g., $x$ and $y$)  to represent scalars,  small Latin letters in boldface (e.g., ${\bf x}$) to represent vectors and capital Latin letters (e.g., $X$) to  represent matrices. We also use Greek letters (e.g., $\alpha$, $\beta$ and $\iota$)  to denote the index sets and  blackboard bold letters  (e.g., $\mathbb{R}^n$ and $\mathbb{O}^n$) to denote spaces or sets.  For  $X\in\mathbb{R}^{n\times n}$, ${\rm diag}(X)$ denotes the column vector consisting of  all the
diagonal entries of $X$ being arranged from
the first to the last. For  ${\bf x}\in\mathbb{R}^n$,
${\rm Diag}({\bf x})$  denotes  the {$n\times n$} diagonal matrix whose $i$-th diagonal entry is ${\bf x}_i$ for $i=1,\ldots,n$. We write $\mathbb{O}^n$ as the set of all  $n\times n$ orthogonal matrices, and ${\bf e}_n$ as the $n$-dimensional vector of all ones.

%%%%%%%%%%%%%%%%%%%%%%%%%%%%%%%%%%%%%%%%%%%%%%%%
\section{Preliminaries and background results}
\label{sec: preliminaries}

%%%%%%%%%%%%%%%%%%%%%%%%%%%%%%%%%%%%%%%%%%%%%%%%%%%
\subsection{Variational analysis of the eigenvalues}
Let $X\in \mathbb{S}^n$ be an arbitrary symmetric matrix. Suppose that $X$ has the following eigenvalue decomposition
\begin{equation}\label{eq:eig-decomp}
X=U\,{\rm Diag}\left(\lambda_1(X), \cdots, \lambda_n(X)\right)\,U^{\top},
\end{equation}
where $\lambda_1(X)\ge \cdots\ge \lambda_n(X)$ are the eigenvalues of $X$ arranged in the non-increasing order and $U$ is a matrix of the corresponding orthonormal eigenvectors.
We denote the set of all matrices $U$ satisfying (\ref{eq:eig-decomp}) as
$\mathbb{O}^n(X)$. We also use $\lambda(X)$ to denote the vector whose $i$-th entry is $\lambda_i(X)$. Let $v_{1}(X)>v_{2}(X)>\ldots>v_{r}(X)$ be the distinct eigenvalues of $X$ arranged in the decreasing order. Define
\begin{equation}\label{eq:ak-symmetric}
\alpha^{l} :=\left\{i\in \{1,\ldots,n\}\,|\, \lambda_{i}(X)={v}_{l}(X)\right\},\quad l=1,\ldots,r.
\end{equation} For each $  i\in \{1,\ldots,n\}$, we denote $k_{i}(X)$ as the number of eigenvalues that  equal to $\lambda_i(X)$ but are ranked  before $i$ (including $i$) and $o_{i}(X)$ as the
number of eigenvalues that equal to $\lambda_i(X)$ but are ranked  after $i$ (excluding $i$). That is, the scalars $k_{i}(X)$ and  $o_{i}(X)$  satisfy
\begin{eqnarray}
&&\lambda_{1}(X)\geq\ldots\geq\lambda_{i-k_{i}(X)}(X)>\lambda_{i-k_{i}(X)+1}(X)=\ldots=\lambda_{i}(X)=\ldots=\lambda_{i+o_{i}(X)}(X)\nonumber\\ 
&&>\lambda_{i+o_{i}(X)+1}(X)\geq\ldots\geq\lambda_{n}(X).\label{eq:symmetric-l_i}
\end{eqnarray}
In the subsequent discussions, when the dependence of $k_{i}$ and $o_{i}$  on $X$ can be easily
seen from the context, we often drop $X$ for simplicity.

In the following, we summarize some results about the properties of the eigenvalues that are essential in our subsequent
discussions. The first result is Ky Fan's inequality \cite{Fan49}.
\begin{lemma}\label{lem:Fan}
	Let   $Y$ and $Z$ be  two matrices in $\mathbb{S}^{n}$. Then
\[
	\langle Y,Z\rangle\leq\lambda(Y)^{\top}\lambda(Z)\, ,
\]
	where the equality holds if and only if $Y$ and $Z$ admit a simultaneous ordered eigenvalue
	decomposition, i.e., there exists an orthogonal matrix $U\in\mathbb{O}^{n}$ such that
	\[
	Y=U\Lambda(Y)U^{\top}\quad {\rm and} \quad Z=U\Lambda(Z)U^{\top}.
	\]
\end{lemma}

%\vskip 10 true pt

The next lemma is about the directional differentiability of the eigenvalue function, which can be found in, for example, \cite[Theorem 7]{Lancaster64} and \cite[Proposition 1.4]{Torki01}.

\begin{lemma}\label{prop:eigenvalue_diff}
	Let $X\in \mathbb{S}^n$   have the eigenvalue decomposition in (\ref{eq:eig-decomp}).
	Then for any $\mathbb{S}^n\ni H\to 0$, we have
\[
	\lambda_{i}(X+H)-\lambda_{i}(X)-\lambda_{k_{i}}({U}_{\alpha^{l}}^{\top} H{U}_{\alpha^{l}})=O(\|H\|^{2}),
	\quad i\in \alpha^{l}, \ l=1,\ldots,r,
\]
	where for each $i\in\{1,\ldots,n\}$, $k_{i}$ is defined in (\ref{eq:symmetric-l_i}).
	Hence,  for any given direction $H\in \mathbb{S}^{n}$, the eigenvalue function $\lambda_{i}(\cdot)$ is
	directionally differentiable at $X$  with the directional derivative
	$\lambda'_{i}(X;H)=\lambda_{k_{i}}({U}_{\alpha^{l}}^{\top} H{U}_{\alpha^{l}})$ for any $i\in \alpha^{l}$, $l=1,
	\ldots, r$.
\end{lemma}

Let $l\in\{1,\ldots,r\}$ be fixed. Consider the following eigenvalue decomposition
of the symmetric matrix $U_{\alpha^{l}}^{\top} HU_{\alpha^{l}}\in\mathbb{S}^{|\alpha^{l}|}$:
\[
U_{\alpha^{l}}^{\top} HU_{\alpha^{l}}=R\Lambda(U_{\alpha^{l}}^{\top} HU_{\alpha^{l}})R^{\top},
\]
where $R\in\mathbb{O}^{|\alpha^{l}|}$. Denote the distinct eigenvalues of $U_{\alpha^{l}}^{\top} HU_{\alpha^{l}}$ by
$\tilde{v}_{1}>\tilde{v}_{2}>\ldots>\tilde{v}_{\tilde{r}}$. Define
\[
\tilde{\alpha}^{j}:=\left\{\,i\in \{1,\ldots,|\alpha^{l}|\}\;|\; \lambda_{i}(U_{\alpha^{l}}^\top  HU_{\alpha^{l}})=\tilde{v}_{j}\,\right\},\quad j=1,\ldots,\tilde{r}.
\]
 For each $i\in \alpha^{l}$, let $\tilde{k}_{i}\in\{1,\ldots,|\alpha^{l}|\}$ and
$\tilde{l}\in\{1,\ldots,\tilde{r}\}$ be such that
\[
\tilde{k}_{i}:=k_{k_{i}}(U_{\alpha^{l}}^{\top} HU_{\alpha^{l}}) \quad {\rm and} \quad
\tilde{k}_{i}\in\tilde{\alpha}^{\tilde{l}},
\]
 where $k_{i}$ is defined by (\ref{eq:symmetric-l_i}).

Let $\mathbb{Z}$ and $\mathbb{Z}'$ be two finite dimensional real Euclidean spaces. We say that a function $\Phi:\mathbb{Z}\to\mathbb{Z}'$ is (parabolic) second-order directionally differentiable at ${\bf z}\in\mathbb{Z}$, if $\Phi$ is directionally differentiable at ${\bf z}$ and for any ${\bf h}, {\bf w}\in\mathbb{Z}$,
\[
\lim_{t\downarrow 0}\frac{\Phi({\bf z}+t{\bf h}+\frac{1}{2}t^2{\bf w})-\Phi({\bf z})-t\, \Phi^{\,\prime}({\bf z};{\bf h})}{\frac{1}{2}t^2}\quad \mbox{exists.}
\]
In this case, the above limit is said to be the (parabolic) second-order directional derivative of $\Phi$ at ${\bf z}$ along the directions ${\bf h}$ and ${\bf w}$, which we denote as $\Phi''({\bf z};{\bf h},{\bf w})$. The following proposition, which has its source from \cite[Proposition 2.2]{Torki01}, provides an explicit formula of the (parabolic)
second-order directional derivative of the eigenvalue function. 

\begin{lemma}\label{prop:second-directional-diff-eigenvalue}
	Let $X\in\mathbb{S}^{n}$ have the eigenvalue decomposition \eqref{eq:eig-decomp}. Then for any 
	$H,W\in\mathbb{S}^{n}$, 
\[
	\lambda''_{i}(X;H,W)=\lambda_{\tilde{k}_{i}}\left(R_{\tilde{\alpha}^{\tilde{l}}}^{\top} \, U_{\alpha^{l}}^{\top}\left[W-2H(X-\lambda_{i}I_{n})^{\dag}H\right]U_{\alpha^{l}}R_{\tilde{\alpha}^{\tilde{l}}}\right),
	\quad i\in \alpha^{l},\; l\in\{1,\ldots,r\},
\]
	where $Z^{\dag}\in \mathbb{R}^{p\times p}$ is the Moore-Penrose pseudo-inverse of the square matrix $Z\in\mathbb{R}^{p\times p}$.
\end{lemma}

%%%%%%%%%%%%%%%%%%%%%%%%%%%%%%%%%%%%%%%%%%%%%%%%%%%%%%%
\subsection{Properties of convex piecewise affine functions}
Let $\phi:\mathbb{R}^n\to(-\infty, +\infty]$ be a proper convex piecewise affine function, i.e., $\phi$ is a convex function whose nonempty effective domain can be represented as the union of finitely many polyhedral sets, relative to each of which $\phi(x)$ is an affine function (cf. \cite[Definition 2.47]{rwets1998}). Such a function is also called a polyhedral convex function by Rockafellar~\cite[Section 19]{rockafellar1970}. Let $\phi$ be a proper convex piecewise affine function with the polyhedral effective domain 
\begin{equation}\label{eq:domphi}
	{\rm dom}\,\phi:=\left\{x\in\mathbb{R}^n\mid \psi(x):=\displaystyle\max_{1\le i\le q} \{ \langle {\bf b}^i,{\bf x}\rangle - d_{i}\}\leq 0 \right\}
\end{equation}
for some $\{({\bf b}^i, d_i)\in \mathbb{R}^n\times \mathbb{R}\}_{i=1}^q$ with a positive integer $q$. It is known  from~\cite[Theorem 2.49]{rwets1998} that $\phi$ can be expressed in the form of
\begin{equation}\label{eq:phisumform}
\phi({\bf x}) \, = \, \underbrace{\max_{1\le i\le p}\left\{\langle {\bf a}^i,{\bf x}\rangle -c_i\right\}}_{\mbox{denoted as $\phi_1({\bf x})$}} + \underbrace{\delta_{{\rm dom}\,\phi}({\bf x})}_{\mbox{denoted as $\phi_2({\bf x})$}},\quad {\bf x}\in\mathbb{R}^n
\end{equation}
for some $\{ ({\bf a}^i, c_i)\in \mathbb{R}^n\times \mathbb{R}\}_{i=1}^p$  with a  positive integer $p$.
 We call $\{{\bf a}^i\}_{i=1}^p$ and $\{{\bf b}^i\}_{i=1}^q$  the bases of $\phi_1$ and $\phi_2$, respectively. 

 We denote the set of all $n\times n$ permutation matrices as $\mathbb{P}^n$.
Recall that the function $\phi$ is called symmetric over $\mathbb{R}^n$ if for any $Q\in \mathbb{P}^{n}$ and any ${\bf x}\in\mathbb{R}^n$, it holds that $\phi({\bf x})=\phi(Q{\bf x})$. The following proposition characterizes the symmetric piecewise affine function.

\begin{proposition}\label{prop:pwlc-sym}
	Let $\phi:\mathbb{R}^n\to(-\infty,\infty]$ be a given proper convex piecewise affine function. Then the function $\phi$ is symmetric over $\mathbb{R}^n$ if and only if its decomposed components $\phi_1:\mathbb{R}^n\to\mathbb{R}$ and $\phi_2:\mathbb{R}^n\to (-\infty,\infty]$ in \eqref{eq:phisumform} satisfy the following conditions:
	\begin{equation}\label{eq:pwlc-sym}
	\left\{\begin{array}{ll}
	\phi_1({\bf x})= \displaystyle\max_{1\le i\le p}\left\{\max_{Q\in \mathbb{P}^n}\left\{ \langle Q{\bf a}^i,{\bf x}\rangle-c_i \right\} \right\}, \\[0.2in]
	  {\rm dom}\,\phi=\left\{{\bf x}\in\mathbb{R}^n \, \mid  \,\displaystyle\displaystyle{\max_{1\le i\le q}}\left\{\max_{Q\in \mathbb{P}^n}\left\{ \langle Q{\bf b}^i,{\bf x}\rangle-d_i \right\} \right\} \le 0 \right\},
	\end{array}\right.
	\quad \forall\;{\bf x}\in\mathbb{R}^n.
	\end{equation}
\end{proposition}
\begin{proof}
	``$\Longleftarrow$"
	Suppose that $\phi_1$ and $\phi_2$ satisfy the conditions in \eqref{eq:pwlc-sym}. Consider any ${\bf x}\in\mathbb{R}^n$ and  $Q'\in\mathbb{P}^n$. If ${\bf x}\notin {\rm dom}\,\phi$, then there exist ${\bf b}^{\hat{i}}\in\mathbb{R}^n$, ${d}_{\hat{i}}\in\mathbb{R}$ and $\widehat{Q}\in \mathbb{P}^n$ such that 
	$$
	\langle \widehat{Q}{\bf b}^{\hat{i}},{\bf x}\rangle-d_{\hat{i}} >0.
	$$
	Since $(Q')^\top  Q'= I$, we have
	$$
	\left\langle Q'\widehat{Q}{\bf b}^{\hat{i}}\,,\,Q' {\bf x}\right\rangle-d_{\hat{i}} >0.
	$$
	By noting that $Q'\widehat{Q}\in\mathbb{P}^n$, we conclude that $Q'{\bf x}\notin {\rm dom}\,\phi$, which implies that 
	$$
	\phi({\bf x})=\phi(Q'{\bf x})= +\infty.
	$$
	Otherwise if ${\bf x}\in {\rm dom}\,\phi$, we deduce that
	\begin{eqnarray*}
		\displaystyle{\max_{1\le i\le q}}\left\{\max_{Q\in \mathbb{P}^n}\left\{ \langle Q{\bf b}^i,Q'{\bf x}\rangle-d_i\right\} \right\}&=&\displaystyle{\max_{1\le i\le q}}\left\{\max_{Q\in \mathbb{P}^n}\left\{ \langle Q'^{\top}Q{\bf b}^i,{\bf x}\rangle-d_i \right\} \right\} \\ [3pt]
		&=& \displaystyle{\max_{1\le i\le q}}\left\{\max_{Q'^{\top}Q\in \mathbb{P}^n}\left\{ \langle Q'^{\top}Q{\bf b}^i,{\bf x}\rangle-d_i \right\} \right\}\le 0,
	\end{eqnarray*}	  
which implies that $Q'{\bf x}\in {\rm dom}\,\phi$. Moreover, we have 
	\begin{eqnarray*}
	\phi(Q'{\bf x})=\phi_1(Q'{\bf x})&=&\max_{1\le i\le p}\left\{\max_{Q\in \mathbb{P}^n}\left\{ \langle Q{\bf a}^i,Q'{\bf x}\rangle-c_i \right\} \right\}=\max_{1\le i\le p}\left\{\max_{Q\in \mathbb{P}^n}\left\{ \langle Q'^{\top}Q{\bf a}^i,{\bf x}\rangle-c_i \right\} \right\} \\ [3pt]
	&=& \max_{1\le i\le p}\left\{\max_{Q'^{\top}Q\in \mathbb{P}^n}\left\{ \langle Q'^{\top}Q{\bf a}^i,{\bf x}\rangle-c_i \right\} \right\}=\phi_1({\bf x})=\phi({\bf x}).
	\end{eqnarray*}
Thus, we know that $\phi$ is symmetric over $\mathbb{R}^n$.

``$\Longrightarrow$" Assume that $\phi$ is a proper convex piecewise affine function with the decomposition in \eqref{eq:phisumform}. Denote 
$$
{\cal Z}:=\left\{{\bf x}\in\mathbb{R}^n\; \bigg|\; \displaystyle{\max_{1\le i\le q}}\left\{\max_{Q\in \mathbb{P}^n}\left\{ \langle Q{\bf b}^i,{\bf x}\rangle-d_i \right\} \right\} \le 0 \right\}.
$$
Obviously ${\cal Z}\subseteq {\rm dom}\,\phi$. Suppose that ${\bf x}\in {\rm dom}\,\phi$ is arbitrarily chosen .  It follows from the symmetric property of $\phi$ that 
\[
\left\langle Q{\bf b}^i,{\bf x}\right\rangle-d_i= \left\langle {\bf b}^i,Q^{\top}{\bf x}\right\rangle-d_i\le 0,  \quad \forall\; Q\in \mathbb{P}^n, \quad \forall \;i=1,\ldots,q \;,
\]
which shows that ${\bf x}\in{\cal Z}$. Therefore, we have ${\rm dom}\,\phi={\cal Z}$.

Denote 
$$
\overline{\phi}(x):=\displaystyle{\max_{1\le i\le p}\left\{\max_{Q\in \mathbb{P}^n}\left\{ \langle Q{\bf a}^i,{\bf x}\rangle-c_i \right\} \right\}}, \quad {\bf x}\in  {\rm dom}\,\phi.
$$ 
It is clear that $\phi({\bf x})\le \overline{\phi}({\bf x})$ for any ${\bf x}\in  {\rm dom}\,\phi$. On the other hand, for any ${\bf x}\in {\rm dom}\,\phi$, there exist $\bar{i}\in\{1,\ldots,p\}$ and $\overline{Q}\in\mathbb{P}^n$ such that
\[
\overline{\phi}({\bf x})=\big\langle \overline{Q}{\bf a}^{\bar{i}},\,{\bf x}\big\rangle - c_{\,\bar{i}} = \big\langle {\bf a}^{\bar{i}},\,\overline{Q}^{\top}{\bf x}\big\rangle - c_{\,\bar{i}} \le \max_{1\le i\le p}\left\{\big\langle {\bf a}^{i}, \,\overline{Q}^{\top}{\bf x}\big\rangle - c_{i}\right\}=\phi(\overline{Q}^{\top}{\bf x}).
\]
Finally, since $\phi$ is symmetric over $\mathbb{R}^n$, we have $\phi(\overline{Q}^{\top}{\bf x})=\phi({\bf x}) $. Thus, we know from the above inequality that $\overline{\phi}({\bf x})\le \phi({\bf x})$ for any ${\bf x}\in {\rm dom}\,\phi$. Therefore, we know that $\phi({\bf x})=\overline{\phi}({\bf x})$ for any ${\bf x}\in {\rm dom}\,\phi$. The proof is completed. \qed
\end{proof}

\begin{remark}\label{remark:symmetric basis}
It is easy to verify from Proposition \ref{prop:pwlc-sym} that if a proper convex piecewise affine function $\phi$ is symmetric, then both of its components $\phi_1$ and $\phi_2$ in  \eqref{eq:phisumform} are symmetric.
%Without loss of generality, we may  assume that for any permutation matrix $Q\in\mathbb{P}^n$, any basis $\{{\bf a}^i\}_{i=1}^p$ of $\phi_1$ and $\{{\bf b}^i\}_{i=1}^q$ of $\phi_2$, $\{Q{\bf a}^i\}_{i=1}^p$ and $\{Q{\bf b}^i\}_{i=1}^q$ are also the respective bases  of $\phi_1$ and $\phi_2$. 	
\end{remark}

 To proceed, we denote 
$$
{\cal D}_i := \left\{{\bf x}\in {\rm dom}\,\phi\mid \langle {\bf a}^j,{\bf x}\rangle - c_j \leq\langle {\bf a}^i,{\bf x}\rangle -c_i ,\, \forall\, j=1,\ldots, p\right\}, \quad i = 1,  \ldots, p.
$$
It follows from \cite[Proposition 3.2]{msarabi2016} that ${\rm dom}\,\phi = \displaystyle\bigcup_{i=1,\ldots,p} {\cal D}_i$.
For any $\overline{\bf x}\in {\rm dom}\,\phi$, we further denote the following two  index sets:
\begin{equation}\label{index1}
\iota_1(\overline{\bf x}) := \{1\leq i\leq p\mid \overline{\bf x}\in {\cal D}_i\}  \quad {\rm and}\quad  \iota_2(\overline{\bf x}) := \{1\leq i\leq q\mid \langle {\bf b}^i,\overline{\bf x}\rangle - d_i= 0\}.
\end{equation}
%	Denote  the convex piecewise affine function $\psi:\mathbb{R}^n\to \mathbb{R}$ by 
%	\begin{equation}\label{eq:def-psi}
%	\psi({\bf x}):=\max_{1\le i\le q}\left\{ \langle {\bf b}^i,{\bf x}\rangle - d_i \right\},\quad {\bf x}\in \mathbb{R}^n.
%	\end{equation}
%	The polyhedral set ${\rm dom}\,\phi$ can be rewritten as 
%	\[
%	{\rm dom}\,\phi=\left\{ {\bf x}\in \mathbb{R}^n\mid \psi({\bf x})\le 0 \right\}.
%	\]
It is known that the pointwise-max function $\phi_1$ in \eqref{eq:phisumform} and $\psi$  in \eqref{eq:domphi} are directionally differentiable everywhere with the following directional derivatives (see, e.g., \cite[Example 2.68]{BShapiro00})
\begin{equation}\label{eq:dirdev_12}
\phi_1'(\overline{\bf x};{\bf h})=\max_{i\in\iota_1(\overline{\bf x})}\langle {\bf a}^i,{\bf h} \rangle\quad {\rm and}\quad \psi'(\overline{\bf x};{\bf h})=\max_{i\in\iota_2(\overline{\bf x})}\langle {\bf b}^i,{\bf h} \rangle,\quad {\bf h}\in\mathbb{R}^n.
\end{equation}
%Indeed, for any ${\bf z}\in\partial\phi_1(\overline{\bf x})$, there exist real numbers $\{u_i\}_{i\in \iota_1(\overline{\bf x})}$ satisfying
%$$
%0\le u_i\le 1, \quad \sum_{i\in\iota_1(\overline{\bf x})}u_i=1 \quad {\rm and}\quad {\bf z}=\sum_{i\in\iota_1(\overline{\bf x})}u_i \, {\bf a}^i,
%$$ 
%which implies that for any ${\bf h}\in\mathbb{R}^n$,
%\[
%\langle {\bf z},{\bf h}\rangle =\sum_{i\in\iota_1(\overline{\bf x})}u_i\langle{\bf a}^i,{\bf h}\rangle\le \sum_{i\in\iota_1(\overline{\bf x})}u_i\max_{i\in\iota_1(\overline{\bf x})}\langle{\bf a}^i,{\bf h}\rangle=\max_{i\in\iota_1(\overline{\bf x})}\langle{\bf a}^i,{\bf h}\rangle.
%\]
%Therefore, we know that for any ${\bf h}\in\mathbb{R}^n$,
%\begin{equation}\label{eq:dir-phi-2}
%\phi'_1(\overline{\bf x};{\bf h}) = \delta_{\partial \phi_1(\overline{\bf x})}^*({\bf h}) = \sup_{{\bf z}\in\partial\phi_1(\overline{\bf x})}\langle {\bf z},{\bf h}\rangle \le  \max_{i\in\iota_1(\overline{\bf x})}\langle {\bf a}^i,{\bf h}\rangle.
%\end{equation}
%Conversely, it is easy to see there exist $\{\hat{u}^i\}_{i\in \iota_1(\overline{\bf x})}$ with $0\le \hat{u}^i\le 1$ and $\displaystyle\sum_{i\in\iota_1(\overline{\bf x})}\hat{u}^i=1$ such that the equality in \eqref{eq:dir-phi-2} holds. 
%In addition, both $\phi_1$ and $\phi_2$ given by \eqref{eq:phisumform} are  subdifferentiable  at any $\overline{\bf x}\in {\rm dom}\,\phi$~\cite[Theorem 23.10]{rockafellar1970} with the following subgradients  (cf.~\cite[Proposition 3.3]{msarabi2016})
Denote $\partial f$ as the subgradient of a convex function $f$.
It also holds that
\begin{equation}\label{subgradient}
\partial \phi_1(\overline{\bf x})={\rm conv}\{{\bf a}^i,\ i\in\iota_1(\overline{\bf x})\} \quad {\rm and}\quad \partial \phi_2(\overline{\bf x})=\mathcal{N}_{{\rm dom}\,\phi}(\overline{\bf x})={\rm cone}\{{\bf b}^i,\ i\in\iota_2(\overline{\bf x}) \},
\end{equation}
where `${\rm conv}\,{\cal C}$' and  `${\rm cone} \,{\cal C}$' stand for the convex hall and conic hall of a given nonempty closed set ${\cal C}$, respectively (if $\iota_2(\overline{\bf x}) = \emptyset$, then $\partial\phi_2(\bar{x}) = \{0\}$).
Therefore, for any given $\overline{\bf y}\in \partial \phi_1(\overline{\bf x})$ and $\overline{\bf z}\in \partial \phi_2(\overline{\bf x})$, we are able to define the following two index sets 
\begin{equation}\label{eq:def-index-eta-1}
\left\{\begin{array}{ll}
\eta_1(\overline{\bf x},\overline{\bf y}):=\big\{i\in\iota_1(\overline{\bf x})\mid \displaystyle\sum_{i\in\iota_1(\overline{\bf x})}u_i \, {\bf a}^i=\overline{\bf y},\ \sum_{i\in\iota_1(\overline{\bf x})}u_i=1,\ 0<u_i\le 1 \big\}, 
\\[0.2in]
\eta_2(\overline{\bf x},\overline{\bf z}):=\big\{i\in\iota_2(\overline{\bf x})\mid \displaystyle\sum_{i\in\iota_2(\overline{\bf x})}u_i{\bf b}^i=\overline{\bf z}, \ u_i>0 \big\}.
\end{array}\right.
\end{equation}
The following corollary is a direct consequence of Proposition \ref{prop:pwlc-sym}.
\begin{corollary}\label{remark:Ix}
The following two statements hold.
\vskip 0.1in
\noindent
	(i) For any $i\in \iota_1(\overline{\bf x})$, $j\in\iota_2(\overline{\bf x})$ and  $Q\in\mathbb{P}^n_{\overline{\bf x}}$ (i.e., $Q\overline{\bf x}=\overline{\bf x}$),  there exist  $i'\in  \iota_1(\overline{\bf x})$ and $j'\in  \iota_2(\overline{\bf x})$ such that ${\bf a}^{i'}=Q{\bf a}^i$ and ${\bf b}^{j'}=Q{\bf b}^j$, respectively. 
\vskip 0.1in	
	\noindent
	(ii) For any $i\in \eta_1(\overline{\bf x}, \overline{\bf y})$, $j\in\eta_2(\overline{\bf x}, \overline{\bf z})$, $Q^{1}\in\mathbb{P}^n_{\overline{\bf x}}\cap\mathbb{P}^n_{\overline{\bf y}} $ and $Q^{2}\in\mathbb{P}^n_{\overline{\bf x}}\cap\mathbb{P}^n_{\overline{\bf z}} $,   there exist  $i'\in  \eta_1(\overline{\bf x}, \overline{\bf y})$ and $j'\in  \eta_2(\overline{\bf x},\overline{\bf z})$ such that ${\bf a}^{i'}=Q^{1}\,{\bf a}^i$ and ${\bf b}^{j'}=Q^{2}\,{\bf b}^j$, respectively.
\end{corollary}

Let $\mathbb{Z}$ be a  finite dimensional Euclidean space and $\varpi:\mathbb{Z}\to(-\infty,\infty]$ be a  proper closed convex function. The Moreau envelop and proximal mapping of  $\varpi$ are defined by
\begin{equation}\label{eq:MY-phi}
	\chi_{\varpi}({\bf z}):=\min_{{\bf w}\in\mathbb{Z}}\left\{\varpi({\bf w})+\frac{1}{2}\|{\bf w}-{\bf z}\|^2 \right\} \epc \mbox{and}\epc {\rm Pr}_{\varpi}({\bf z}):=\displaystyle\operatornamewithlimits{argmin}_{{\bf w}\in\mathbb{Z}}\left\{\varpi({\bf w})+\frac{1}{2}\|{\bf w}-{\bf z}\|^2 \right\}, \quad {\bf z}\in\mathbb{Z}.
\end{equation}
It is known that ${\rm Pr}_{\varpi}$ is globally Lipschitz continuous with modulus $1$ \cite[Proposition 12.19]{rwets1998}.  The directional derivative of the proximal mapping is closely related to the critical cone associated with the generalized equation $\overline{\bf z}\in\partial\varpi({\bf z})$, which is defined as
\[
{\cal C}({\bf \bar{z}};\partial\varpi({\bf z})):=\left\{{\bf d}\in\mathbb{R}^n\mid \varpi'({\bf z};{\bf d})=\langle {\bf \bar{z}} - {\bf z},\, {\bf d}\rangle  \right\}.
\]
The following proposition shows the directional derivative of the proximal mappings associated with $\partial \phi_1$ and $\partial \phi_2$. Necessary and sufficient conditions for them to be F(r\'{e}chet)-differentiable are also provided.
\begin{proposition}\label{prop:MY-phi-direction-diff}
	Let $\phi_1:\mathbb{R}^n\to\mathbb{R}$ and $\phi_2:\mathbb{R}^n\to (-\infty,\infty]$ by given by \eqref{eq:phisumform}. 
	Then the following two properties hold for the corresponding  proximal mappings ${\rm Pr}_{\phi_1}:\mathbb{R}^n\to \mathbb{R}^n$ and ${\rm Pr}_{\phi_2}:\mathbb{R}^n\to \mathbb{R}^n$.\\[0.1in]
	\noindent 
	(i) ${\rm Pr}_{\phi_1}$ and ${\rm Pr}_{\phi_2}$ are directionally differentiable everywhere with the directional derivatives  
\begin{equation}\label{eq:dir-diff-phi1}
\left\{\begin{array}{ll}
{\rm Pr}_{\phi_1}'({\bf x};{\bf h})=\displaystyle\operatornamewithlimits{argmin}_{{\bf d}\in \mathbb{R}^n}\left\{ \|{\bf d}-{\bf h}\|^2\mid {\bf d}\in {\cal C}({\bf x};\partial\phi_1({\rm Pr}_{\phi_1}({\bf x})))\right\}\\[0.2in]

	{\rm Pr}_{\phi_2}'({\bf x};{\bf h})=\displaystyle\operatornamewithlimits{argmin}_{{\bf d}\in \mathbb{R}^n}\left\{ \|{\bf d}-{\bf h}\|^2\mid {\bf d}\in {\cal C}({\bf x};\partial\phi_2({\rm Pr}_{\phi_2}({\bf x})))\right\}
	\end{array}\right. \quad {\bf x, h}\in \mathbb{R}^n.
\end{equation} 
(ii) Let the index sets  $\iota_1$, $\iota_2$, $\eta_1$ and $\eta_2$ be  given by \eqref{index1} and \eqref{eq:def-index-eta-1}, respectively. Then ${\rm Pr}_{\phi_1}$ is F-differentiable at  ${\bf x}$ if and only if
$$\eta_1\left(\,{\rm Pr}_{\phi_1}({\bf x}),\,{\bf x}-{\rm Pr}_{\phi_1}({\bf x})\,\right)=\iota_1\left(\,{\rm Pr}_{\phi_1}({\bf x})\,\right);$$ similarly, ${\rm Pr}_{\phi_2}$ is F-differentiable at ${\bf x}$  if and only if 
$$\eta_2\left(\,{\rm Pr}_{\phi_2}({\bf x}), \, {\bf x}-{\rm Pr}_{\phi_2}({\bf x}) \,\right)=\iota_2\left(\, {\rm Pr}_{\phi_2}({\bf x})\, \right).$$
 Moreover, under the above two conditions, the derivatives ${\rm Pr}_{\phi_1}'({\bf x})$ and ${\rm Pr}_{\phi_2}'({\bf y})$ are given by
\begin{equation}\label{eq:diff-phi}
\left\{\begin{array}{ll}

{\rm Pr}_{\phi_1}'({\bf x})\,{\bf h}=\displaystyle\operatornamewithlimits{argmin}_{{\bf d}\in \mathbb{R}^n}\left\{ \, \|{\bf d}-{\bf h}\|^2\mid \langle {\bf d},{\bf a}^i-{\bf a}^j\rangle =0,\ i,j\in\iota_1\left(\,{\rm Pr}_{\phi_1}({\bf x})\,\right)\,\right\} \\[0.2in]
	{\rm Pr}_{\phi_2}'({\bf x})\,{\bf h}=\displaystyle\operatornamewithlimits{argmin}_{{\bf d}\in \mathbb{R}^n}\left\{\, \|{\bf d}-{\bf h}\|^2\mid \langle {\bf d},{\bf b}^i\rangle =0,\ i\in\iota_2\left(\,{\rm Pr}_{\phi_2}({\bf x})\,\right) \, \right\}
		\end{array}\right.\quad  {\bf x, h}\in\mathbb{R}^n.
\end{equation}

\end{proposition}
\begin{proof}
Statement (i) follows from  \cite[Proposition 7.1 and Theorem 7.2]{BCShapiro98}. To prove  statement (ii), we first note that based on similar arguments of \cite[Corollary 4.1.2]{FPang2003}, ${\rm Pr}_{\phi_1}$ and ${\rm Pr}_{\phi_2}$ are F(r\'{e}chet)-differentiable at ${\bf x}$ and ${\bf y}$ if and only if the critical cones ${\cal C}({\bf x};\partial\phi_1({\rm Pr}_{\phi_1}({\bf x})))$ and ${\cal C}({\bf y};\partial\phi_2({\rm Pr}_{\phi_2}({\bf x})))$ are two linear subspaces in $\mathbb{R}^n$. The stated results then follow from \cite[Proposition 3.2]{msarabi2016b}.
\end{proof}

%%%%%%%%%%%%%%%%%%%%%%%%%%%%%%%%%%%%%%%%%%%%%%%%%%%%%%%%%%%%

\section{Variational analysis of spectral functions}
\label{sec: spectral functions}

In this section, we  study several important variational properties of the spectral function $\theta\equiv\phi\circ\lambda$ for a symmetric piecewise affine function $\phi$. According to the decomposition of $\phi = \phi_1 + \phi_2$ in \eqref{eq:phisumform}, the function $\theta$ can be decomposed as 
\begin{equation}\label{eq:thetasum}
	\theta(X) \, = \, \underbrace{\phi_1\circ\lambda(X)}_{\mbox{denoted as $\theta_1(X)$}} \, + \, \underbrace{\phi_2\circ\lambda(X)}_{\mbox{denoted as $\theta_2(X)$}}, \quad X\in \mathbb{S}^n.
\end{equation}
The following lemma on the subdifferentials of spectral functions can be found in~\cite{lewis1996,Lewis1996a},
\begin{lemma}\label{lemma:subdiff-spectral-function}
	Let $\phi:\mathbb{R}^n\to (-\infty, +\infty]$ be a proper closed convex and  symmetric function.
	 Let $X\in\mathbb{S}^n$ have the eigenvalue $\lambda(X)$ in ${\rm dom}\, \phi$. Let $W\in\mathbb{S}^n$. Then   $W\in\partial (\phi\circ \lambda)(X)$ if and only if $\lambda(W)\in\partial \phi(\lambda(X))$ and there exists $U\in\mathbb{O}^n(X)\cap \mathbb{O}^n(W)$. In fact,
	$\partial (\phi\circ \lambda)(X) =
\{U{\rm Diag}(\mu)U^\top \mid \mu\in  \partial g(\lambda(X)),\; U \in \mathbb{O}^n(X)\}$.
	
\end{lemma}

In the following two subsections, we characterize the  tangent sets, critical cones and the so-called sigma term associated with $\theta_1$ and $\theta_2$, respectively. Since the analysis of $\theta_1$ and $\theta_2$ are similar, we only give the detailed proof of  the results for $\theta_1$; the properties of $\theta_2$ are presented without proof.

\subsection{Variational properties of $\theta_1$}\label{subsection:theta1}

We first study the variational properties of the spectral function $\theta_1=\phi_1\circ\lambda$.
\vskip 10 true pt
\noindent
\underline{\bf The tangent set and its lineality space}. Let $\overline{X}\in\mathbb{S}^n$ be given. Since $\theta_1$ is Lipschitz continuous on $\mathbb{S}^n$, it follows  from \cite[Proposition 2.58]{BShapiro00} that the tangent cone ${\cal T}_{{\rm epi}\,\theta_1}(\overline{X},\theta_1(\overline{X}))$ of the epigraph ${\rm epi}\,\theta_1$ at $(\overline{X},\theta_1(\overline{X}))\in {\rm epi}\,\theta_1$ is given by
\[
{\cal T}_{{\rm epi}\,\theta_1}(\overline{X},\theta_1(\overline{X}))={\rm epi}\,\theta_1'(\overline{X};\,\bullet\,)=\left\{(H,h)\in \mathbb{S}^n\times\mathbb{R}\mid \theta_1'(\overline{X};H)\le h\right\}.
\]
The lineality space  of ${\cal T}_{{\rm epi}\,\theta_1}(\overline{X},\theta_1(\overline{X}))$, i.e., the largest linear subspace contained in ${\cal T}_{{\rm epi}\,\theta_1}(\overline{X},\theta_1(\overline{X}))$, is given by
\[\begin{array}{lll}
 	{\rm lin}({\cal T}_{{\rm epi}\,\theta_1}\left(\overline{X},\theta_1(\overline{X}))\right) &:= &{\cal T}_{{\rm epi}\,\theta_1}(\overline{X},\theta_1(\overline{X})) \, \cap \, \left[-{\cal T}_{{\rm epi}\,\theta_1}(\overline{X},\theta_1(\overline{X}))\right] \nonumber\\[0.1in] 
&=& \left\{(H,y)\in \mathbb{S}^n\times\mathbb{R}\mid \theta_1'(\overline{X};H)\le h \le -\theta_1'(\overline{X};-H)\right\} \nonumber \\[0.1in] 
&=& \left\{(H,y)\in \mathbb{S}^n\times\mathbb{R}\mid \theta_1'(\overline{X};H)= h = -\theta_1'(\overline{X};-H)\right\},
\end{array}
\]
where the last equality follows from \cite[Theorem 23.1]{rockafellar1970}. We consider the following linear subspace
\begin{equation}\label{eq:def-Tlin}
	{\cal T}^{\rm lin}_{\theta_1}(\overline{X}):=\left\{H\in \mathbb{S}^n\mid \theta'_1(\overline{X};H)= -\theta'_1(\overline{X};-H)\right\}.
\end{equation}
 Let $\{\alpha^l\}_{l=1}^r$ be the index sets given by \eqref{eq:ak-symmetric} with respect to $\overline{X}$. Define the index set 
\[
\widetilde{\cal E}:=\{l\in\{1,\ldots,r\}\mid \mbox{$\exists\,i,j\in \alpha^l$ such that $({\bf a}^w)_i\neq({\bf a}^w)_j$ for some $w\in \iota_1(\lambda(\overline{X}))$}\},
\]
where $\iota_1(\lambda(\overline{X}))$ is the set defined in \eqref{index1} with respect to $\lambda(\overline{X})$.
The following proposition  characterizes ${\cal T}^{\rm lin}_{\theta_1}(\overline{X})$.
\begin{proposition}\label{prop:chara-Tlin}
Let $H\in \mathbb{S}^n$. Then $H\in {\cal T}^{\rm lin}_{\theta_1}(\overline{X})$ implies the existence of scalars $\{\widehat{\rho}_{l}\}_{l\in\widetilde{\cal E}}$ such that  $\overline{U}^\top _{\alpha^{l}}H\overline{U}_{\alpha^{l}}=\widehat{\rho}_{l} \, I_{|\alpha^{l}|}$ for any $\overline{U}\in{\mathbb O}^n(\overline{X})$. In fact,
\begin{equation*}
H\in {\cal T}^{\rm lin}_{\theta_1}(\overline{X}) \;\Longleftrightarrow \; \left[\,	\langle \lambda'(\overline{X};H), {\bf a}^{i}-{\bf a}^j\rangle = 0,\quad \forall\; i,j\in \iota_1(\lambda(\overline{X}))\,\right].
\end{equation*}
%In fact, for each given $H\in {\cal T}^{\rm lin}_{\theta_1}(\overline{X})$ and $l\in\widetilde{\cal E}$, there exists a scalar $\widehat{\rho}_{l}$ such that for any $\overline{U}\in{\mathbb O}^n(\overline{X})$, $\overline{U}^\top _{\alpha^{l}}H\overline{U}_{\alpha^{l}}=\widehat{\rho}_{l} \, I_{|\alpha^{l}|}$.
%Moreover, we have
%\[
%\langle \lambda'(\overline{X};H),\, {\bf a}^{i}-{\bf a}^j\rangle = \sum_{l\in \widetilde{\cal E}} \left\langle \, \widehat{\rho}_{l}\, {\bf e}_{|\alpha^{l}|}\, ,\, ({\bf a}^{i})_{\alpha^{l}}-({\bf a}^j)_{\alpha^{l}} \right\rangle = 0,\quad \forall\; i,j\in \iota_1(\lambda(\overline{X})).
%\]
%
\end{proposition}
\begin{proof}
Notice that for any $H\in\mathbb{S}^n$ and $U\in \mathbb{O}^n(\overline{X})$, 
\[
\left\{\begin{array}{ll}
\lambda'(\overline{X};H)=\left(\lambda(\overline{U}^\top_{\alpha^1}H\overline{U}_{\alpha^1}),\ldots,\lambda(\overline{U}^\top_{\alpha^r}H\overline{U}_{\alpha^r})\right)^\top, \\[0.1in]
\lambda'(\overline{X};-H)=\left(\lambda(-\overline{U}^\top_{\alpha^1}H\overline{U}_{\alpha^1}),\ldots,\lambda(-\overline{U}^\top_{\alpha^r}H\overline{U}_{\alpha^r})\right)^\top.
\end{array}\right.
\]
For each $1\le l\le r$, we have $\lambda(-\overline{U}^\top_{\alpha^l}H\overline{U}_{\alpha^l})=-\lambda^{\uparrow}(\overline{U}^\top_{\alpha^l}H\overline{U}_{\alpha^l})$.  Moreover,  we obtain from the symmetry of $\phi$  that for each permutation matrix $Q\in \mathbb{P}^n_{\lambda(\overline{X})}$, i.e., $Q\lambda(\overline{X})=\lambda(\overline{X})$, 
\[
\phi'(\lambda(\overline{X});h)=\phi'(\lambda(\overline{X});Qh),\quad \forall\, h\in\mathbb{R}^n.
\]
Therefore, 
$\theta'_1(\overline{X};H)=\phi'_1(\lambda(\overline{X});\lambda'(\overline{X};H)) = \displaystyle\sup_{{\bf z}\in\partial\phi_1(\lambda(\overline{X}))}\langle {\bf z},\lambda'(\overline{X};H)\rangle$  and
\begin{eqnarray*}
- \theta'_1(\overline{X};-H) = 	-\phi'_1(\lambda(\overline{X});\lambda'(\overline{X};-H))=-\phi'_1(\lambda(\overline{X});-\lambda'(\overline{X};H))\\[0.1in]
=-\sup_{{\bf z}\in\partial\phi_1(\lambda(\overline{X}))}\langle {\bf z},-\lambda'(\overline{X};H)\rangle  
	= \inf_{{\bf z}\in\partial\phi_1(\lambda(\overline{X}))}\langle {\bf z},\lambda'(\overline{X};H)\rangle.
\end{eqnarray*}
We thus derive from  \eqref{eq:def-Tlin} that $H\in {\cal T}^{\rm lin}_{\theta_1}(\overline{X})$ if and only if $\langle {\bf z},\lambda'(\overline{X};H)\rangle$ is invariant over ${\bf z}\in\partial\phi_1(\lambda(\overline{X}))$.
The claimed results then follow from \eqref{subgradient}, Lemma \ref{prop:eigenvalue_diff} and  Corollary \ref{remark:Ix}.
\qed	
\end{proof}

 \vskip 0.1in
\noindent
\underline{\bf The critical cone}.  Suppose that  $\overline{Y}\in\partial \theta_1(\overline{X})$. Then the critical cone of $\partial \theta_1(\overline{X})$ at $\overline{X} + \overline{Y}$ is 
defined as 
%\[
%{\cal C}(\overline{X}+\overline{Y};\partial \theta_1(\overline{X})):=\left\{H\in\mathbb{S}^n\mid \theta_1'(\overline{X};H)\le \langle \overline{Y},H\rangle \right\}.
%\]
%Since for any $H\in\mathbb{S}^n$, $\theta_1'(\overline{X};H)=\delta^*_{\partial\theta_1(\overline{X})}(H)$ and $\overline{Y}\in\partial \theta_1(\overline{X})$, the above definition can be equivalently written as
\begin{equation}\label{eq:def-critical-cone}
{\cal C}(\overline{X}+\overline{Y};\partial \theta_1(\overline{X})):=\left\{H\in\mathbb{S}^n\mid \theta_1'(\overline{X};H)= \langle \overline{Y},H\rangle \right\}.
\end{equation}
For each $l\in\{1,\ldots,r\}$, we further partition  the index set $\alpha^l$ into $\{\beta_k^l\}_{k=1}^{s_l}$ such that each $\beta_k^l$ contains one distinct eigenvalue of $\overline{Y}$, i.e.,
\begin{equation}\label{eq:def-beta-index}
	\left\{ \begin{array}{ll}
 	\lambda_i(\overline{Y})=\lambda_j(\overline{Y}) \epc & \mbox{if $i,j\in \beta_k^l$}, \\ [0.1in]
 	\lambda_i(\overline{Y})>\lambda_j(\overline{Y}) & \mbox{if $i\in \beta_k^l$, $j\in \beta_{k'}^{l}$ and $k, k'\in \{1,\ldots, s_l\}$ with $k<k'$}.
 \end{array}
\right.
\end{equation}
We also denote
	\begin{equation}\label{def-E-index}
	{\cal E}^l:=\{k\in\{1,\ldots,s_l\}\mid \mbox{$\exists\,i,j\in \beta^l_{k}$ such that $({\bf a}^w)_i\neq({\bf a}^w)_j$ for some $w\in \eta_1(\lambda(\overline{X}),\lambda(\overline{Y}))$}\},
	\end{equation} 
	where $\eta_1(\lambda(\overline{X}),\lambda(\overline{Y}))$ is the index set defined in \eqref{eq:def-index-eta-1} with respect to $\lambda(\overline{X})$ and $\lambda(\overline{Y})$. The characterization of the critical cone ${\cal C}(\overline{X} + \overline{Y};\partial\theta_1(\overline{X}))$ is provided in the following proposition.
		
	\begin{proposition}\label{prop:critical_cone-eq}
		Suppose that $(\overline{X},\overline{Y})\in{\rm gph}\,\partial\theta_1$ and $\overline{U}\in \mathbb{O}^n(\overline{X})\cap \mathbb{O}^n(\overline{Y})$. If $H\in {\cal C}(\overline{X} + \overline{Y};\partial\theta_1(\overline{X}))$, then the following three properties hold:
\vskip 0.1in
\noindent		
(i) for each $l\in\{1,\ldots,r\}$, $\overline{U}_{\alpha^l}^\top H\overline{U}_{\alpha^l}$ has the following block diagonal structure:
	\[
		\overline{U}_{\alpha^l}^\top H\overline{U}_{\alpha^l}={\rm Diag}\left((\overline{U}_{\alpha^l}^\top H\overline{U}_{\alpha^l})_{\beta_1^l\beta_1^l},\ldots, (\overline{U}_{\alpha^l}^\top H\overline{U}_{\alpha^l})_{\beta_{s_l}^l\beta_{s_l}^l}\right);
	\]
(ii)
$\langle\lambda'(\overline{X};H),{\bf a}^i\rangle=\langle\lambda'(\overline{X};H),{\bf a}^j\rangle=\displaystyle\max_{\kappa\in\iota_1(\lambda(\overline{X}))} \langle\lambda'(\overline{X};H),{\bf a}^\kappa\rangle,\quad \forall\,i,j\in\eta_1(\lambda(\overline{X}),\lambda(\overline{Y}))$;

\noindent
(iii) for each $l\in\{1,\ldots,r\}$ and $k\in {\cal E}^l$, there exists a scalar $\rho^l_{k}$ such that $\left(\overline{U}_{\alpha^l}^\top H\overline{U}_{\alpha^l}\right)_{\beta_{k}^l\beta_{k}^l}=\rho^l_{k} I_{|\beta_{k}^l|}$.% where ${\cal E}_l$ is the index set defined by \eqref{def-E-index}.

In fact, $H\in {\cal C}(\overline{X}+\overline{Y};\partial\theta_1(\overline{X}))$ if and only if for any $i,j\in\eta_1(\lambda(\overline{X}),\lambda(\overline{Y}))$,
		\begin{equation}\label{eq:Condition iii}
		\left\langle {\rm diag}(\overline{U}^\top  H\overline{U}), \,{\bf a}^i  \right\rangle\, =\, \left\langle  {\rm diag}(\overline{U}^\top  H\overline{U}), \,{\bf a}^j  \right\rangle \, =\, \max_{\kappa\in\iota_1(\lambda(\overline{X}))} \left\langle\lambda'(\overline{X};H),{\bf a}^\kappa\right\rangle,
		\end{equation}
		where the index sets $\eta_1$ and $\iota_1$ are defined in \eqref{eq:def-index-eta-1} and \eqref{index1}, respectively.
			\end{proposition}
	\begin{proof}
%	We first show that \eqref{eq:Condition iii} is necessary for a direction $H$ belongs to ${\cal C}(\overline{X}+\overline{Y};\partial\theta_1(\overline{X}))$.
		It follows from  Ky Fan's inequality in Lemma \ref{lem:Fan} that for any $H\in \mathbb{S}^n$,
\begin{eqnarray}
	\langle \overline{Y},H\rangle &=&\langle \Lambda(\overline{Y}), \overline{U}^\top H\overline{U}\rangle = \sum_{l=1}^r \langle \Lambda(\overline{Y})_{\alpha^l\alpha^l}\, , \,\overline{U}_{\alpha^l}^\top H\overline{U}_{\alpha^l} \rangle \nonumber \\[3pt] 
	&\le &  \sum_{l=1}^r\langle \lambda(\overline{Y})_{\alpha^l}\, , \,\lambda(\overline{U}^\top _{\alpha^l}H\overline{U}_{\alpha^l})\rangle =\langle \lambda(\overline{Y})\,,\,\lambda'(\overline{X};H)\rangle\label{eq:cir-c-ineq1}  \\ [3pt]
	&\le & \sup_{{\bf z}\in\partial\phi_1(\lambda(\overline{X}))}\left\langle {\bf z},\lambda'(\overline{X};H)\right\rangle=\theta_1'(\overline{X};H).\label{eq:cir-c-ineq2}
\end{eqnarray}
Therefore, in order for $H\in {\cal C}(\overline{X} + \overline{Y};\partial\theta_1(\overline{X}))$, the equalities in \eqref{eq:cir-c-ineq1} and \eqref{eq:cir-c-ineq2} must hold. 

Consider the inequality in \eqref{eq:cir-c-ineq1}. 
For each $l\in\{1,\ldots,r\}$, we know from Ky Fan's inequality that the equality in \eqref{eq:cir-c-ineq1} holds if and only if there exists $R^l\in\mathbb{O}^{|\alpha^l|}$ such that
\[
\Lambda(\overline{Y})_{\alpha^l\alpha^l}=R^l\Lambda(\overline{Y})_{\alpha^l\alpha^l}(R^l)^\top \quad {\rm and}\quad \overline{U}_{\alpha^l}^\top H\overline{U}_{\alpha^l} = R^l\Lambda(\overline{U}_{\alpha^l}^\top H\overline{U}_{\alpha^l})(R^l)^\top ,
\]
where $R^l$ has the following block diagonal structure:
\[
R^l={\rm Diag}\left(R^l_{1},\ldots,R^l_{s_l}\right)\quad {\rm with} \quad R^l_{k}\in\mathbb{O}^{|\beta_k^l|},\quad k=1,\ldots,s_l.
\]
Here $\{\beta_k^l\}_{k=1}^{s_l}$ is defined in \eqref{eq:def-beta-index}. The statement (i) thus follows.

In order for the equality  in \eqref{eq:cir-c-ineq2} holds, we must have $\langle \lambda(\overline{Y}),\lambda'(\overline{X};H)\rangle=\displaystyle{\sup_{{\bf z}\in\partial\phi_1(\lambda(\overline{X}))}}\langle {\bf z},\lambda'(\overline{X};H)\rangle$, which implies the statement (ii).

On the other hand, for each $l\in\{1,\ldots,r\}$, if $k\in {\cal E}^l$, then there exist $i,j\in \beta^l_{k}$ such that $({\bf a}^w)_i\neq({\bf a}^w)_j$ for some $w\in \eta_1(\lambda(\overline{X}),\lambda(\overline{Y}))$. Consider the $n\times n$ permutation matrix $Q^{i,j}$  satisfying
	\[
	(Q^{i,j}{\bf a}^w)_z=\left\{\begin{array}{ll}
	({\bf a}^w)_j & \mbox{if $z=i$,} \\ [3pt]
	({\bf a}^w)_i & \mbox{if $z=j$,} \\ [3pt]
	({\bf a}^w)_z & \mbox{otherwise,}
	\end{array} \right. \quad z=1,\ldots,n.
	\]  
	Since $\lambda_i(\overline{X})=\lambda_j(\overline{X})$ and $\lambda_i(\overline{Y})=\lambda_j(\overline{Y})$, it is clear that $Q^{i,j}\lambda(\overline{X})=\lambda(\overline{X})$ and $Q^{i,j}\lambda(\overline{Y})=\lambda(\overline{Y})$. It then follows from Corollary \ref{remark:Ix} that
	 there exists $w'\in \eta_1(\lambda(\overline{X}),\lambda(\overline{Y}))$ such that ${\bf a}^{w'}=Q^{i,j}{\bf a}^w$. Therefore, we derive from (ii) that
	\[
	\langle\lambda'(\overline{X};H),{\bf a}^w-{\bf a}^{w'}\rangle=(\lambda'_i(\overline{X};H)-\lambda'_j(\overline{X};H))(({\bf a}^w)_i-({\bf a}^{w})_j)=0,
	\] 
	which implies that 
	$$
	\lambda'_i(\overline{X};H)=\lambda'_j(\overline{X};H).
	$$ 
	For any $i'\in \beta^l_{k}$ with $i'\neq i$ and $i'\neq j$, if $({\bf a}^w)_{i'}\neq({\bf a}^w)_i$, by replacing $i$ by $i'$ and $j$ by $i$ in the above argument, we obtain that 
	$$
	\lambda'_{i'}(\overline{X};H)=\lambda'_{i}(\overline{X};H)=\lambda'_j(\overline{X};H);
	$$
	otherwise if $({\bf a}^w)_{i'}=({\bf a}^w)_i$,  then by replacing $i$ by $i'$ in the above argument, we can also obtain the above equality.
	Consequently, we know that for any $k\in {\cal E}^l$, there exists some $\rho^l_{k}\in\mathbb{R}$ such that for any $i\in \beta^l_{k}$,
	\[
	\lambda'_{i}(\overline{X};H)=\rho^l_{k},
	\]
	which, together with Lemma \ref{prop:eigenvalue_diff}, shows the property (ii).

To establish the last statement of this proposition, we observe that
for each $l\in\{1,\ldots,r\}$, if $k\notin {\cal E}^l$, then for any $w\in \eta_1(\lambda(\overline{X}),\lambda(\overline{Y}))$, there exists a scalar $\tilde{\rho}^l_{k}$ such that 
	\[
	({\bf a}^w)_i=({\bf a}^w)_j=\tilde{\rho}^l_{k}, \quad \forall\,i,j\in \beta^l_{k},
\]
which yields
	\begin{eqnarray*}
	&&\langle \lambda'(\overline{X},H),{\bf a}^w\rangle = \sum_{l=1}^r \Big( \sum_{k\in {\cal E}^l}\big\langle \lambda'(\overline{X},H)_{\beta_{k}^l},({\bf a}^w)_{\beta_{k}^l}\big\rangle +  \sum_{k\notin {\cal E}^l}\big\langle \lambda'(\overline{X},H)_{\beta_{k}^l},({\bf a}^w)_{\beta_{k}^l}\big\rangle \Big) \nonumber \\ [3pt]
	&=& \sum_{l=1}^r \Big( \sum_{k\in {\cal E}^l}\big\langle \rho^l_{k} {\bf e}_{|\beta_{k}^l|},({\bf a}^w)_{\beta_{k}^l}\big\rangle +  \sum_{k\notin {\cal E}^l}\big\langle \lambda((\overline{U}_{\alpha^l}^\top H\overline{U}_{\alpha^l})_{\beta_{k}^l\beta_{k}^l}), \, \tilde{\rho}^l_{k} {\bf e}_{|\beta_{k}^l|}\big\rangle \Big) \nonumber \\ [3pt]
	&=& \sum_{l=1}^r \Big( \sum_{k\in {\cal E}^l}\big\langle \rho^l_{k} {\bf e}_{|\beta_{k}^l|},({\bf a}^w)_{\beta_{k}^l}\big\rangle +  \sum_{k\notin {\cal E}^l} \tilde{\rho}^l_{k}\,{\rm trace}\big((\overline{U}_{\alpha^l}^\top H\overline{U}_{\alpha^l})_{\beta_{k}^l\beta_{k}^l}\big)\Big) \nonumber  \\ [3pt]
	&=& \sum_{l=1}^r \Big( \sum_{k\in {\cal E}^l}\big\langle {\rm diag}\big((\overline{U}_{\alpha^l}^\top H\overline{U}_{\alpha^l})_{\beta_{k}^l\beta_{k}^l}\big), \, ({\bf a}^w)_{\beta_{k}^l}\big\rangle +  \sum_{k\notin {\cal E}^l} \big\langle {\rm diag}\big((\overline{U}_{\alpha^l}^\top H\overline{U}_{\alpha^l})_{\beta_{k}^l\beta_{k}^l}\big),\,\tilde{\rho}^l_{k}\, {\bf e}_{|\beta_{k}^l|}\big\rangle\Big) \nonumber \\ [3pt]
	&=& \Big\langle {\rm diag}(\overline{U}^\top H\overline{U}), {\bf a}^w  \Big\rangle. \label{eq:dir-diag-1}
	\end{eqnarray*}
	Conversely, suppose that $H\in\mathbb{S}^n$ satisfies \eqref{eq:Condition iii}.  We have 
	\[
	\langle \overline{Y},H\rangle=\langle \Lambda(\overline{Y}), \overline{U}^\top H\overline{U}\rangle=  \langle \lambda(\overline{Y}), {\rm diag}(\overline{U}^\top H\overline{U})\rangle=\sum_{i\in \eta_1(\lambda(\overline{X}),\lambda(\overline{Y}))} u_i\langle {\bf a}^i,{\rm diag}(\overline{U}^\top H\overline{U})\rangle = \theta_1'(\overline{X};H),
	\] 
	which shows that $H\in  {\cal C}(\overline{X}+\overline{Y};\partial\theta_1(\overline{X}))$.
 The proof of this proposition is thus completed. \qed
	\end{proof}

Based on the above proposition, we can further characterize the 
 affine hull of ${\cal C}(\overline{X} + \overline{Y};\partial\theta_1(\overline{X}))$, which we denoted as ${\rm aff}\,({\cal C}(A;\partial\theta_1(\overline{X})))$. The proof  can be directly obtained from Proposition \ref{prop:critical_cone-eq}. For simplicity, we omit the details here.

\begin{proposition}\label{prop:critical_cone_aff}
	Suppose that $(\overline{X},\overline{Y})\in{\rm gph}\,\partial\theta_1$ and $\overline{U}\in \mathbb{O}^n(\overline{X})\cap \mathbb{O}^n(\overline{Y})$.  Then $H\in {\rm aff}\,({\cal C}(A;\partial\theta_1(\overline{X})))$ if and only if 
	it satisfies the the properties (i) and (iii) in Proposition \ref{prop:critical_cone-eq}, and
\begin{equation}\label{eq:Condition iii-aff}
\langle {\rm diag}(\overline{U}^\top H\overline{U}), {\bf a}^i  \rangle=\langle  {\rm diag}(\overline{U}^\top H\overline{U}), {\bf a}^j  \rangle, \quad \forall\; i,j\in\eta_1(\lambda(\overline{X}),\lambda(\overline{Y})).
\end{equation} 
%where the index set $\eta_1(\lambda(\overline{X}),\lambda(\overline{Y}))$ is defined in \eqref{eq:def-index-eta-1} with respect to $\lambda(\overline{X})$ and $\lambda(\overline{Y})$.

\end{proposition}

%\begin{remark}
%Suppose that $\overline{Y}\in\partial \theta_1(\overline{X})$ with the decompositions \eqref{eq:eig-deXS}. The set satisfying Conditions (i)-(iii) of Proposition \ref{prop:critical_cone_aff} is indeed a affine subspace in $\mathbb{S}^n$.
%\end{remark}

\vskip 10 true pt
\noindent
\underline{\bf The sigma term}.  Suppose that $\overline{Y}\in\,\partial\theta_1(\overline{X})$. Let $H\in{\cal C}(\overline{X}+\overline{Y},\partial\theta_1(\overline{X}))$ be arbitrarily given.  Since $\phi_1$ is Lipschitz continuous,  we know from \cite[Lemma 3.1]{BZowe82} that $\theta_1$ is (parabolic) second-order directionally differentiable with the second-order directional derivative
\begin{equation}\label{eq:def-sec-dir-diff}
\digamma_{\overline{X}, H}(W):=\theta_1''(\overline{X};H,W)=\phi_1''(\lambda(\overline{X});\lambda'(\overline{X};H),\lambda''(\overline{X};H,W)), \quad W\in\mathbb{S}^n.
\end{equation}
Moreover, it is easy to see that $\digamma_{\overline{X}, H}:\mathbb{S}^n\to \mathbb{R}$ is convex. We define the sigma term associated with the spectral function $\theta_1=\phi_1\circ\lambda$ at $\overline{Y}\in\partial \theta_1(\overline{X})$ as the conjugate function (cf. \cite{rockafellar1970} for the definition) of $\digamma_{\overline{X}, H}$ at $\overline{Y}$, that is, we consider the function
\[
	\digamma_{\overline{X}, H}^*(\overline{Y})=\sup_{W\in\mathbb{S}^n }\left\{\langle W,\overline{Y}\rangle-\digamma_{\overline{X}, H}(W)  \right\}.
\]
The proposition below characterizes the property of  $\digamma^*_{\overline{X}, H}$.

\begin{proposition}\label{prop:sigma-term-1}
	Suppose that $(\overline{X},\overline{Y})\in{\rm gph}\,\partial\theta_1$ and $\overline{U}\in \mathbb{O}^n(\overline{X})\cap \mathbb{O}^n(\overline{Y})$.   Denote $\bar{v}_1>\bar{v}_2>\cdots>\bar{v}_r$ as the distinct eigenvalues of $\overline{X}$. Let $H\in{\cal C}(\overline{X}+\overline{Y},\partial\theta_1(\overline{X}))$ be given. Then
		\begin{equation}\label{eq:sigma1}
		\digamma_{\overline{X}, H}^*(\overline{Y})=2\sum_{l=1}^r\langle \Lambda(\overline{Y})_{\alpha^l\alpha^l} \, , \, \overline{U}^\top _{\alpha^l}H(\overline{X}-\bar{v}_lI)^{\dagger}H\overline{U}_{\alpha^l}\rangle, 
\end{equation}
	where for each $l\in\{1,\ldots,r\}$, $(\overline{X}-\bar{v}_lI)^{\dagger}$ is the Moore-Penrose pseudo-inverse of $\overline{X}-\bar{v}_lI$.
\end{proposition}
\begin{proof} For any $W\in\mathbb{S}^n$, we have 
	\begin{eqnarray*}
		\langle W,\overline{Y}\rangle &=&\langle W, \overline{U}\Lambda(\overline{Y})\overline{U}^\top \rangle =\langle \overline{U}^\top W\overline{U}, \Lambda(\overline{Y})\rangle \\ [3pt]
&=& \sum_{l=1}^r\langle\Lambda(\overline{Y})_{\alpha^l\alpha^l}, \overline{U}^\top _{\alpha^l}W\overline{U}_{\alpha^l}\rangle =\sum_{l=1}^r \langle\Lambda(\overline{Y})_{\alpha^l\alpha^l}, \overline{U}^\top _{\alpha^l}(W-2H(\overline{X}-\bar{v}_lI)^{\dagger}H)\overline{U}_{\alpha^l}\rangle \\ [3pt]
&& + 2\sum_{l=1}^r\langle \Lambda(\overline{Y})_{\alpha^l\alpha^l}, \overline{U}^\top _{\alpha^l}H(\overline{X}-\bar{v}_lI)^{\dagger}H\overline{U}_{\alpha^l}\rangle.
\end{eqnarray*}
It follows from \cite[Example 2.68]{BShapiro00} that 
\[
\phi_1''(\lambda(\overline{X});\lambda'(\overline{X};H),\lambda''(\overline{X};H,W))=\max_{i\in\xi_1(\lambda(\overline{X}),\lambda'(\overline{X};H))} \langle \lambda''(\overline{X};H,W), {\bf a}^i\rangle,
\]
where  $\xi_1(\lambda(\overline{X}),\lambda'(\overline{X};H))\subseteq\iota_1(\lambda(\overline{X}))$ is defined by
\[
\xi_1(\lambda(\overline{X}),\lambda'(\overline{X};H)):=\left\{i\in \iota_1(\lambda(\overline{X}))\mid \langle \lambda'(\overline{X};H),{\bf a}^i\rangle =\displaystyle{\max_{j\in\iota_1(\lambda(\overline{X}))}} \langle \lambda'(\overline{X};H),{\bf a}^j\rangle \right\}.
\]
We then have
\begin{eqnarray*}
\digamma_{\overline{X}, H}^*(\overline{Y})&=&\sup_{W\in\mathbb{S}^n}\left\{ \langle W,\overline{Y}\rangle  - \phi_1''(\lambda(\overline{X});\lambda'(\overline{X};H),\lambda''(\overline{X};H,W))\right\} \\ [3pt]
&=& 2\sum_{l=1}^r\langle \Lambda(\overline{Y})_{\alpha^l\alpha^l}, \overline{U}^\top _{\alpha^l}H(\overline{X}-\bar{v}_lI)^{\dagger}H\overline{U}_{\alpha^l}\rangle\\[3pt]
&& \hskip -0.2in +\sup_{W\in\mathbb{S}^n}\Big\{\underbrace{\sum_{l=1}^r \langle\Lambda(\overline{Y})_{\alpha^l\alpha^l}, \overline{U}^\top _{\alpha^l}(W-2H(\overline{X}-\bar{v}_lI)^{\dagger}H)\overline{U}_{\alpha^l}\rangle-\max_{i\in\xi_1(\lambda(\overline{X}),\lambda'(\overline{X};H))} \langle \lambda''(\overline{X};H,W),{\bf a}^i\rangle }_{\mbox{denoted as $\Xi(W)$}}\Big\}.
\end{eqnarray*}
Therefore, in order to prove this proposition, it suffices to show that
 $\displaystyle{\sup_{W\in\mathbb{S}^n}}\Xi(W)=0$. 
In fact, since $H\in{\cal C}(\overline{X} + \overline{Y},\partial\theta_1(\overline{X}))$, we know from Proposition \ref{prop:critical_cone-eq}(i) that for each $l\in\{1,\ldots,r\}$, there exists $R^l\in\mathbb{O}^{|\alpha^l|}$ such that
\[
\Lambda(\overline{Y})_{\alpha^l\alpha^l}=R^l\Lambda(\overline{Y})_{\alpha^l\alpha^l}(R^l)^\top \quad {\rm and}\quad \overline{U}_{\alpha^l}^\top H\overline{U}_{\alpha^l} = R^l\Lambda(\overline{U}_{\alpha^l}^\top H\overline{U}_{\alpha^l})(R^l)^\top .
\]
Therefore, for any $W\in \mathbb{S}^n$ and  $l\in\{1,\ldots,r\}$,
\begin{eqnarray*}
	&&\left\langle \,\Lambda(\overline{Y})_{\alpha^l\alpha^l}\, , \, \overline{U}^\top _{\alpha^l}(W-2H(\overline{X}-\bar{v}_lI)^{\dagger}H)\overline{U}_{\alpha^l} \,\right\rangle \\ [3pt]
	&=&\left\langle \,\Lambda(\overline{Y})_{\alpha^l\alpha^l}\, , \, (R^l)^\top \overline{U}^\top _{\alpha^l}(W-2H(\overline{X}-\bar{v}_lI)^{\dagger}H)\overline{U}_{\alpha^l}R^l\,\right\rangle \\ [2pt]
	&=&  \sum_{\tilde{l}=1}^{\tilde{r}} \left\langle \,(\Lambda(\overline{Y})_{\alpha^l\alpha^l})_{\tilde{\alpha}^{\tilde{l}}\tilde{\alpha}^{\tilde{l}}}\, ,\,  (R^l_{\tilde{\alpha}^{\tilde{l}}})^{\top}\overline{U}_{\alpha^{l}}^{\top}\left[\, W-2H(\overline{X}-\bar{v}_lI)^{\dag}H\,\right]\overline{U}_{\alpha^{l}}R^l_{\tilde{\alpha}^{\tilde{l}}} \,\right\rangle \\[2pt]
	&\leq &  \sum_{\tilde{l}=1}^{\tilde{r}}  \, \sum_{i\in \tilde{\alpha}^{\tilde{l}}} \lambda_i(\overline{Y}) \, \lambda_{i}\left(R_{\tilde{\alpha}^{\tilde{l}}}^{\top}\overline{U}_{\alpha^{l}}^{\top}\left[ \, W-2H(\overline{X}-\bar{v}_lI)^{\dag}H \, \right]\overline{U}_{\alpha^{l}}R_{\tilde{\alpha}^{\tilde{l}}}\right)=   \sum_{i\in \alpha^l} \lambda_i(\overline{Y})\lambda''_i(\overline{X};H,W),
\end{eqnarray*}
where the first inequality follows from Ky Fan's inequality in Lemma \ref{lem:Fan} and the last equality is due to Lemma \ref{prop:second-directional-diff-eigenvalue}. 
%We further obtain that
%\begin{eqnarray*}
%	\Xi(W)& \, \le \, & \sum_{l=1}^r\sum_{i\in \alpha^l} \lambda_i(\overline{Y})\lambda''_i(\overline{X};H,W)-\max_{i\in\xi_1(\lambda(\overline{X}),\lambda'(\overline{X};H))} \langle \lambda''(\overline{X};H,W),{\bf a}^i\rangle \\ [3pt]
%	&=& \langle \lambda(\overline{Y}), \lambda''(\overline{X};H,W)\rangle -\max_{i\in\xi_1(\lambda(\overline{X}),\lambda'(\overline{X};H))} \langle \lambda''(\overline{X};H,W),{\bf a}^i\rangle.
%\end{eqnarray*}
Since $\lambda(\overline{Y})\in\partial\phi_1(\lambda(\overline{X}))$, we know from \eqref{subgradient} that there exists $\{u_i \in [0, 1]\}_{i\in\iota_1(\lambda(\overline{X}))}$ with   $\displaystyle\sum_{i\in\iota_1(\lambda(\overline{X}))}u_i=1$ such that 
$
\lambda(\overline{Y})=\displaystyle\sum_{i\in\iota_1(\lambda(\overline{X}))}u_i \, {\bf a}^i.
$
It then follows from Proposition \ref{prop:critical_cone-eq}~(ii) that if $H\in{\cal C}(\overline{X}  + \overline{Y},\partial\theta_1(\overline{X}))$, then 
\[
\eta_1(\lambda(\overline{X}),\lambda(\overline{Y}))\subseteq\xi_1(\lambda(\overline{X}),\lambda'(\overline{X};H)),
\]
where the index set $\eta_1(\lambda(\overline{X}),\lambda(\overline{Y}))$ is defined in \eqref{eq:def-index-eta-1}. We  then derive
\begin{eqnarray*}
\Xi(W)& \, \le \, & \sum_{l=1}^r\sum_{i\in \alpha^l} \lambda_i(\overline{Y})\lambda''_i(\overline{X};H,W)-\max_{i\in\xi_1(\lambda(\overline{X}),\lambda'(\overline{X};H))} \langle \lambda''(\overline{X};H,W),{\bf a}^i\rangle \\ [3pt]
%	&=& \langle \lambda(\overline{Y}), \lambda''(\overline{X};H,W)\rangle -\max_{i\in\xi_1(\lambda(\overline{X}),\lambda'(\overline{X};H))} \langle \lambda''(\overline{X};H,W),{\bf a}^i\rangle\\[3pt]
 & = & \left\langle \sum_{i\in\iota_1(\lambda(\overline{X}))}u_i \, {\bf a}^i, \lambda''(\overline{X};H,W) \right\rangle -\max_{i\in\xi_1(\lambda(\overline{X}),\lambda'(\overline{X};H))} \langle \lambda''(\overline{X};H,W),{\bf a}^i\rangle \\ [3pt]
	&\le& \sum_{i\in \eta_1(\lambda(\overline{X}),\lambda(\overline{Y}))}u_i\langle \lambda''(\overline{X};H,W), {\bf a}^i\rangle-\max_{i\in\xi_1(\lambda(\overline{X}),\lambda'(\overline{X};H))} \langle \lambda''(\overline{X};H,W),{\bf a}^i\rangle\\ [3pt]
	&\le & \sum_{i\in \eta_1(\lambda(\overline{X}),\lambda(\overline{Y}))}u_i\max_{i\in \eta_1(\lambda(\overline{X}),\lambda(\overline{Y}))} \langle \lambda''(\overline{X};H,W),{\bf a}^i\rangle  -\max_{i\in\xi_1(\lambda(\overline{X}),\lambda'(\overline{X};H))} \langle \lambda''(\overline{X};H,W),{\bf a}^i\rangle \\ [3pt]
	&=& \max_{i\in \eta_1(\lambda(\overline{X}),\lambda(\overline{Y}))}  \langle \lambda''(\overline{X};H,W),{\bf a}^i\rangle -\max_{i\in\xi_1(\lambda(\overline{X}),\lambda'(\overline{X};H))} \langle \lambda''(\overline{X};H,W),{\bf a}^i\rangle \le 0.
\end{eqnarray*}
On the other hand, it is easy to see that $\Xi(\widehat{W})=0$ if $$\overline{U}_{\alpha^{l}}^{\top}\,\widehat{W}\,\overline{U}_{\alpha^{l}}=2\,\overline{U}_{\alpha^{l}}^{\top}H(\overline{X}-\bar{v}_lI)^{\dag}H\overline{U}_{\alpha^{l}},\epc l=1,\ldots,r.$$ Therefore, we know that
$\displaystyle{\sup_{W\in\mathbb{S}^n}}\Xi(W)=0$. This completes the proof of the proposition.
\qed

\end{proof}

%\begin{definition}\label{def:Upsilon-function-MCP}
%		Suppose that $(\overline{X},\overline{Y})\in{\rm gph}\,\partial\theta_1$ and $\overline{U}\in \mathbb{O}^n(\overline{X})\cap \mathbb{O}^n(\overline{Y})$.   Define the function $\Upsilon_{ \overline{X}}^1:  \mathbb{S}^n\times\mathbb{S}^n\to\mathbb{R}$ by 
%	\begin{eqnarray*}
%		\Upsilon_{\overline{X}}^1\left(\overline{Y},H\right):=2\sum_{l=1}^r\langle \Lambda(\overline{Y})_{\alpha^l\alpha^l}, \overline{U}^\top _{\alpha^l}H(\overline{X}-\bar{v}^lI)^{\dagger}H\overline{U}_{\alpha^l}\rangle, \quad \overline{Y},H\in\mathbb{S}^n.
%	\end{eqnarray*} 
%\end{definition}

\begin{remark}\label{remark:Upsilon}
In fact, for any given $\overline{X}\in \mathbb{S}^n$  and any $\overline{Y}, H\in \mathbb{S}^n$ (not necessary in ${\cal C}(\overline{X}+\overline{Y},\partial\theta_1(\overline{X}))$), we can define the function $\Upsilon_{ \overline{X}}^1:  \mathbb{S}^n\times\mathbb{S}^n\to\mathbb{R}$ as the right side of \eqref{eq:sigma1}, i.e.,  
	\begin{equation}\label{eq:def-Upsilon1}
		\Upsilon_{\overline{X}}^1\left(\overline{Y},H\right):=2\sum_{l=1}^r\langle \Lambda(\overline{Y})_{\alpha^l\alpha^l}, \overline{U}^\top _{\alpha^l}H(\overline{X}-\bar{v}_lI)^{\dagger}H\overline{U}_{\alpha^l}\rangle,
	\end{equation}
where $\overline{U}\in \mathbb{O}^n(\overline{X})\cap \mathbb{O}^n(\overline{Y})$. Notice that if $\overline{Y}\in \partial \theta_1(\overline{X})$, it holds that
\begin{equation}\label{eq:Upsilon1-eq}
\Upsilon_{\overline{X}}^1\left(\overline{Y},H\right)=-2\sum_{1\le l<l'\le r}\sum_{i\in\alpha^l}\sum_{j\in\alpha^{l'}}\frac{\lambda_i(\overline{Y})-\lambda_j(\overline{Y})}{\lambda_i(\overline{X})-\lambda_j(\overline{X})}(\overline{U}_{\alpha^l}^\top H\overline{U}_{\alpha^{l'}})_{ij}^2.
\end{equation}
Since for any $i\in\alpha^l$ and $j\in\alpha^{l'}$ with $1\le l<l'\le r$, $\displaystyle\frac{\lambda_i(\overline{Y})-\lambda_j(\overline{Y})}{\lambda_i(\overline{X})-\lambda_j(\overline{X})}\ge 0$, we conclude that 
\[
\Upsilon_{\overline{X}}^1\left(\overline{Y},H\right)\le 0,\quad \mbox{$\forall\,H\in\mathbb{S}^n$}.
\]
\end{remark}

\subsection{Variational properties of $\theta_2$}\label{subsection:theta2}
In this subsection, we present analogue results  with respect to the function $\theta_2$. 
Recall the definition of the convex piecewise affine function $\psi$ in \eqref{eq:domphi}.
For notational simplicity, we denote $\zeta:=\psi\circ\lambda$ as the spectral function associated with $\psi$. Thus, the function $\theta_2$ can be viewed as the 
 the indictor function of the closed convex set $\mathcal{K}$ that is defined in the following way
\begin{equation}\label{eq:def-K-set}
{\cal K}:=\left\{X\in \mathbb{S}^n\mid \lambda(X)\in {\rm dom}\, \phi\right\}=\left\{X\in \mathbb{S}^n\mid \zeta(X)\le 0\right\}.
\end{equation}
Let $\overline{X}\in {\cal K}$ be given. Denote $ {\cal N}_{\cal K}(\overline{X})$ as  the normal cone of ${\cal K}$ at $\overline{X}\in\mathbb{S}^n$ in the sense of convex analysis \cite{rockafellar1970}. In the rest of the paper, we assume the following Slater condition for the closed convex set ${\cal K}$.
\begin{assumption}\label{ass:Slater-set}
There exists $\widetilde{X}\in {\cal K}$ such that $\zeta(\widetilde{X})<0$.
\end{assumption}
It is worth mentioning that the above assumption automatically holds for many interesting matrix optimization problems, such as the negative semidefinite programming (where $\mathcal{K}$ is the negative semidefinite matrix cone).
%By further employing \cite[Theorem 1.3]{Lewis1996a}, we know that $\overline{Y}\in\partial\theta_1(\overline{X})$ and  $\overline{Z}\in \partial\theta_2(\overline{X})\equiv{\cal N}_{\cal K}(\overline{X})$ if and only if 
%\begin{equation}\label{eq:lambdaY-lambdaZ}
%\lambda(\overline{Y})\in \partial\phi_1(\lambda(\overline{X})),\quad \lambda(\overline{Z})\in \partial\phi_2(\lambda(\overline{X}))\equiv{\cal N}_{{\rm dom}\,\phi}(\lambda(\overline{X}))
%\end{equation}
%and there exist orthogonal matrices $\overline{U}$ and $\overline{V}\in \mathbb{O}^n$ such that
%\begin{equation}\label{eq:eig-deXS}
%\overline{X}=\overline{U}\Lambda(\overline{X})\overline{U}^\top \quad {\rm and}\quad \overline{Y}=\overline{U}\Lambda(\overline{Y})\overline{U}^\top ,
%\end{equation}
%\begin{equation}\label{eq:eig-deXGamma}
%\overline{X}=\overline{V}\Lambda(\overline{X})\overline{V}^\top \quad {\rm and}\quad \overline{Z}=\overline{V}\Lambda(\overline{Z})\overline{V}^\top .
%\end{equation}

Recall the index sets $\{\alpha^l\}_{l=1}^r$ given by \eqref{eq:ak-symmetric} with respect to $\overline{X}$.  Three variational properties with respect to $\theta_2$  are in order.

\vskip 10 true pt 
\noindent
\underline{\bf The tangent cone and its lineality space}. Let $\overline{X}\in {\cal K}$ be such that $\zeta(\overline{X})=0$. Since $\zeta:\mathbb{S}^n\to \mathbb{R}$ is a closed convex function, it follows from \cite[Proposition 2.61]{BShapiro00} that the tangent cone ${\cal T}_{\cal K}(\overline{X})$ of the closed convex set ${\cal K}$ is given by
\begin{equation*} 
	{\cal T}_{\cal K}(\overline{X}):=\left\{H\in\mathbb{S}^n\mid \zeta'(\overline{X};H)\le 0 \right\}.
\end{equation*}
Let $\iota_2(\lambda(\overline{X}))$ be the index set defined by \eqref{index1} with respect to $\lambda(\overline{X})$, i.e., 
\[
\iota_2(\lambda(\overline{X})) = \{1\leq i\leq q\mid \langle {\bf b}^i,\lambda(\overline{X})\rangle - d_i= 0\}.
\]  
It then follows from \cite[Example 2.68]{BShapiro00} that for any $H\in \mathbb{S}^n$,
\begin{equation*}
\zeta'(\overline{X};H)=\psi'(\lambda(\overline{X});\lambda'(\overline{X};H))=\max_{i\in \iota_2(\lambda(\overline{X}))}\langle {\bf b}^i,\lambda'(\overline{X};H) \rangle.
\end{equation*}
Thus, the tangent cone ${\cal T}_{\cal K}(\overline{X})$ of the convex set ${\cal K}$ can be re-written as 
\begin{equation}\label{eq:Tangent2}
{\cal T}_{\cal K}(\overline{X})=\left\{H\in\mathbb{S}^n\mid \langle {\bf b}^i,\lambda'(\overline{X};H) \rangle\le 0,\  \forall\, i\in \iota_2(\lambda(\overline{X})) \right\}.
\end{equation}
Moreover, the corresponding lineality space ${\rm lin}({\cal T}_{\cal K}(\overline{X}))$ of ${\cal T}_{\cal K}(\overline{X})$ is given by
\begin{eqnarray}
{\rm lin}({\cal T}_{\cal K}(\overline{X})) &:=&{\cal T}_{\cal K}(\overline{X})\cap (-{\cal T}_{\cal K}(\overline{X}))=\left\{H\in \mathbb{S}^n\mid \zeta'(\overline{X};H)\le 0 \le -\zeta'(\overline{X};-H)\right\} \nonumber \\[0.1in] 
&=& \left\{H\in \mathbb{S}^n\mid \zeta'(\overline{X};H)=  -\zeta'(\overline{X};-H)=0\right\}, \label{eq:def-Tlin2}
\end{eqnarray}
where the last equality follows from \cite[Theorem 23.1]{rockafellar1970}. 
Define the index set 
\[
\widetilde{\cal F}:=\{l\in\{1,\ldots,r\}\mid \mbox{$\exists\,i,j\in \alpha^l$ such that $({\bf a}^w)_i\neq({\bf a}^w)_j$ for some $w\in \iota_2(\lambda(\overline{X}))$}\}.
\]
By employing similar arguments in the proof of Proposition \ref{prop:chara-Tlin}, we obtain the following characterization of ${\rm lin}({\cal T}_{\cal K}(\overline{X}))$ based on Corollary \ref{remark:Ix}.

\begin{proposition}\label{prop:chara-Tlin2}
Let $H\in \mathbb{S}^n$. Then $ H \in {\rm lin}({\cal T}_{\cal K}(\overline{X}))$ implies the existence of scalars $\{\widetilde{\rho}_{l}\}_{l\in \widetilde{F}}$ such that for  any $\overline{V}\in{\mathbb O}^{n}(\overline{X})$,
	\[
	\overline{V}^\top _{\alpha^l}\, H \, \overline{V}_{\alpha^l}=\widetilde{\rho}_{l} \, I_{|\alpha^l|}.
	\]
	In fact, 
 \[
 H \in {\rm lin}({\cal T}_{\cal K}(\overline{X})) \; \Longleftrightarrow\;
 \left[\, \langle \lambda'(\overline{X};H), {\bf b}^i\rangle=0,\quad \forall\; i\in \iota_2(\lambda(\overline{X}))\,\right].\]

\end{proposition}

\vskip 10 true pt
\noindent
\underline{\bf The critical cone}. 
Let $\overline{X}\in \mathcal{K}$ and $\overline{Z}\in {\cal N}_{\cal K}(\overline{X})$.
The critical cone of  $\mathcal{N}_{\cal K}(\overline{X})$ at $\overline{X} + \overline{Z}$ is defined by 
\begin{equation}\label{eq:critical-cone-def}
{\cal C}(\overline{X} + \overline{Z}; {\cal N}_{\cal K}(\overline{X})):= {\cal T}_{\cal K}(\overline{X})\cap\overline{Z}^{\perp}= \left\{H\in \mathbb{S}^n\mid \zeta'(\overline{X};H)\le 0,\  \langle \overline{Z}, H\rangle =0  \right\}.
\end{equation}
For each $l\in\{1,\ldots,r\}$, similar to the definition of the index sets $\beta_k^l$ in \eqref{eq:def-beta-index}, we use the notation $\{\gamma_k^l\}_{k=1}^{t_l}$ to further partition the set $\alpha^l$ based on the eigenvalue of $\overline{Z}$ as
\begin{equation}\label{eq:def-omega-index}
\left\{ \begin{array}{ll}
\lambda_i(\overline{Z})=\lambda_j(\overline{Z}) & \epc \mbox{if $i,j\in \gamma_k^l$ and $k\in\{1,\ldots, t_l\}$}, \\ [0.1in]
\lambda_i(\overline{Z})>\lambda_j(\overline{Z}) & \epc \mbox{if $i\in \gamma_k^l$, $j\in \gamma_{k'}^{l}$ and $k, k'\in \{1,\ldots, t_l\}$ with $k<k'$}.
\end{array}
\right.
\end{equation}
Recall the index set $\eta_2(\lambda(\overline{X}),\lambda(\overline{Z}))$ defined in \eqref{eq:def-index-eta-1} with respect to $\lambda(\overline{X})$ and $\lambda(\overline{Z})$.  For each $l\in\{1,\ldots,r\}$, define the index set
\begin{equation}\label{def-E-index2}
{\cal F}_l:=\{k\in\{1,\ldots,t_l\}\mid \mbox{$\exists\,i,j\in \gamma^l_{k}$ such that $({\bf b}^w)_i\neq({\bf b}^w)_j$ for some $w\in \eta_2(\lambda(\overline{X}),\lambda(\overline{Z}))$}\}.
\end{equation} 
The following  result on the characterization of ${\cal C}(\overline{X} + \overline{Z}; {\cal N}_{\cal K}(\overline{X}))$ can be obtained similarly as Proposition \ref{prop:critical_cone-eq} for $\theta_1$. For brevity, we omit the proof here.

\begin{proposition}\label{prop:critical_cone2-eq}
	Suppose that $(\overline{X},\overline{Z})\in {\rm gph}\,{\cal N}_{\cal K}$ and $\overline{V}\in \mathbb{O}^n(\overline{X})\cap \mathbb{O}^n(\overline{Z})$. If $H\in {\cal C}(\overline{X} + \overline{Z}; {\cal N}_{\cal K}(\overline{X}))$, then the following three conditions hold:
	\vskip 0.1in
\noindent
(i) for each $l\in\{1,\ldots,r\}$, $\overline{V}_{\alpha^l}^\top H\overline{V}_{\alpha^l}$ has the following block diagonal structure, i.e.,
		\[
		\overline{V}_{\alpha^l}^\top H\overline{V}_{\alpha^l}={\rm Diag}\left( (\overline{V}_{\alpha^l}^\top H\overline{V}_{\alpha^l})_{\gamma_1^l\gamma_1^l}, \cdots,(\overline{V}_{\alpha^l}^\top H\overline{V}_{\alpha^l})_{\gamma_{t_l}^l\gamma_{t_l}^l}\right); 
		\]
		(ii) 
$
\langle\lambda'(\overline{X};H),{\bf b}^i\rangle=\displaystyle\max_{j\in\iota_2(\lambda(\overline{X}))} \langle\lambda'(\overline{X};H),{\bf b}^j\rangle = 0,\quad \forall \; i\in\eta_2(\lambda(\overline{X}),\lambda(\overline{Y}));
$ 
\vskip 0.05in
\noindent
(iii) for each $l\in\{1,\ldots,r\}$ and $k\in {\cal F}_l$, there exists a scalar $\rho^l_{k}\in\mathbb{R}$ such that $(\overline{V}_{\alpha^l}^\top H\overline{V}_{\alpha^l})_{\gamma_{k}^l\gamma_{k}^l}=\rho^l_{k} \, I_{|\gamma_{k}^l|}$.
\vskip 0.1in
\noindent
	In fact, $H\in {\cal C}(\overline{X} + \overline{Z}; {\cal N}_{\cal K}(\overline{X}))$ if and only if for any $i \in\eta_2(\lambda(\overline{X}),\lambda(\overline{Z}))$,
	\[
	\langle {\rm diag}(\overline{V}^\top H\overline{V}), {\bf b}^i  \rangle=\max_{j\in\iota_2(\lambda(\overline{X}))} \langle\lambda'(\overline{X};H),{\bf b}^j\rangle = 0,
	\] 
	where the index set $\eta_2$ and $\iota_2$ are defined in \eqref{eq:def-index-eta-1} and \eqref{index1}. 
\end{proposition}

The  results below on the characterization of the affine hull  of the critical cone ${\cal C}(\overline{X} + \overline{Z}; {\cal N}_{\cal K}(\overline{X}))$ follows from Proposition \ref{prop:critical_cone2-eq}. 
\begin{proposition}\label{prop:critical_cone_aff2}
	Suppose that $(\overline{X},\overline{Z})\in {\rm gph}\,{\cal N}_{\cal K}$. Let $\overline{V}\in \mathbb{O}^n(\overline{X}) \cap \mathbb{O}^n(\overline{Z})$. Then $H\in {\rm aff}\,({\cal C}(\overline{X} + \overline{Z}; {\cal N}_{\cal K}(\overline{X})))$ if and only if it satisfies the properties (i) and (iii) in Proposition \ref{prop:critical_cone2-eq}, and  for any $i \in\eta_2(\lambda(\overline{X}),\lambda(\overline{Z}))$,
	\[
	\langle {\rm diag}(\overline{V}^\top H\overline{V}), {\bf b}^i  \rangle = 0,
	\] 
	where the index set $\eta_2(\lambda(\overline{X}),\lambda(\overline{Z}))\subseteq\iota_{2}(\lambda(\overline{X}))$ is defined in \eqref{eq:def-index-eta-1} with respect to $\lambda(\overline{X})$ and $\lambda(\overline{Z})$.
\end{proposition}

\vskip 10 true pt
\noindent
\underline{\bf The sigma term}. Suppose that $(\overline{X},\overline{Z})\in {\rm gph}\,{\cal N}_{\cal K}$. Let $H\in{\cal C}(\overline{X} + \overline{Z},{\cal N}_{\cal K}(\overline{X}))$ be arbitrarily given. Since $\zeta:\mathbb{S}^n\to\mathbb{R}$ is Lipschitz continuous,  we know from \cite[Lemma 3.1]{BZowe82} that $\zeta$ is (parabolic) second-order directionally differentiable and for any $W\in\mathbb{S}^n$,
\[
\zeta''(\overline{X};H,W)=\psi''(\lambda(\overline{X});\lambda'(\overline{X};H),\lambda''(\overline{X};H,W)).
\]
Since ${\cal K}$ is ${\cal C}^2$-cone reducible (see \cite[Definition 3.135]{BShapiro00} for the definition) and Assumption \ref{ass:Slater-set} holds,  the second-order tangent set of ${\cal K}$ at $\overline{X}$ along $H$ is given by
\[
{\cal T}^2_{\cal K}(\overline{X},H)=\{W\in\mathbb{S}^n\mid \zeta''(\overline{X};H,W)\le 0\}.
\]
As in the conventional conic programming, the sigma term associated with $\mathcal{K}$ is defined as the support function of its second-order tangent set,
whose explicit expression is given in the following proposition. The proof can be obtained in a similar fashion as that of Proposition \ref{prop:sigma-term-1}.
\begin{proposition}
	Suppose that $(\overline{X},\overline{Z})\in {\rm gph}\,{\cal N}_{\cal K}$ and $\overline{V}\in \mathbb{O}^n(\overline{X}) \cap \mathbb{O}^n(\overline{Z})$. Let $H\in{\cal C}(\overline{X} + \overline{Z};{\cal N}_{\cal K}(\overline{X}))$ be given. Then the support function of ${\cal T}^2_{\cal K}(\overline{X},H)$ at $\overline{Z}$ takes the following form
	\begin{equation}\label{eq:conj-sigma2}
		\delta^*_{{\cal T}^2_{\cal K}(\overline{X},H)}(\overline{Z})=2\sum_{l=1}^r\langle \Lambda(\overline{Z})_{\alpha^l\alpha^l}, \overline{V}^\top _{\alpha^l}H(\overline{X}-\bar{v}_lI)^{\dagger}H\overline{V}_{\alpha^l}\rangle.
	\end{equation}
	%where for each $l\in\{1,\ldots,r\}$, $(\overline{X}-\bar{v}_lI)^{\dagger}$ is the Moore-Penrose pseudo-inverse of $\overline{X}-\bar{v}_lI$.
\end{proposition}

\begin{remark}\label{remark:Upsilon2}
Similarly as that for $\theta_1$, 
for any given $\overline{X}\in\mathbb{S}^n$, define the function $\Upsilon_{ \overline{X}}^2: {\cal N}_{\cal K}(\overline{X})\times\mathbb{S}^n\to\mathbb{R}$ as the value of the right side of  \eqref{eq:conj-sigma2}, i.e.,
	\begin{equation}\label{eq:def-Upsilon2}
		\Upsilon_{\overline{X}}^2\left(\overline{Z},H\right):=2\sum_{l=1}^r\langle \Lambda(\overline{Z})_{\alpha^l\alpha^l}, \overline{V}^\top _{\alpha^l}H(\overline{X}-\bar{v}_lI)^{\dagger}H\overline{V}_{\alpha^l}\rangle, \quad \overline{Z}\in{\cal N}_{\cal K}(\overline{X}) \quad {\rm and} \quad H\in\mathbb{S}^n,
	\end{equation} 
	where $\overline{V}\in\mathbb{O}^n(\overline{X})$.
	If $\overline{Z}\in{\cal N}_{\cal K}(\overline{X})$, then for $\overline{V}\in\mathbb{O}^n(\overline{X})\cap \mathbb{O}^n(\overline{Z})$,
\begin{equation}\label{eq:Upsilon2-eq}
	\Upsilon_{\overline{X}}^2\left(\overline{Z},H\right)=-2\sum_{1\le l<l'\le r}\sum_{i\in\alpha^l}\sum_{j\in\alpha^{l'}}\frac{\lambda_i(\overline{Z})-\lambda_j(\overline{Z})}{\lambda_i(\overline{X})-\lambda_j(\overline{X})}(\overline{V}_{\alpha^l}^\top H\overline{V}_{\alpha^{l'}})_{ij}^2.
\end{equation}
	Moreover, since for any $i\in\alpha^l$ and $j\in\alpha^{l'}$ with $1\le l<l'\le r$, $\displaystyle\frac{\lambda_i(\overline{Z})-\lambda_j(\overline{Z})}{\lambda_i(\overline{X})-\lambda_j(\overline{X})}\ge 0$, we know that 
	\[
	\Upsilon_{\overline{X}}^2\left(\overline{Z},H\right)\le 0,\quad \mbox{$\forall\,H\in\mathbb{S}^n$}.
	\]
\end{remark}

%%%%%%%%%%%%%%%%%%%%%%%%%%%%%%%%%%%%%%%%%%%%%%%%%%%%%%%%%
%%%%%%%%%%%%%%%%%%%%%%%%%%%%%%%%%%%%%%%%%%%%%%%%%%%
\section{Characterization of the strong regularity}
\label{sec: main results}

This section is devoted to the characterization of the strong regularity of the solution to the KKT optimality condition for problem \eqref{opt}.
 Based on the decomposition of $\theta$ in \eqref{eq:thetasum}, we can rewrite problem  \eqref{opt} as follows:
\[
\begin{array}{cl}
\displaystyle\operatornamewithlimits{minimize}_{{\bf x}\in \mathbb{X}} &\; f({\bf x})  + \theta_1(g({\bf x}))\\[0.1in] 
\mbox{subject to} &\; h({\bf x})=0, \\ [0.1in]
 &\; g({\bf x})\in {\cal K},
\end{array}
\]
where  the closed convex set ${\cal K}$ is given by \eqref{eq:def-K-set}. In fact, all the subsequent analysis does not require the function $g$ in the objective and constraint to be the same. In order to make the discussions more general, we allow two different continuously differentiable functions $g_1$ and $g_2$ in this problem, i.e., we consider the problem
\begin{equation}\label{opt-eq}
\begin{array}{cl}
\displaystyle\operatornamewithlimits{minimize}_{{\bf x}\in \mathbb{X}} &\; f({\bf x})  + \theta_1(g_1({\bf x}))\\[0.1in] 
\mbox{subject to} &\; h({\bf x})=0, \\ [0.1in]
 &\; g_2({\bf x})\in {\cal K}.
\end{array}
\end{equation}
The Lagrangian function ${\cal L}:\mathbb{X}\times\mathbb{S}^n\times\mathbb{Y}\times\mathbb{S}^n\to \mathbb{R}$ of the above problem can be written as
\[
{\cal L}({\bf x}, {\bf y}, Y, Z):=f({\bf x})+ \langle Y,g_1({\bf x})\rangle+\langle {\bf y},h({\bf x})\rangle  +\langle Z,g_2({\bf x})\rangle, \quad ({\bf x}, {\bf y}, Y, Z)\in \mathbb{X}\times\mathbb{S}^n\times\mathbb{Y}\times\mathbb{S}^n,
\]
yielding the following KKT optimality condition of \eqref{opt-eq}:
\begin{equation}\label{eq:KKT-opt}
	\left\{ 
	\begin{array}{l}
		{\cal L}'_{\bf x}({\bf x}, {\bf y}, Y, Z)=0,\quad h({\bf x})=0, \\ [3pt]
		Y\in \partial\theta_1(g_1({\bf x})),  \quad Z\in {\cal N}_{\cal K}(g_2({\bf x})),
	\end{array}
\right.
\end{equation}
where 
${\cal L}'_{\bf x}({\bf x}, {\bf y}, Y, Z)$ is the partial derivative of ${\cal L}$  with respect to ${\bf x}$. For any $({\bf x},{\bf y},  Y, Z)\in\mathbb{X}\times\mathbb{Y}\times\mathbb{S}^n\times\mathbb{S}^n$ satisfying \eqref{eq:KKT-opt}, we call  ${\bf x}$ a stationary point, $({\bf y}, Y,Z)$ the corresponding multiplier and $({\bf x},{\bf y},Y,Z)$ a KKT point of \eqref{opt-eq}, respectively. We also  use ${\cal M}({\bf x})$ to denote the set of multipliers $({\bf y},Y,Z)$ for any stationary point ${\bf x}$ such that $({\bf x},{\bf y},Y,Z)$ is a KKT point.

The following concept of constraint nondegeneracy for the nonsmooth matrix optimization problem \eqref{opt-eq} is adopted from Robinson \cite{Robinson84}, which reduces to the linear independence constraint qualification for the conventional nonlinear programming problem.

\begin{definition}
The constraint nondegeneracy of problem \eqref{opt-eq} is defined as 
\begin{equation}\label{eq:nondegen}
\left[\begin{array}{c}
h'(\overline{\bf x}) \\ [3pt]
g_1'(\overline{\bf x}) \\ [3pt]
g_2'(\overline{\bf x})
\end{array} \right]\mathbb{X}+\left[\begin{array}{c}
\{0\} \\ [3pt]
{\cal T}^{\rm lin}_{\theta_1}(g_1(\overline{\bf x})) \\ [3pt]
{\rm lin}\left({\cal T}_{\cal K}(g_2(\overline{\bf x}))\right)
\end{array} \right]=\left[\begin{array}{c}
\mathbb{Y}  \\ [3pt]
\mathbb{S}^n \\ [3pt]
\mathbb{S}^n
\end{array} \right],
\end{equation}  
where the affine spaces ${\cal T}^{\rm lin}_{\theta_1}(g_1(\overline{\bf x}))$ and ${\rm lin}\left({\cal T}_{\cal K}(g_2(\overline{\bf x}))\right)$ are defined in \eqref{eq:def-Tlin} and \eqref{eq:def-Tlin2}. 
\end{definition}

Let $\overline{\bf x}\in \mathbb{X}$ be a stationary point of problem \eqref{opt-eq} and $(\overline{\bf y},\overline{Y},\overline{Z})\in {\cal M}(\overline{\bf x})$. Since  ${\cal M}(\overline{\bf x})$ is nonempty, the critical cone of \eqref{opt-eq} can be defined as
\begin{equation}\label{eq:def-critical-MOP}
{\cal C}(\overline{\bf x}):=\left\{{\bf d}\in \mathbb{X} \;\;\bigg|\;\;
\begin{array}{ll}
h'(\overline{\bf x}){\bf d}=0,\;\; g_1'(\overline{\bf x}){\bf d}\in {\cal C}(g_1(\overline{\bf x})+\overline{Y};\partial \theta_1(g_1(\overline{\bf x}))),\\[0.1in]
  g_2'(\overline{\bf x}){\bf d}\in {\cal C}(g_2(\overline{\bf x})+\overline{Z};{\cal N}_{\cal K}(g_2(\overline{\bf x})))
\end{array}
 \right\}.
\end{equation}
where ${\cal C}(g_1(\overline{\bf x})+\overline{Y};\partial \theta_1(g_1(\overline{\bf x})))$ and ${\cal C}(g_2(\overline{\bf x})+\overline{Z};{\cal N}_{\cal K}(g_2(\overline{\bf x})))$ are the critical cones defined in \eqref{eq:def-critical-cone} and \eqref{eq:critical-cone-def}, respectively.

For notationally simplicity,  we define the  outer approximation set to ${\cal C}(\overline{\bf x})$ with respect to $(\overline{\bf y},\overline{Y},\overline{Z})\in {\cal M}(\overline{\bf x})$ as
\begin{equation}\label{eq:def-app}
 {\rm app}(\overline{\bf y},\overline{Y},\overline{Z}):=\left\{
 {\bf d}\in \mathbb{X} \;\;\bigg| \;\;\begin{array}{ll}
 h'(\overline{\bf x}){\bf d}=0,\ g_1'(\overline{\bf x}){\bf d}\in {\rm aff}\left({\cal C}(g_1(\overline{\bf x})+\overline{Y};\partial \theta_1(g(\overline{\bf x})))\right),\\[0.1in]
   g_2'(\overline{\bf x}){\bf d}\in {\rm aff}\left({\cal C}(g_2(\overline{\bf x})+\overline{Z};{\cal N}_{\cal K}(g_2(\overline{\bf x})))\,\right) 
 \end{array}
 \right\}.
\end{equation}
 The following definition of the strong second-order sufficient condition of problem \eqref{opt-eq}  generalizes the concept from the conventional nonlinear programming introduced by Robinson \cite{Robinson80} to the nonsmooth matrix optimization. 
\begin{definition}\label{def:ssoc}
	Let $\overline{\bf x}\in\mathbb{X}$ be a stationary point of the problem \eqref{opt-eq}. We say the strong second-order sufficient condition holds at $\overline{\bf x}$ if 
	\begin{equation}\label{eq:ssoc}
	\begin{array}{ll}
	\displaystyle\sup_{(\overline{\bf y},\overline{Y},\overline{Z})\in{\cal M}(\overline{\bf x})}\left\{\langle {\bf d}, {\cal L}''_{{\bf x}{\bf x}}(\overline{\bf x},\overline{\bf y},\overline{Y},\overline{Z}){\bf d}\rangle -\Upsilon_{g_1(\overline{\bf x})}^1\left(\overline{Y},g_1'(\overline{\bf x}){\bf d}\right) - \Upsilon_{g_2(\overline{\bf x})}^2\left(\overline{Z},g_2'(\overline{\bf x}){\bf d}\right)    \right\} >0, \\[0.15in]
	  \qquad  \qquad  \qquad  \qquad  \qquad \qquad  \qquad  \qquad 	 \qquad	 \qquad\forall\, {\bf d}\in 
	\displaystyle\bigcap_{(\overline{\bf y},\overline{Y},\overline{Z})\in{\cal M}(\overline{\bf x})} {\rm app}(\overline{\bf y},\overline{Y},\overline{Z})\setminus\{0\}.
	\end{array}
	\end{equation}
%	where for any $(\overline{\bf y},\overline{Y},\overline{Z})\in{\cal M}(\overline{\bf x})$, the set ${\rm app}(\overline{\bf y},\overline{Y},\overline{Z})$ is given by \eqref{eq:def-app}. 
\end{definition}

Let $\overline{\bf x}$ be a local optimal solution  to \eqref{opt-eq} with ${\cal M}(\overline{\bf x})\neq \emptyset$. Then there exists  $(\overline{\bf y}, \overline{Y},\overline{Z})\in \mathbb{Y} \times \mathbb{S}^n\times\mathbb{S}^n$ such that the KKT condition \eqref{eq:KKT-opt} holds, i.e., $(\overline{\bf x},\overline{\bf y},\overline{Y},\overline{Z})$ is a solution of the following generalized equation:
\begin{equation}\label{eq:GE-KKT}
0\in\left[ 
\begin{array}{c}
{\cal L}'_{\bf x}({\bf x}, {\bf y}, Y, Z) \\ [3pt]
h({\bf x})  \\ [3pt]
-g_1({\bf x}) \\ [3pt]
-g_2({\bf x})
\end{array}
\right] + \left[
\begin{array}{c}
\{0\} \\ [3pt]
\{0\} \\ [3pt]
\partial\theta_1^*(Y) \\ [3pt]
\partial\delta^*_{\cal K}(Z)
\end{array}\right],
\end{equation} 
where $\delta^*_{\cal K}$ is the support function of the nonempty closed convex set ${\cal K}$.
The following concept of strong regularity for a solution of the generalized equation \eqref{eq:GE-KKT} is adapted from Robinson \cite{Robinson80}.

\begin{definition}\label{def:strong-regularity-MCP}
Let $\mathbb{T}\equiv\mathbb{X}\times\mathbb{Y} \times\mathbb{S}^n\times\mathbb{S}^n$. We say that $(\overline{\bf x},\overline{\bf y},\overline{Y},\overline{Z})\in\mathbb{T}$ is a strongly regular solution of the generalized equation \eqref{eq:GE-KKT} if there exist neighborhoods ${\cal U}$ of the origin $0$ and ${\cal V}$ of $(\overline{\bf x},\overline{\bf y},\overline{Y},\overline{Z})$ such that for every $\delta\in{\cal U}$, the following generalized equation
\begin{equation}\label{eq:perturb-ge-equation-MCP}
\delta\in\left[ 
\begin{array}{c}
{\cal L}'_{\bf x}({\bf x}, {\bf y}, Y, Z) \\ [3pt]
h({\bf x})  \\ [3pt]
-g_1({\bf x}) \\ [3pt]
-g_2({\bf x})
\end{array}
\right] + \left[
\begin{array}{c}
\{0\} \\ [3pt]
\{0\} \\ [3pt]
\partial\theta_1^*(Y) \\ [3pt]
\partial\delta^*_{\cal K}(Z)
\end{array}\right]
\end{equation}
has a unique solution in ${\cal V}$, denoted by ${S}_{\cal V}({\delta})$, and the mapping ${S}_{\cal V}:{\cal U}\to{\cal V}$ is Lipschitz continuous.
\end{definition}

In fact, the solution of the generalized equation  \eqref{eq:GE-KKT} can be viewed as the solution of the following nonsmooth equation
\begin{equation}\label{eq:def-nonsmooth-equation-MOP}
F({\bf x}, {\bf y}, Y, Z):=\left[\begin{array}{c}
{\cal L}'_{\bf x}({\bf x}, {\bf y}, Y, Z)\\ [3pt]
h({\bf x})\\ [3pt]
g_1({\bf x})-{\rm Pr}_{\theta_1}(g_1({\bf x})+Y) \\ [3pt]
g_2({\bf x})-\Pi_{\cal K}(g_2({\bf x})+Z)
\end{array}\right]=0,
\end{equation}
 where ${\rm Pr}_{\theta_1}:\mathbb{S}^n\to \mathbb{S}^n$ is the proximal mapping of $\theta_1$ and $\Pi_{\cal K}:\mathbb{S}^n\to \mathbb{S}^n$ is the metric projection onto ${\cal K}$. The function $F$ is said to be a locally Lipschitz homeomorphism near $S$ if there exists an open neighborhood $\mathcal{U}$ such that the restricted mapping $S|_\mathcal{U}:\mathcal{U}\to S(\mathcal{U})$ is Lipschitz continuous and bijective, and its inverse is also Lipschitz continuous. The following result on the relationship between the strong regularity of \eqref{eq:GE-KKT} and the locally Lipschitz homeomorphism of $F$  in \eqref{eq:def-nonsmooth-equation-MOP} can be obtained directly from their definitions. 

\begin{lemma}\label{lem:equivalence-strong-regularity-Lip-Homeomorphism}
 Suppose that   $F(\overline{\bf x},\overline{\bf y},\overline{Y},\overline{Z})=0$. Then $F$ is a locally Lipschitz homeomorphism near $(\overline{\bf x},\overline{\bf y},\overline{Y},\overline{Z})$ if and only if $(\overline{\bf x},\overline{\bf y},\overline{Y},\overline{Z})$ is a strongly regular solution of the generalized equation \eqref{eq:GE-KKT}.
\end{lemma}

Let $\overline{S}=(\overline{\bf x},\overline{\bf y},\overline{Y},\overline{Z})\in\mathbb{T}$ be such that $F(\overline{\bf x},\overline{\bf y},\overline{Y},\overline{Z})=0$. By \cite[Lemma 1]{CSun08}, we know that ${\bm W}\in\partial F(\overline{\bf x},\overline{\bf y},\overline{Y},\overline{Z})$ (respectively, ${\bm W}\in\partial_{B}F(\overline{\bf x},\overline{\bf y},\overline{Y},\overline{Z})$) if and only if there exist ${\cal S}^1\in\partial{\rm Pr}_{\theta_1}(g_1(\overline{\bf x})+\overline{Y})$ (respectively, ${\cal S}^1\in\partial_{B}{\rm Pr}_{\theta_1}(g_1(\overline{\bf x})+\overline{Y})$) and ${\cal S}^2\in\partial\Pi_{\cal K}(g_2(\overline{\bf x})+\overline{Z})$ (respectively, ${\cal S}^2\in\partial_{B}\Pi_{\cal K}(g_2(\overline{\bf x})+\overline{Z})$) such that for any $(\triangle {\bf x} ,\triangle {\bf y}, \triangle Y,\triangle Z)\in\mathbb{T}$,
\begin{equation}\label{eq:def-B-Jocbian-F}
{\bm W}\left(\triangle {\bf x} ,\triangle {\bf y}, \triangle Y,\triangle Z\right)=\left[\begin{array}{c}
{\cal L}^{''}_{{\bf x}{\bf x}}(\overline{\bf x},\overline{\bf y},\overline{Y},\overline{Z})\triangle {\bf x}+h'(\overline{\bf x})^{*}\triangle {\bf y}+g_1'(\overline{\bf x})^*\triangle Y+g_2'(\overline{\bf x})^*\triangle Z\\ [0.1in]
h'(\overline{\bf x})\triangle {\bf x} \\ [0.1in]
g_1'(\overline{\bf x})\triangle {\bf x}-{\cal S}^1(g_1'(\overline{\bf x})\triangle {\bf x}+\triangle Y) \\ [0.1in]
g_2'(\overline{\bf x})\triangle {\bf x}-{\cal S}^2(g_2'(\overline{\bf x})\triangle {\bf x}+\triangle Z)
\end{array}\right].
\end{equation} 

Next, we shall provide the explicit formula of the generalized Jacobian $\partial F(\overline{\bf x},\overline{\bf y},\overline{Y},\overline{Z})$. Let $\overline{Y}\in \partial\theta_1(g_1(\overline{\bf x}))$.
We first consider the characterization of ${\cal S}\in\partial{\rm Pr}_{\theta_1}(g_1(\overline{\bf x})  + \overline{Y})$.  
%Recall the Moreau-Yosida regularization of the spectral function $\theta_1=\phi_1\circ\lambda$, i.e.,
%\begin{equation}\label{eq:MY-theta}
%\chi_{\theta_1}(A):=\min_{Z\in\mathbb{S}^n}\left\{\theta_1(Z)+\frac{1}{2}\|Z-A\|^2 \right\}.
%\end{equation}
Recall that the proximal mapping ${\rm Pr}_{\theta_1}:\mathbb{S}^n\to\mathbb{S}^n$ is the spectral operator with respect to the proximal mapping ${\rm Pr}_{\phi_1}:\mathbb{R}^n\to\mathbb{R}^n$, i.e.,
\[
g_1(\overline{\bf x}) = {\rm Pr}_{\theta_1}(g_1(\overline{\bf x})  + \overline{Y})=\overline{U}\,{\rm Diag}({\rm Pr}_{\phi_1}(\lambda(g_1(\overline{\bf x})  + \overline{Y})))\,\overline{U}^\top ,
\]
where $\overline{U}\in\mathbb{O}^n(g_1(\overline{\bf x}))\cap \mathbb{O}^n(\overline{Y})$. To proceed, we denote $
\mathbb{W}^1:=\prod_{l=1}^{r}\mathbb{S}^{|\beta_{1}^l|}\times\ldots\times\mathbb{S}^{|\beta_{s_l}^l|}$ and
the spectral operator with respect to the directional derivative ${\rm Pr}'_{\phi_1}(\lambda(g_1(\overline{\bf x}));\,\bullet\,)$ as
\[\Sigma:=(\Sigma^1_1,\ldots,\Sigma^1_{s^1},\ldots,\Sigma^r_1,\ldots,\Sigma^r_{s^r}):\mathbb{W}^1\to\mathbb{W}^1.
\]
It follows from Proposition \ref{prop:MY-phi-direction-diff} that $\Sigma$ is actually the metric projection operator over the following nonempty closed convex set 
\begin{equation}\label{eq:def-Delta}
\Delta^1:=\left\{W\in\mathbb{W}^1\mid \langle\mu(W),{\bf a}^i\rangle=\langle\mu(W),{\bf a}^j\rangle =\max_{\kappa\in\iota_1(\lambda(g_1(\overline{\bf x})))}\langle\mu(W),{\bf a}^\kappa\rangle,\ \forall\,i,j\in \eta_1(\lambda(g_1(\overline{\bf x})),\lambda(\overline{Y}))\right\},
\end{equation}
where the index set $\eta_1$ is defined in \eqref{eq:def-index-eta-1},  and for any $W=(W^1_1,\ldots,W^1_{s^1},\ldots,W^{r}_1,\ldots,W^{r}_{s^r})\in\mathbb{W}^1$, 
\[
\mu(W):=(\lambda(W^1_1),\ldots,\lambda(W^1_{s^1}),\ldots,\lambda(W^{r}_1),\ldots,\lambda(W^{r}_{s^r})).
\]
Let
\begin{equation}\label{eq:def-DH}
D^1(H):=\left( \overline{U}^\top _{\beta^1_1}H\overline{U}_{\beta^1_1}\,,\,\ldots,\,\overline{U}^\top _{\beta^1_{s^1}}H\overline{U}_{\beta^1_{s^1}},\,\ldots,\,\overline{U}^\top _{\beta^r_1}H\overline{U}_{\beta^r_1}\,,\ldots,\,\overline{U}^\top _{\beta^r_{s^r}}H\overline{U}_{\beta^r_{s^r}}\right)\in \mathbb{W}^1
\end{equation}
and a matrix $\mathcal{A} = (\mathcal{A}_{ij})_{n\times n}$ whose $(i,j)$-th entry is
\begin{equation}\label{eq:A-def}
{\cal A}_{ij}:=\left\{\begin{array}{ll}
\left[\lambda_i(g_1(\overline{\bf x}))-\lambda_j(g_1(\overline{\bf x}))\right]/\left[\lambda_i(g_1(\overline{\bf x}) + \overline{Y})-\lambda_j(g_1(\overline{\bf x}) + \overline{Y})\right] & \; \mbox{if $\lambda_i(g_1(\overline{\bf x}) + \overline{Y})\neq\lambda_j(g_1(\overline{\bf x}) + \overline{Y})$,} \\ [3pt]
0 & \; \mbox{otherwise.}
\end{array}
\right.
\end{equation}
 Since ${\rm Pr}_{\phi_1}$ is globally Lipchitz continuous and directionally differentiable at $\lambda(g_1(\overline{\bf x}) + \overline{Y})$, we know from \cite[Remark 1 and Theorem 6]{DSSToh14} that ${\rm Pr}_{\theta_1}$ is directionally differentiable at $g_1(\overline{\bf x}) + \overline{Y}$ and the directional derivative ${\rm Pr}_{\theta_1}'(g_1(\overline{\bf x}) + \overline{Y};H)$ at $g_1(\overline{\bf x}) + \overline{Y}$ along $H\in\mathbb{S}^n$ is given by
\[
{\rm Pr}_{\theta_1}'(g_1(\overline{\bf x}) + \overline{Y};H)=\overline{U}\,{\rm Pr}_{\phi_1}^{[1]}(g_1(\overline{\bf x}) + \overline{Y};H)\overline{U}^\top ,
\]
where ${\rm Pr}_{\phi_1}^{[1]}(g_1(\overline{\bf x}) + \overline{Y};H)$ is the first divided directional difference of ${\rm Pr}_{\phi_1}$ at $g_1(\overline{\bf x}) + \overline{Y}$ along $H$ with the expression
\[
{\rm Pr}_{\phi_1}^{[1]}(g_1(\overline{\bf x}) + \overline{Y};H):={\cal A}\circ\left(\overline{U}^\top H\overline{U}\right)+{\rm Diag}\left(\Sigma^1_1(D^1(H)),\ldots,\Sigma^r_{s^r}(D^1(H)) \right)\in\mathbb{S}^n.
\]
%Furthermore, we know from \cite[Theorem 7]{DSSToh14} that ${\rm Pr}_{\theta}$ is F-differentiable at $X=\overline{X}+\overline{S}$ with $\overline{S}\in\partial\theta(\overline{X})$ if and only if ${\cal I}(\lambda(\overline{X},\overline{S}))={\cal I}(\lambda(\overline{X}))$ and the derivative ${\rm Pr}'_{\theta}(X)$ is given by
%\[
%{\rm Pr}'_{\theta}(X)H=\overline{P}[{\cal A}\circ\overline{P}^\top H\overline{P}+{\rm Diag}\left(\Phi_1(D(H)),\ldots,\Phi_r(D(H)) \right)]\overline{P}^\top ,
%\]
%where $\Phi:=(\Phi_1,\ldots,\Phi_r):{\cal W}\to{\cal W}$ is the spectral operator with respect to the derivative ${\rm Pr}'_{\phi}(\lambda(\overline{X}))$, where ${\cal W}:=\mathbb{S}^{|\alpha_1|}\times\ldots\times\mathbb{S}^{|\alpha_r|}$, and 
%\[
%D(H)=\left( \overline{P}^\top _{\alpha_1}H\overline{P}_{\alpha_1},\ldots,\overline{P}^\top _{\alpha_r}H\overline{P}_{\alpha_r}\right)\in {\cal W}.
%\]
Finally, since ${\rm Pr}_{\phi_1}$ is piecewise affine, we know from \cite[Section 7.3]{BCShapiro98} that 
\[
{\rm Pr}_{\phi_1}({\bf x}+{\bf h})-{\rm Pr}_{\phi_1}({\bf x})={\rm Pr}_{\phi_1}'({\bf x};{\bf h}),\quad \forall\,{\bf x},{\bf h}\in\mathbb{R}^n.
\]
Thus, it follows from \cite[Theorem 7.8]{DSSToh18} that 
\[
\partial{\rm Pr}_{\theta_1}(g_1(\overline{\bf x}) + \overline{Y})=\partial\Psi(0),
\]
where $\Psi:={\rm Pr}'_{\theta_1}(g_1(\overline{\bf x}) + \overline{Y};\,\bullet\,):\mathbb{S}^n\to\mathbb{S}^n$ is the directional derivative of ${\rm Pr}_{\theta_1}$ at $g_1(\overline{\bf x}) + \overline{Y}$. Based on the above discussions, we obtain the following result on the characterization of $\partial{\rm Pr}_{\theta_1}(g_1(\overline{\bf x}) + \overline{Y})$.

\begin{lemma}\label{prop:subdiff-theta}
Let $\overline{Y}\in \partial\theta_1(g_1(\overline{\bf x}))$ and $\overline{U}\in \mathbb{O}^n(g_1(\overline{\bf x}))\cap \mathbb{O}^n(\overline{Y})$.
	It holds that ${\cal S}\in \partial{\rm Pr}_{\theta_1}(g_1(\overline{\bf x}) + \overline{Y})$ if and only if there exists ${\cal U}:=({\cal U}^1_1,\ldots,{\cal U}^1_{s^1},\ldots,{\cal U}^r_1,\ldots,{\cal U}^r_{s^r})\in \partial\Pi_{\Delta^1}(0)$ such that for any $H\in\mathbb{S}^n$,
	\[
	{\cal S}(H)=\overline{U}\left[{\cal A}\circ \left(\overline{U}^\top H\overline{U}\right)\right]\overline{U}^\top+\overline{U}\, {\rm Diag}\left({\cal U}^1_1(D^1(H)),\ldots,{\cal U}^r_{s^r}(D^1(H))\right)\overline{U}^\top ,
	\]
	where $\Delta^1$ is the nonempty convex set defined by \eqref{eq:def-Delta} and $\Pi_{\Delta^1}$ denotes the metric projection onto $\Delta^1$, $D^1(H)\in \mathbb{W}^1$ and $\mathcal{A} = (\mathcal{A}_{ij})_{n\times n}$ are defined in \eqref{eq:def-DH} and \eqref{eq:A-def}, respectively.
\end{lemma}

Similarly, for $\overline{Z}\in {\cal N}_{\cal K}(g_2(\overline{\bf x}))$ and $\overline{V}\in \mathbb{O}^n(g_2(\overline{\bf x}))\cap \mathbb{O}^n(\overline{Z})$, we denote $\mathbb{W}^2:=\prod_{l=1}^{r}\mathbb{S}^{|\gamma_{1}^l|}\times\ldots\times\mathbb{S}^{|\gamma_{t_l}^l|}$ and a matrix $\mathcal{B} = (\mathcal{B}_{ij})_{n\times n}$ whose $(i,j)$-th entry is given by
\[
\left\{\begin{array}{ll}
{\cal B}_{ij}:=\left\{\begin{array}{ll}
(\lambda_i(g_2(\overline{\bf x}))-\lambda_j(g_2(\overline{\bf x})))/(\lambda_i(g_2(\overline{\bf x}) + \overline{Z})-\lambda_j(g_2(\overline{\bf x}) + \overline{Z})) & \; \mbox{if $\lambda_i(g_2(\overline{\bf x}) + \overline{Z})\neq\lambda_j(g_2(\overline{\bf x}) + \overline{Z})$,} \\ [3pt]
0 & \; \mbox{otherwise}.
\end{array}
\right.\\[0.2in]
D^2(H):=\left( \overline{V}^\top _{\gamma^1_1}H\overline{V}_{\gamma^1_1},\ldots,\overline{V}^\top _{\gamma^1_{t^1}}H\overline{V}_{\gamma^1_{t^1}},\ldots,\overline{V}^\top _{\gamma^r_1}H\overline{V}_{\gamma^r_1},\ldots,\overline{V}^\top _{\gamma^r_{t^r}}H\overline{V}_{\gamma^r_{t^r}}\right)\in \mathbb{W}^2,\\[0.15in]
\Delta^2:=\left\{W\in\mathbb{W}^2\mid 0=\langle\nu(W),{\bf b}^i\rangle=\displaystyle\max_{j\in\iota_2(\lambda(g_2(\overline{\bf x})))}\langle\nu(W),{\bf b}^j\rangle,\ \forall\,i\in \eta_2\left(\lambda(g_2(\overline{\bf x})),\lambda(\overline{Z})\right)\right\},
\end{array}\right.
\]
where the index set $\eta_2(\lambda(g_2(\overline{\bf x})),\lambda(\overline{Y}))$ is defined in \eqref{eq:def-index-eta-1} with respect to $\lambda(g_2(\overline{\bf x}))$ and $\lambda(\overline{Z})$,  and for any $W=(W^1_1,\ldots,W^1_{t^1},\ldots,W^{r}_1,\ldots,W^{r}_{t^r})\in\mathbb{W}^2$, 
\[
\nu(W):=(\lambda(W^1_1),\ldots,\lambda(W^1_{t^1}),\ldots,\lambda(W^{r}_1),\ldots,\lambda(W^{r}_{t^r})).
\]
We have the following characterization of $\partial\Pi_{\cal K}(g_2(\overline{\bf{x}}) + \overline{Z})$. 
\begin{lemma}\label{prop:subdiff-theta2}
Let $\overline{Z}\in {\cal N}_{\cal K}(g_2(\overline{\bf x}))$ and $\overline{V}\in \mathbb{O}^n(g_2(\overline{\bf x}))\cap \mathbb{O}^n(\overline{Z})$.
It holds that	${\cal S}\in\partial\Pi_{\cal K}(g_2(\overline{\bf{x}}) + \overline{Z})$ if and only if there exists ${\cal V}:=({\cal V}^1_1,\ldots,{\cal V}^1_{t^1},\ldots,{\cal V}^r_1,\ldots,{\cal V}^r_{t^r})\in \partial\Pi_{\Delta^2}(0)$ such that for any $H\in\mathbb{S}^n$,
\[
	{\cal S}(H)=\overline{V}\left[\,{\cal B}\circ \left(\,\overline{V}^\top H\overline{V}\,\right)\,\right]\overline{V}^\top +\overline{V}\,{\rm Diag}\left({\cal V}^1_1(D^2(H)),\ldots,{\cal V}^r_{t^r}(D^2(H))\right)\overline{V}^\top.
\]

\end{lemma}

By comparing  the characterizations of Clarke's generalized Jacobian of the proximal mapping  ${\rm Pr}_{\theta_1}$ in Lemma \ref{prop:subdiff-theta} with  ${\rm aff}\left({\cal C}(g_1(\overline{\bf x})+\overline{Y};\partial\theta_1(g_1(\overline{\bf x})))\right)$ in Proposition \ref{prop:critical_cone_aff}, we derive the following lemma.

\begin{lemma}\label{lemma:V-in-affC}
	Suppose that $\overline{Y}\in\partial\theta_1(g_1(\overline{\bf x}))$ with $\overline{U}\in \mathbb{O}^n(g_1(\overline{\bf x}))\cap \mathbb{O}^n(\overline{Y})$ and ${\cal S}\in\partial{\rm Pr}_{\theta_1}(g_1(\overline{\bf x})+\overline{Y})$.  Then
\begin{equation}\label{SH in affine hull}
	{\cal S}(H)\in {\rm aff}\left({\cal C}(g_1(\overline{\bf x})+\overline{Y};\partial\theta_1(g_1(\overline{\bf x})))\right),\quad \forall\, H\in\mathbb{S}^n.
\end{equation}
In addition, if  ${\cal S}(H)=0$ for some $H\in\mathbb{S}^n$, then the following two conditions hold:
	\begin{itemize}
		\item[(i)] the matrix $\overline{U}^\top H\overline{U}\in \mathbb{S}^n$ has the following block diagonal structure:
	\[
	\overline{U}^\top H\overline{U}= {\rm Diag}\left(\overline{U}^\top _{\alpha^1}H\overline{U}_{\alpha^1},\cdots, \overline{U}^\top _{\alpha^r}H\overline{U}_{\alpha^r}\right);
\]
	\item[(ii)] for $l=1,\ldots,r$, if $k\in {\cal E}^l$, then there exists $\{\kappa_{ij}\in \mathbb{R}\}_{i,j\in \iota_1(\lambda(g_1(\overline{\bf{x}})))}$ such that  
	\[
	{\rm tr}\left(\overline{U}^\top _{\beta^l_k}H\overline{U}_{\beta^l_k}\right)=\sum_{i,j\in \iota_1(\lambda(g_1(\overline{\bf{x}})))} \kappa_{ij}\left\langle \, {\bf e}_{|\beta^l_k|}\, ,  \, ({\bf a}^i-{\bf a}^j)_{\beta^l_k}\right\rangle;
	\] 
	otherwise if $k\notin {\cal E}^l$, then $\overline{U}^\top _{\beta^l_k}H\overline{U}_{\beta^l_k}=0$, where  the index set ${\cal E}^l$ is defined by \eqref{def-E-index}.
	\end{itemize}

\end{lemma}
\begin{proof} 
To prove the inclusion \eqref{SH in affine hull} in this lemma, it suffices to check the three conditions in Proposition \ref{prop:critical_cone_aff} hold.
	For any given ${\cal S}\in\partial{\rm Pr}_{\theta_1}(g_1(\overline{\bf x})+\overline{Y})$ and $H\in\mathbb{S}^n$,  we obtain from Lemma \ref{prop:subdiff-theta} that  for each $l\in\{1,\ldots,r\}$, there exists ${\cal U}=({\cal U}^1_1,\ldots,{\cal U}^r_{s^r})\in \partial\Pi_{\Delta^1}(0)$ such that
	\[
	\overline{U}_{\alpha^l}^\top {\cal S}(H)\overline{U}_{\alpha^l}={\cal A}_{\alpha^l\alpha^l}\circ \left(\overline{U}_{\alpha^l}^\top H\overline{U}_{\alpha^l}\right)+{\rm Diag}\left(\,{\cal U}^l_1(D^1(H)),\ldots,{\cal U}^l_{s_l}(D^1(H))\right).
	\]
%where $D^1(H)\in \mathbb{W}^1$ is defined by \eqref{eq:def-DH}.
	For each $l\in\{1,\ldots,r\}$,	since $\lambda_i(g_1(\overline{\bf x}))=\lambda_j(g_1(\overline{\bf x}))$ for any $i,j\in \alpha^l$, it follows from \eqref{eq:A-def} that ${\cal A}_{\alpha^l\alpha^l}=0$, which implies that the condition (i) in Proposition \ref{prop:critical_cone_aff} holds. 
	
%	By Lemma \ref{prop:subdiff-theta}, we know that there exists ${\cal U}=({\cal U}^1_1,\ldots,{\cal U}^r_{s^r})\in \partial\Pi_{\Delta^1}(0)$, where $\Delta^1\subseteq \mathbb{W}^1$ is the closed convex subset given by \eqref{eq:def-Delta}, and $\Pi_{\Delta^1}:\mathbb{W}^1\to\mathbb{W}^1$ is the metric projection operator over $\Delta^1$. 
	Let ${\cal D}_{\Pi_{\Delta^1}}\subseteq\mathbb{W}^1$ be the set of all points at which $\Pi_{\Delta^1}$ is differentiable. We  define 
	%the subset $\daleth\subseteq\mathbb{W}^1$ by 
	\[
	\daleth:=\left\{W\in\mathbb{W}^1\mid \mbox{for each $l\in\{1,\ldots,r\}$ and $k\in\{1,\ldots,s_l\}$, the eigenvalues of $W^l_{k}$ are distinct}\right\}.
	\]
	Since $\mathbb{W}^1\setminus\daleth$ has measure zero (in the sense of Lebesgue), we know from \cite[Theorem 4]{Warga81} that 
	\begin{equation}\label{eq:conv hull}
	\partial\Pi_{\Delta^1}(0)={\rm conv} \left\{\, \lim_{W\to 0}\Pi'_{\Delta^1}(W)\mid W\in {\cal D}_{\Pi_{\Delta^1}}\cap \daleth\right\}.
	\end{equation}	
	Then for any $\Theta\in \left\{ \displaystyle\lim_{W\to 0}\Pi'_{\Delta^1}(W)\mid W\in {\cal D}_{\Pi_{\Delta^1}}\cap \daleth\right\}$, there exits a sequence 
	$$
	\left\{W^{q}:=\left((W^1_{1})^{q},\ldots,(W^{s^1}_{1})^{q},\ldots,(W^1_{r})^{q},\ldots,(W^{s^r}_{r})^{q}\right)\right\}\subseteq {\cal D}_{\Pi_{\Delta^1}}\cap \daleth
	$$
	 converging to $0\in \mathbb{W}$ such that for any $H\in\mathbb{S}^n$,
	\[
	\Theta(D^1(H))=\lim_{q\to \infty}\Pi'_{\Delta^1}(W^q)(D^1(H)).
	\]
	Consider any fixed $l=1,\ldots,r$ and $k=1,\ldots,s_l$.
	For each $q$, assume that $(W^l_{k})^{q}$ admits the eigenvalue decomposition  
	$$
	(W^l_{k})^{q}=(R^l_{k})^{q}\Lambda((W^l_{k})^{q})((R^l_{k})^{q})^\top  \quad  {\rm with} \quad (R^l_{k})^{q}\in\mathbb{O}^{|\beta_k^l|}.
	$$ 
	Notice that $\Pi_{\Delta^1}:\mathbb{W}^1\to \mathbb{W}^1$ is the spectral operator with respect to the symmetric mapping  $\pi:={\rm Pr}'_{\phi_1}(\lambda(g_1(\overline{\bf x}));\,\bullet\,)$ defined by \eqref{eq:dir-diff-phi1}. Thus, we know from \cite[Theorem 7]{DSSToh14} that $\Pi_{\Delta^1}$ is differentiable at $W^q$ if and only if $\pi$ is differentiable at 
	$$
	\mu^q:=\mu(W^q)=(\lambda((W^1_1)^q),\ldots,\lambda((W^r_{s^r})^q)).
	$$ 
	For each $l\in\{1,\ldots,r\}$ and $k\in\{1,\ldots,s_l\}$, denote
	\begin{equation}\label{defn:Ahomega}
	\left\{\begin{array}{ll}
	({\cal A}^l_{k}(\mu^q))_{ij}=\left\{
	\begin{array}{ll}
	\displaystyle{\frac{(\pi^l_k(\mu^q))_i-(\pi^l_k(\mu^q))_j}{\lambda_i((W^l_{k})^{q})-\lambda_j((W^l_{k})^{q})}} & \epc\mbox{if $i\neq j$} \\ [0.15in]
	0 & \epc\mbox{otherwise} 
	\end{array}
	\right. \quad \forall\,i,j\in\left\{\, 1,\ldots,|\beta^l_k| \,\right\},\\[0.25in]
	
	h^q=\left( {\rm diag}\big(\,((R^1_{1})^{q})^\top  \widetilde{H}_{\beta^1_1\beta^1_1})(R^1_{1})^{q} \,\big),\ldots,{\rm diag}\big( ((R^r_{s^r})^{q})^\top \widetilde{H}_{\beta^r_{s^r}\beta^r_{s^r}}(R^r_{s^r})^{q} \big)\right),\\[0.15in]
	\widetilde{H}=\overline{U}^\top H\overline{U}, \\[0.1in]
	(\Omega^l_{k})^q=(R^l_{k})^{q} \left[{\cal A}^l_{k}(\mu^q)\circ \left[\,((R^l_{k})^{q})^\top \widetilde{H}_{\beta^l_k\beta^l_k}(R^l_{k})^{q}\,\right]+{\rm Diag}\left((\pi'(\mu^q)h^q)_{\beta^l_k}\right)\right]((R^l_{k})^{q})^\top .
	\end{array}\right.
	\end{equation}
It has been shown in \cite[Theorem 7]{DSSToh14} that for each $q$, the derivative of $\Pi_{\Delta^1}$ at $W^q$ is given by 
	\[
	\Pi_{\Delta^1}'(W^q)(D^1(H))=\left((\Omega^1_1)^q,\ldots,(\Omega^1_{s^1})^q,\ldots,(\Omega^{r}_1)^q,\ldots,(\Omega^r_{s^r})^q\right)\in\mathbb{W}^1.
	\] 
	For each $l\in\{1,\ldots,r\}$, recall the index set ${\cal E}^l$ defined in \eqref{def-E-index}. We know from  Proposition \ref{prop:MY-phi-direction-diff} that for each $l\in\{1,\ldots,r\}$, $k\in {\cal E}^l$ and $q$, 
	\begin{equation}\label{eq:Al-beta2-case1}
	{\cal A}^l_k(\mu^q)=0.
	\end{equation} 
	Denote $X^q:=g_1(\overline{\bf x})+W^q$. Based on \cite[Theorem 2.1]{msarabi2016b},  for any $q$ sufficient large, it holds that \[\eta_1(\lambda(g_1(\overline{\bf x})),\lambda(\overline{Y}))\subseteq \iota_1(\lambda(X^q)).\]  Again using Proposition \ref{prop:MY-phi-direction-diff}, we obtain that for each $l\in\{1,\ldots,r\}$, $k\in {\cal E}^l$ and $q$, there exists $\rho^l_{k^l}\in\mathbb{R}$ such that
	\[
	(\pi'(\mu^q)h^q)_{\beta^l_{k}}=\rho^l_{k} {\bf e}_{|\beta_{k}^l|},
	\] 
	and for any $i,j\in\eta_1(\lambda(g_1(\overline{\bf x})),\lambda(\overline{Y}))$,
	\[\begin{array}{ll}
	\langle\, {\rm diag}(\Omega^q), {\bf a}^i-{\bf a}^j  \,\rangle=\displaystyle\sum_{l=1}^r\sum_{k\in {\cal E}^l}  \left\langle {\rm diag}((\Omega^l_k)^q),  \, ({\bf a}^i)_{\beta_{k}^l}-({\bf a}^j)_{\beta_{k}^l} \right\rangle \\
	=\displaystyle\sum_{l=1}^r\sum_{k\in {\cal E}^l}  \left\langle \rho^l_{k} {\bf e}_{|\beta_{k}^l|}, \, ({\bf a}^i)_{\beta_{k}^l}-({\bf a}^j)_{\beta_{k}^l} \right\rangle=0,
	\end{array}
	\] 
	where ${\rm diag}(\Omega^q)=\left({\rm diag}((\Omega^1_1)^q),\ldots,{\rm diag}((\Omega^1_{s^1})^q),\ldots,{\rm diag}((\Omega^{r}_1)^q),\ldots,{\rm diag}((\Omega^r_{s^r})^q)\right)$.
	Thus, we know that ${\cal S}(H)$ satisfies condition (iii) and \eqref{eq:Condition iii-aff} of Proposition \ref{prop:critical_cone_aff}. The above arguments show that \eqref{SH in affine hull} holds.
	
	In the following, we assume ${\cal S}(H) = 0$ for some $H\in \mathbb{S}^n$ and prove the rest of the lemma. Since ${\cal S}(H) = 0$, it is clear that $\overline{U}^\top {\cal S}(H)\overline{U}=0$. It then follows from Lemma \ref{prop:subdiff-theta} that for any $1\le l<l'\le r$,
	\[
	0=\overline{U}^\top _{\alpha^l}{\cal S}(H)\overline{U}_{\alpha^{l'}}=\left(\overline{U}^\top _{\alpha^{l'}}{\cal S}(H)\overline{U}_{\alpha^l}\right)^\top = {\cal A}_{\alpha^l\alpha^{l'}}\circ\left(\overline{U}_{\alpha^l}^\top H \overline{U}_{\alpha^{l'}}\right).
	\]
	By the definition of ${\cal A}$ in \eqref{eq:A-def}, we have ${\cal A}_{\alpha^l\alpha^{l'}}=\left({\cal A}_{\alpha^{l'}\alpha^l} \right)^\top \neq 0$ for all $1\le l<l'\le r$, which implies that 
	\[
	\overline{U}_{\alpha^l}^\top H\overline{U}_{\alpha^{l'}}=\left( \overline{U}_{\alpha^l}^\top H \overline{U}_{\alpha^{l'}}\right)^\top =0, \quad 1\le l<l'\le r.
	\]
	This proves the statement (i).
	
We further obtain from Lemma \ref{prop:subdiff-theta} that
	\begin{equation}\label{eq:diag=0}
	0=\overline{U}^\top {\cal S}(H)\overline{U}={\rm Diag}\left({\cal U}^1_1(D^1(H)),\ldots,{\cal U}^r_{s^r}(D^1(H))\right).
	\end{equation}
%	where  ${\cal U}=({\cal U}^1_1,\ldots,{\cal U}^r_{s^r})\in \partial\Pi_{\Delta^1}(0)$, $\Delta^1\subseteq \mathbb{W}^1$ is the closed convex subset given by \eqref{eq:def-Delta}, $D^1(H)\in \mathbb{W}^1$ is defined by \eqref{eq:def-DH} and $\Pi_{\Delta^1}:\mathbb{W}^1\to\mathbb{W}^1$ is the metric projection operator over $\Delta^1$.
%	 Let ${\cal D}_{\Pi_{\Delta^1}}\subseteq\mathbb{W}^1$ be the set points where $\Pi_{\Delta^1}$ is differentiable. We further define the subset $\daleth\subseteq\mathbb{W}^1$ by 
%	\[
%	\daleth:=\left\{W\in\mathbb{W}^1\mid \mbox{for each $l\in\{1,\ldots,r\}$ and $k\in\{1,\ldots,\tau_l\}$, the eigenvalues of $W^l_{k}$ are distinct}\right\}.
%	\]
%	Since $\mathbb{W}^1\setminus\daleth$ has zero  measure (in the sense of Lebesgue), we know from \cite[Theorem 4]{Warga81} that 
%	\[
%	\partial\Pi_{\Delta^1}(0)={\rm conv} \left\{ \lim_{W\to 0}\Pi'_{\Delta^1}(W)\mid W\in {\cal D}_{\Pi_{\Delta^1}}\cap \daleth\right\}.
%	\]
By Carath\'{e}odory's theorem, we obtain from \eqref{eq:conv hull} that there exist 
$$
{\cal U}^{\{i\}}\in\left\{ \lim_{W\to 0}\Pi'_{\Delta^1}(W)\mid W\in {\cal D}_{\Pi_{\Delta^1}}\cap \daleth\right\}, \quad i=1,\ldots,w
$$ 
for some positive integer $w$
such that  ${\cal U}=({\cal U}^1_1,\ldots,{\cal U}^1_{s^1},\ldots,{\cal U}^r_1,\ldots,{\cal U}^r_{s^r})\in \partial\Pi_{\Delta^1}(0)$ can be written as 
$$
{\cal U}=\displaystyle\sum_{i=1}^w\varsigma_i\,{\cal U}^{\{i\}}\epc \mbox{for some $\varsigma_i\ge 0$, $i=1,\ldots,w$, and $\displaystyle\sum_{i=1}^w\varsigma_i=1$.}
$$
%where $\varsigma^i\ge 0$, $i=1,\ldots,w$, and $\displaystyle\sum_{i=1}^w\varsigma^i=1$.
 For each ${\cal U}^{\{i\}}$,  there exits a sequence $\left\{W^{q}\right\}\subseteq {\cal D}_{\Pi_{\Delta^1}}\cap \daleth$ converging to $0 \in \mathbb{W}$  such that
	\[
	{\cal U}^{\{i\}}(D^1(H))=\lim_{q\to \infty}\Pi'_{\Delta^1}(W^q)(D^1(H)).
	\]
	Following the same notation in \eqref{defn:Ahomega} with respect to the newly defined  sequence  $\left\{W^{q}\right\}$, we
	derive from  Proposition \ref{prop:MY-phi-direction-diff} that for each $l\in\{1,\ldots,r\}$ and $k\in {\cal E}^l$ and all $q$, 
	\begin{equation*} 
	({\cal A}^l_{k}(\mu^q))_{ij}=\left\{ \begin{array}{ll}
	0 & \epc \mbox{if $k\in {\cal E}^l$} \\ [3pt]
	1 & \epc \mbox{if $k\notin {\cal E}^l$}
	\end{array}\right. \quad \forall\,i,j\in\{1,\ldots,|\beta^l_k|\} \ {\rm with}\ i\neq j.
	\end{equation*} 
	Since $\eta_1(\lambda(g_1(\overline{\bf x})),\lambda(\overline{Y}))\subseteq \iota_1(\lambda(X^q))$, we know that  if $k\in {\cal E}^l$ for some  $l\in\{1,\ldots,r\}$, then  there exists a scalar $(\rho^l_{k})^q$ such that 
	\[
	(\pi'(\mu^q)h^q)_{\beta^l_{k}}=\left\{\begin{array}{ll}
	(\rho^l_{k})^q \, {\bf e}_{|\beta_{k}^l|}  &\epc  \mbox{if $k\in {\cal E}^l$,} \\ [0.1in]
	(h^q)_{\beta^l_{k}}  & \epc \mbox{if $k\notin {\cal E}^l$},
	\end{array} \right. 
	\] 
	yielding 
	\[
	(\Omega^l_{k})^q=\left\{\begin{array}{ll}
	{\rm Diag}\left((\pi'(\mu^q)h^q)_{\beta^l_k}\right)=(\rho^l_{k})^q \, I_{|\beta_{k}^l|}  & \epc \mbox{if $k\in {\cal E}^l$,} \\ [0.1in]
	\widetilde{H}_{\beta^l_k\beta^l_k} & \epc  \mbox{if $k\notin {\cal E}^l$}.
	\end{array} \right. 
		\]
	This further implies the existence of a scalar $\hat{\rho}^l_{k}$ such that
	\[
	\left({\cal U}^{\{i\}}\right)_k^l(D^1(H))=\lim_{q\to \infty}(\Omega^l_{k})^q=\left\{\begin{array}{ll}
	\hat{\rho}^l_{k}I_{|\beta_{k}^l|}  & \epc \mbox{if $k\in {\cal E}^l$,} \\ [0.1in]
	\widetilde{H}_{\beta^l_k\beta^l_k} & \epc \mbox{if $k\notin {\cal E}^l$.}
	\end{array} \right.
	\]
	Taking into account the equality in \eqref{eq:diag=0}, we derive
	\[
	{\rm tr}\left(\widetilde{H}_{\beta^l_k\beta^l_k}\right)=\sum_{i,j\in \iota_1(\lambda(g_1(\overline{\bf{x}})))} \kappa_{ij}\left\langle {\bf e}_{|\beta^l_k|}, \, ({\bf a}^i-{\bf a}^j)_{\beta^l_k}\right\rangle \epc \mbox{for some scalars $\{\kappa_{ij}\}_{i,j\in \iota_1(\lambda(g_1(\overline{\bf{x}})))}$}.
	\] 
	 If $k\notin {\cal E}^l$, then $\widetilde{H}_{\beta^l_k\beta^l_k}=0$.
	This completes the proof of this lemma. \qed

\end{proof}

By comparing  the characterization of Clarke's generalized Jacobian of the proximal mapping  $\Pi_{\cal K}$ in Lemma \ref{prop:subdiff-theta2} with  ${\rm aff}\,({\cal C}(\lambda(g_2(\overline{\bf{x}})+\overline{Z}; {\cal N}_{\cal K}(\lambda(g_2(\overline{\bf{x}}))))$ in Proposition \ref{prop:critical_cone_aff2}, we can obtain the following results with respect to $\partial \theta_2$. Its proof can be obtained similarly as that of Lemma \ref{lemma:V-in-affC}. We omit the details here for brevity. 

\begin{lemma}\label{lemma:V-in-affC2}
	Suppose that $\overline{Z}\in{\cal N}_{\cal K}(g_2(\overline{\bf x}))$ with $\overline{V}\in \mathbb{O}^n(g_2(\overline{\bf x}))\cap \mathbb{O}^n(\overline{Z})$ and  ${\cal S} \in\partial\Pi_{\cal K}(g_2(\overline{\bf x}) + \overline{Z})$. Then
	\[
	{\cal S}(H)\in {\rm aff}\,({\cal C}(g_2(\overline{\bf x}) + \overline{Z}; {\cal N}_{\cal K}(g_2(\overline{\bf x})))),\quad \forall\, H\in\mathbb{S}^n.
	\]
In addition, if ${\cal S}(H)=0$ for some $H\in \mathbb{S}^n$, then the following two conditions hold:
	\begin{itemize}
		\item[(i)] $\overline{V}^\top H\overline{V}\in \mathbb{S}^n$ has the following block diagonal structure:
	\begin{equation}\label{eq:bdiag-triangleZ}
	\overline{V}^\top H\overline{V}={\rm Diag}\left(\overline{V}^\top _{\alpha^1}H\overline{V}_{\alpha^1},\cdots, \overline{V}^\top _{\alpha^r}H\overline{V}_{\alpha^r}\right); 
	\end{equation}
	\item[(ii)] for $l=1,\ldots,r$, let ${\cal F}^{\,l}$ be the index set defined by \eqref{def-E-index2}. If $k\in {\cal F}^{\,l}$, then there exist $\{{\kappa}_{ij}\in \mathbb{R}\}_{i,j\in \iota_2(\lambda(g_2(\overline{\bf{x}})))}$ such that  
	\[
	{\rm tr}\left(\overline{V}^\top _{\gamma^l_k}H\overline{V}_{\gamma^l_k}\right)=\sum_{i,j\in \iota_2(\lambda(g_2(\overline{\bf{x}})))} {\kappa}_{ij}\left\langle \, {\bf e}_{|\gamma^l_k|}\, , \, ({\bf b}^i-{\bf b}^j)_{\gamma^l_k}\,\right\rangle;
	\] 
	otherwise if $k\notin {\cal F}^{\,l}$, then $\overline{V}^\top _{\gamma^l_k}H\overline{V}_{\gamma^l_k}=0$.
		\end{itemize}
\end{lemma}

Finally, we  establish  a connection between the function $\Upsilon_{g_1(\overline{\bf{x}})}^1$ defined in Remark \ref{remark:Upsilon} and the Clarke generalized Jacobian of ${\rm Pr}_{\theta_1}$ given by Proposition \ref{prop:subdiff-theta}. This result plays a key role in our subsequent analysis.

\begin{lemma}\label{prop:inequality-V}
	Suppose that $\overline{Y} \in \partial\theta_1(g_1(\overline{\bf x}))$. Then, for any ${\cal S}\in\partial{\rm Pr}_{\theta_1}(g_1(\overline{\bf x}) + \overline{Y})$ and $ (\triangle{X}, \triangle{Y})\in\mathbb{S}^n\times\mathbb{S}^n$ such that $ \triangle{X}={\cal S}(\triangle{X}+\triangle{Y})$, it holds that
	\[	\left\langle \triangle{X}, \triangle{Y} \right\rangle\ge -\Upsilon^1_{ g_1(\overline{\bf x})}\left( \overline{Y} , \triangle{X}\right).
\]
\end{lemma}
\begin{proof} Denote $H:=\triangle{X}+\triangle{Y}$. For any given ${\cal S}\in\partial{\rm Pr}_{\theta_1}(g_1(\overline{\bf x}) + \overline{Y})$, it is known from Lemma \ref{lemma:V-in-affC} that
	\[
	\overline{U}^\top \triangle X\overline{U}={\cal A}\circ\left(\overline{U}^\top H\overline{U}\right)+{\rm Diag}\left({\cal U}^1_1(D^1(\overline{U}^\top H\overline{U})),\ldots,{\cal U}^r_{s^r}(D^1(\overline{U}^\top H\overline{U}))\right),
	\]
	where $\overline{U}\in\mathbb{O}^n(g_1(\overline{\bf x})) \cap \mathbb{O}^n(\overline{Y})$. This further yields that
\begin{equation}\label{eq:triX-1}
\left\{\begin{array}{lll}	

	\overline{U}^\top _{\alpha^l}\triangle X\overline{U}_{\alpha^{l'}}&=&{\cal A}_{\alpha^l\alpha^{l'}}\circ \left(\;\overline{U}^\top _{\alpha^l}(\triangle{X}+\triangle{Y})\overline{U}_{\alpha^{l'}}\right)\quad \mbox{$\forall \, l,\,l'\in\{1,\ldots,r\}$ with $l\neq l'$},\\ [0.1in]
	\overline{U}^\top _{\alpha^l}\triangle X\overline{U}_{\alpha^{l}}&=&{\rm Diag}\left({\cal U}^l_{1}(D^1(\overline{U}^\top H\overline{U})),\ldots,{\cal U}^l_{s_l}(D^1(\overline{U}^\top H\overline{U})) \right) \quad \forall\, l\in\{1,\ldots,r\}.
	\end{array}\right.
	\end{equation}
		Therefore, we have
	\begin{eqnarray*}
		&&\left\langle \triangle{X}, \triangle{Y} \right\rangle = \left\langle \overline{U}^\top \triangle{X}\overline{U}, \overline{U}^\top \triangle{Y}\overline{U} \right\rangle \\ [3pt]
		&=&\sum_{l=1}^r\langle \overline{U}^\top _{\alpha^l}\triangle X\overline{U}_{\alpha^{l}}, \overline{U}^\top _{\alpha^l}\triangle Y\overline{U}_{\alpha^{l}}\rangle + 2\sum_{1\le l<l'\le r} \langle \overline{U}^\top _{\alpha^l}\triangle X\overline{U}_{\alpha^{l'}}, \overline{U}^\top _{\alpha^l}\triangle Y\overline{U}_{\alpha^{l'}}\rangle \\ [3pt]
		&=& \sum_{l=1}^r\sum_{k=1}^{s_l}\langle \overline{U}^\top _{\beta^l_k}\triangle X\overline{U}_{\beta^l_k}, \overline{U}^\top _{\beta^l_k}\triangle Y\overline{U}_{\beta^l_k}\rangle+ 2\sum_{1\le l<l'\le r} \langle \overline{U}^\top _{\alpha^l}\triangle X\overline{U}_{\alpha^{l'}}, \overline{U}^\top _{\alpha^l}\triangle Y\overline{U}_{\alpha^{l'}}\rangle \\ [3pt]
		&=& \sum_{l=1}^r\sum_{k=1}^{s_l}\langle \, {\cal U}^l_{k}(D^1(\overline{U}^\top H\overline{U})), \overline{U}^\top _{\beta^l_k}\triangle Y\overline{U}_{\beta^l_k}\rangle+ 2\sum_{1\le l<l'\le r} \langle \overline{U}^\top _{\alpha^l}\triangle X\overline{U}_{\alpha^{l'}}, \overline{U}^\top _{\alpha^l}\triangle Y\overline{U}_{\alpha^{l'}}\rangle \\ [3pt]
		&=& \sum_{l=1}^r\sum_{k=1}^{s_l}\langle \, {\cal U}^l_{k}(D^1(\overline{U}^\top H\overline{U})), \overline{U}^\top _{\beta^l_k}H\overline{U}_{\beta^l_k}-\overline{U}^\top _{\beta^l_k}\triangle X\overline{U}_{\beta^l_k}\rangle+ 2\sum_{1\le l<l'\le r} \langle \overline{U}^\top _{\alpha^l}\triangle X\overline{U}_{\alpha^{l'}}, \overline{U}^\top _{\alpha^l}\triangle Y\overline{U}_{\alpha^{l'}}\rangle \\ [3pt]
		&=& \sum_{l=1}^r\sum_{k=1}^{s_l}\langle \, {\cal U}^l_{k}(D^1(\overline{U}^\top H\overline{U})), \overline{U}^\top _{\beta^l_k}H\overline{U}_{\beta^l_k}-{\cal U}^l_{k}(D(\overline{U}^\top H\overline{U}))\rangle+ 2\sum_{1\le l<l'\le r} \langle \overline{U}^\top _{\alpha^l}\triangle X\overline{U}_{\alpha^{l'}}, \overline{U}^\top _{\alpha^l}\triangle Y\overline{U}_{\alpha^{l'}}\rangle \\ [3pt]
		&=& \left\langle {\cal U}(D^1(\overline{U}^\top H\overline{U})), D^1(\overline{U}^\top H\overline{U})- {\cal U}(D^1(\overline{U}^\top H\overline{U}))\right\rangle +2\sum_{1\le l<l'\le r} \langle \overline{U}^\top _{\alpha^l}\triangle X\overline{U}_{\alpha^{l'}}, \overline{U}^\top _{\alpha^l}\triangle Y\overline{U}_{\alpha^{l'}}\rangle.
	\end{eqnarray*}
	Since ${\cal U}\in \partial\Pi_{\Delta^1}(0)$ and $\Delta^1$ is a nonempty closed convex set defined by \eqref{eq:def-Delta}, we know from \cite[Proposition 1 (c)]{MSZhao05} that 
	$$
	\left\langle {\cal U}(D^1(\overline{U}^\top H\overline{U})), D^1(\overline{U}^\top H\overline{U})- {\cal U}(D^1(\overline{U}^\top H\overline{U}))\right\rangle\ge 0.
	$$ Therefore, 
\[
	\left\langle \triangle{X}, \triangle{Y} \right\rangle \ge  2\sum_{1\le l<l'\le r} \langle \overline{U}^\top _{\alpha^l}\triangle X\overline{U}_{\alpha^{l'}}, \overline{U}^\top _{\alpha^l}\triangle Y\overline{U}_{\alpha^{l'}}\rangle.
\]
		On the other hand, one can easily verify from Remark \ref{remark:Upsilon} and \eqref{eq:triX-1} that
	\begin{equation*}
	-\Upsilon^1_{ g_1(\overline{\bf x})}\left( \overline{Y} , \triangle{X}\right)=2\sum_{1\le l<l'\le r} \langle \overline{U}^\top _{\alpha^l}\triangle X\overline{U}_{\alpha^{l'}}, \overline{U}^\top _{\alpha^l}\triangle Y\overline{U}_{\alpha^{l'}}\rangle.
	\end{equation*}
	Combining the above equality and inequality together, we establish the desired result of this lemma. \qed 
\end{proof}

We also have the following analogous result with respect to the function $\partial\theta_2$. 

\begin{lemma}\label{prop:inequality-V2}
Suppose that $\overline{Z}\in \mathcal{N}_{\mathcal{K}}(g_2(\overline{\bf x}))$. Then for any ${\cal S}\in\partial\Pi_{\cal K}(g_2(\overline{\bf x}) + \overline{Z})$ and $ (\triangle{X}, \triangle{Z})\in\mathbb{S}^n\times\mathbb{S}^n$ such that $ \triangle{X}={\cal S}(\triangle{X}+\triangle{Z})$, it holds that
	\begin{equation}\label{eq:inequality-Upsilon2}
	\left\langle \triangle{X}, \triangle{Z} \right\rangle\ge -\Upsilon^2_{ g_2(\overline{\bf x})}\left( \overline{Z} , \triangle{X}\right).
	\end{equation} 
\end{lemma}

The following theorem, which is the main result of this paper, establishes the relationship between the strong second-order sufficient condition \eqref{eq:ssoc}  and constraint nondegeneracy \eqref{eq:nondegen} for problem \eqref{opt-eq}, the non-singularity of Clarke's Jacobian of the mapping $F$ and the strong regularity of a solution to the generalized equation \eqref{eq:GE-KKT}.

\begin{theorem}\label{prop:equivalent-1}
Let $\overline{\bf x}\in\mathbb{X}$ be a feasible solution to problem \eqref{opt-eq} with ${\cal M}(\overline{\bf x})\neq \emptyset$. Suppose that $(\overline{\bf y},\overline{Y},\overline{Z})\in{\cal M}(\overline{\bf x})$. Consider the following three statements:
\vskip 0.1in
\noindent
(i) the strong second-order sufficient condition \eqref{eq:ssoc} and constraint nondegeneracy \eqref{eq:nondegen} hold at $\overline{\bf x}$ for problem \eqref{opt-eq};
\vskip 0.1in
\noindent
(ii) every element in $\partial F(\overline{\bf x},\overline{\bf y},\overline{Y},\overline{Z})$ is nonsingular;
\vskip 0.1in
\noindent
(iii) $(\overline{\bf x},\overline{\bf y},\overline{Y},\overline{Z})$ is a strongly regular solution of the generalized equation \eqref{eq:GE-KKT}.
\vskip 0.1in
\noindent
It holds that $(i)\Longrightarrow(ii)\Longrightarrow(iii)$.
\end{theorem}
\begin{proof} ``(i) $\Longrightarrow$ (ii)'' Since the constraint nondegeneracy \eqref{eq:nondegen} holds at $\overline{\bf x}$, we know that ${\cal M}(\overline{\bf x})=\{(\overline{\bf y},\overline{Y},\overline{Z})\}$. Then the strong second-order sufficient condition in \eqref{eq:ssoc} reduces to
\begin{equation}\label{eq:SSOSC-MCP'}
\langle {\bf d}, {\cal L}''_{{\bf x}{\bf x}}(\overline{\bf x},\overline{\bf y},\overline{Y},\overline{Z}){\bf d}\rangle -\Upsilon_{g_1(\overline{\bf x})}^1\left(\overline{Y},g_1'(\overline{\bf x}){\bf d}\right) - \Upsilon_{g_2(\overline{\bf x})}^2\left(\overline{Z},g_2'(\overline{\bf x}){\bf d}\right)    > 0,\quad  \forall\,{\bf d}\in{\rm app}(\overline{\bf y},\overline{Y},\overline{Z})\setminus\{0\}.
\end{equation}
Let ${\bm W}$ be an arbitrary element in $\partial F(\overline{\bf x},\overline{\bf y},\overline{Y},\overline{Z})$. We shall show that ${\bm W}$ is nonsingular. Suppose that $(\triangle {\bf x},\triangle {\bf y},\triangle Y,\triangle Z)\in\mathbb{X}\times\mathbb{Y}\times\mathbb{S}^{n}\times\mathbb{S}^{n}$ satisfies ${\bm W}(\triangle {\bf x},\triangle {\bf y},\triangle Y,\triangle Z)=0$. By \eqref{eq:def-B-Jocbian-F}, we know that there exists ${\cal S}^1\in\partial{\rm Pr}_{\theta_1}(g_1(\overline{\bf x})+\overline{Y})$ and ${\cal S}^2\in\partial\Pi_{\cal K}(g_2(\overline{\bf x})+\overline{Z})$ such that
\[
{\bm W}\left(\triangle {\bf x} ,\triangle {\bf y}, \triangle Y,\triangle Z\right)=\left[\begin{array}{c}
{\cal L}^{''}_{{\bf x}{\bf x}}(\overline{\bf x},\overline{\bf y},\overline{Y},\overline{Z})\triangle {\bf x}+h'(\overline{\bf x})^{*}\triangle {\bf y}+g_1'(\overline{\bf x})^*\triangle Y+g_2'(\overline{\bf x})^*\triangle Z\\ [0.1in]
h'(\overline{\bf x})\triangle {\bf x} \\ [0.1in]
g_1'(\overline{\bf x})\triangle {\bf x}-{\cal S}^1(g_1'(\overline{\bf x})\triangle {\bf x}+\triangle Y) \\ [0.1in]
g_2'(\overline{\bf x})\triangle {\bf x}-{\cal S}^2(g_2'(\overline{\bf x})\triangle {\bf x}+\triangle Z)
\end{array}\right]=0.
\]
It then follows from Lemma \ref{lemma:V-in-affC} and Lemma \ref{lemma:V-in-affC2} that 
\[\left\{\begin{array}{ll}
g_1'(\overline{\bf x})\triangle {\bf x}={\cal S}^1(g_1'(\overline{\bf x})\triangle {\bf x}+\triangle Y) \in {\rm aff}\left(\,{\cal C}(g_1(\overline{\bf x})+\overline{Y};\partial\theta_1(g_1(\overline{\bf x})))\,\right),\\[0.1in]
g_2'(\overline{\bf x})\triangle {\bf x}={\cal S}^2(g_2'(\overline{\bf x})\triangle {\bf x}+\triangle Z)\in {\rm aff}\left(\,{\cal C}(g_2(\overline{\bf x})+\overline{Z};{\cal N}_{\cal K}(g_2(\overline{\bf x}))) \,\right).
\end{array}\right.
\]
We thus obtain from \eqref{eq:def-app} that $\triangle {\bf x}\in{\rm app}(\overline{\bf y},\overline{Y},\overline{Z})$. 
In addition, we derive from Propositions \ref{prop:inequality-V} and \ref{prop:inequality-V2}  that
\begin{eqnarray}
0&=&\left\langle \triangle {\bf x}, \,{\cal L}^{''}_{{\bf x}{\bf x}}(\overline{\bf x},\overline{\bf y},\overline{Y},\overline{Z})\triangle {\bf x}+h'(\overline{\bf x})^{*}\triangle {\bf y}+g_1'(\overline{\bf x})^*\triangle Y+g_2'(\overline{\bf x})^*\triangle Z\right\rangle \nonumber \\ [0.05in]
&=& \left\langle \triangle {\bf x}, \,{\cal L}^{''}_{{\bf x}{\bf x}}(\overline{\bf x},\overline{\bf y},\overline{Y},\overline{Z})\triangle {\bf x} \rangle + \langle g_1'(\overline{\bf x})\triangle {\bf x},  \triangle Y\rangle + \langle g_2'(\overline{\bf x})\triangle {\bf x},  \triangle Z\right\rangle \nonumber \\ [0.05in]
&\ge& \left\langle \triangle {\bf x}, {\cal L}^{''}_{{\bf x}{\bf x}}(\overline{\bf x},\overline{\bf y},\overline{Y},\overline{Z})\triangle {\bf x} \right\rangle -\Upsilon^1_{g_1(\overline{\bf x})}\left(\overline{Y},g_1'(\overline{\bf x})\triangle {\bf x}\right)-\Upsilon^2_{g_2(\overline{\bf x})}\left(\overline{Z},g_2'(\overline{\bf x})\triangle {\bf x}\right). \label{eq:SSOSC-MCP'-inv}
\end{eqnarray}
Since $\triangle {\bf x}\in{\rm app}(\overline{\bf y},\overline{Y},\overline{Z})$, we conclude from \eqref{eq:SSOSC-MCP'} and \eqref{eq:SSOSC-MCP'-inv} that $\triangle {\bf x}=0$ and
\begin{equation}\label{eq:def-Jocbian-F-reduce}
\left[\begin{array}{c}
h'(\overline{\bf x})^{*}\triangle {\bf y}+g_1'(\overline{\bf x})^*\triangle Y+g_2'(\overline{\bf x})^*\triangle Z \\ [3pt]
{\cal S}^1(\triangle Y) \\ [3pt]
{\cal S}^2(\triangle Z)
\end{array}\right]=0.
\end{equation}
 By the assumed constraint nondegeneracy condition \eqref{eq:nondegen}, we know that there exist ${\bf h}\in\mathbb{X}$, $H^1\in{\cal T}^{\rm lin}_{\theta_1}(g_1(\overline{\bf x}))$ and $H^2\in{\rm lin}\,\left({\cal T}_{\cal K}(g_2(\overline{\bf x}))\right)$ such that
\[
h'(\overline{\bf x}){\bf h}=\triangle {\bf y}, \quad g_1'(\overline{\bf x}){\bf h}+H^1=\triangle Y \quad {\rm and} \quad g_2'(\overline{\bf x}){\bf h}+H^2=\triangle Z.
\]
It then follows from the first equation of \eqref{eq:def-Jocbian-F-reduce} that 
\begin{eqnarray*}
&& \langle\triangle {\bf y},\triangle {\bf y} \rangle+\langle \triangle Y, \triangle Y \rangle+\langle \triangle Z, \triangle Z \rangle \nonumber \\ [0.05in]
 &=& \langle h'(\overline{\bf x}){\bf h},\,\triangle {\bf y}\rangle+\langle g_1'(\overline{\bf x}){\bf h}+H^1, \, \triangle Y \rangle + \langle g_2'(\overline{\bf x}){\bf h}+H^2, \, \triangle Z \rangle \nonumber \\ [0.05in]
&=&\langle {\bf h}, h'(\overline{\bf x})^{*}\triangle {\bf y}+g_1'(\overline{\bf x})^*\triangle Y+g_2'(\overline{\bf x})^*\triangle Z\rangle + \langle H^1,\triangle Y\rangle + \langle H^2,\triangle Z\rangle \nonumber \\ [0.05in]
&=& \langle H^1,\triangle Y\rangle + \langle H^2,\triangle Z\rangle = \left\langle \overline{U}^\top H^1\overline{U}, \, \overline{U}^\top \triangle Y\overline{U}\right\rangle + \left\langle \overline{V}^\top H^2\overline{V}, \, \overline{V}^\top \triangle Z\overline{V}\right\rangle. %\label{eq:yy-zetazeta-sqrt}
\end{eqnarray*}
Combining Proposition \ref{prop:chara-Tlin}, Proposition \ref{prop:chara-Tlin2}, Lemma \ref{lemma:V-in-affC} and Lemma \ref{lemma:V-in-affC2},  we derive
\begin{eqnarray*}
	\left\langle \overline{U}^\top H^1\overline{U},\overline{U}^\top \triangle Y\overline{U}\right\rangle &=& \sum_{1\le l\le r} \sum_{k\in {\cal E}^l}\left\langle \overline{U}_{\beta^l_k}^\top H^1\overline{U}_{\beta^l_k}, \,\overline{U}_{\beta^l_k}^\top \triangle Y\overline{U}_{\beta^l_k} \right\rangle \\ [3pt]
	&=& \sum_{1\le l\le r} \sum_{k\in {\cal E}^l} \left\langle \widehat{\rho}_l \, {\bf e}_{|\beta_k^l|}, \, {\rm diag}\left((R^l_{k})^\top  \overline{U}^\top _{\beta^l_k}\,\triangle Y\,\overline{U}_{\beta^l_k}R^l_{k}\right)\right\rangle=0
\end{eqnarray*}
and
\begin{eqnarray*}
\left\langle \overline{V}^\top H^2\overline{V},\, \overline{V}^\top \triangle Z\overline{V}\right\rangle &=& \sum_{1\le l\le r} \sum_{k\in {\cal F}^l}\left\langle\, \overline{V}_{\gamma^l_k}^\top H^2\overline{V}_{\gamma^l_k},\, \overline{V}_{\gamma^l_k}^\top \triangle Z\overline{V}_{\gamma^l_k} \,\right\rangle \\ [3pt]
&=& \sum_{1\le l\le r} \sum_{k\in {\cal F}^l}\left\langle \,\widetilde{\rho}_l \, {\bf e}_{|\gamma_k^l|}, \, {\rm diag}\left((R^l_{k})^\top  \overline{V}^\top _{\gamma^l_k}\,\triangle Z\,\overline{V}_{\gamma^l_k}R^l_{k}\right)\,\right\rangle=0.
\end{eqnarray*}
Therefore, we get
\[
\langle\triangle {\bf y},\triangle {\bf y} \rangle+\langle \triangle Y, \triangle Y \rangle+\langle \triangle Z, \triangle Z \rangle= 0,
\]
which implies that ${\bm W}\in\partial F(\overline{\bf x},\overline{\bf y},\overline{Y},\overline{Z})$ is nonsingular.

``(ii) $\Longrightarrow$ (iii)'' By Clarke's inverse function theorem \cite{Clarke76,Clarke83}, we know that $F$ is a locally Lipschitz homeomorphism near $(\overline{\bf x},\overline{\bf y},\overline{Y},\overline{Z})$. Thus, we know from Lemma \ref{lem:equivalence-strong-regularity-Lip-Homeomorphism} that $(\overline{\bf x},\overline{\bf y},\overline{Y},\overline{Z})$ is a strongly regular solution of the generalized equation \eqref{eq:GE-KKT}. \qed
\end{proof}

It has been proved in \cite{Sun06} that for the nonlinear semidefinite programming problem, the 
three statements in Theorem \ref{prop:equivalent-1} are actually equivalent. Regrettably, we fail to establish the implication (iii) $\Longrightarrow$ (i) for the general CMatOPs. Nevertheless, for a class of special nonsmooth matrix optimization problem involving the largest eigenvalue, we can show the equivalence of these three conditions via the relationship between the so-called uniform quadratic growth condition and the strong second-order sufficient condition. This is the subject of the next section.
%	It seems difficult to show the inverse direction ``(iii) $\Longrightarrow$ (i)" for the general CMatOPs. However, for some particular CMatOs (e.g., the nonlinear SDP, cf. \cite{Sun06}), one may able to show that the uniform second-order growth condition (cf. e.g., \cite[Definition 5.16]{BShapiro00}) of the CMatOP implies the strong second-order sufficient condition (Definition \ref{def:ssoc}). Then, it follows from \cite[Theorem 5.24]{BShapiro00} that condition (iii) is equivalent to the uniform second-order growth condition and constraint nondegeneracy.  

%%%%%%%%%%%%%%%%%
\section{Applications to CMatOPs involving the largest eigenvalue}\label{sec:example}
In this section, we apply our obtained results to a class of CMatOPs involving the largest eigenvalue of a symmetric matrix. Relying on its special structure, we improve the results in Theorem \ref{prop:equivalent-1} by showing that the three statements therein are actually equivalent.

Specifically,  we consider the following problem
\begin{equation}\label{opt-eq-eig}
\begin{array}{cl}
\displaystyle\operatornamewithlimits{minimize}_{{\bf x}\in \mathbb{X}} &\; f({\bf x})  + \lambda_1(g({\bf x}))\\[0.1in] 
\mbox{subject to} &\; h({\bf x})=0,
\end{array}
\end{equation}
where $\lambda_1(g({\bf x}))$ denotes the largest eigenvalue of a symmetric matrix $g({\bf x})$.
This corresponds to a special case of problem \eqref{opt} where 
\[
\phi({\bf x})=\max_{1\le i\le p}\left\{\langle {\bf a}^i,{\bf x}\rangle \right\},\quad {\bf x}\in \mathbb{R}^n
\]
with ${\bf a}^i$ being the unit vector whose $i$-th component is 1 and others are zero. Based on the formulas derived in Section \ref{sec: spectral functions}, we get the following results.

\vskip 10 true pt 
\noindent
\underline{\bf The tangent cone and its lineality space}. 
Recall the definitions of $\{\alpha^l\}$ in \eqref{eq:ak-symmetric}
and $\iota_1(\lambda(\overline{X}))$  in \eqref{index1}. We have
\[
\iota_1(\lambda(\overline{X}))=\left\{1\le i\le n\mid \lambda_i(\overline{X})=\bar{v}_1 \right\}=\alpha^1.
\]
%Let ${\cal T}^{\rm lin}_{\theta_1}(\overline{X})$ be the affine space defined by \eqref{eq:def-Tlin}. Then, i
It follows from Proposition \ref{prop:chara-Tlin}  that
\[
H\in {\cal T}^{\rm lin}_{\theta_1}(\overline{X}) \;\Longleftrightarrow\; \left[\, \lambda'_i(\overline{X};H)=\lambda'_j(\overline{X};H),\quad \forall\,i,j\in\alpha^1 \,\right].
\]
Based on Lemma \ref{prop:eigenvalue_diff}, the above right-side is further equivalent to the existence of  a scalar $\widehat{\rho}$ such that
\[
\overline{U}^\top _{\alpha^1}H\overline{U}_{\alpha^1}=\widehat{\rho}\,I_{|\alpha^1|} \;\; \mbox{for some $\overline{U}\in\mathbb{O}^n(\overline{X})$}.
\]
The value of $\widehat{\rho}$ is independent of the selected orthogonal matrix  $\overline{U}$ in $\mathbb{O}^n(\overline{X})$ (\cite[Proposition 2]{DSSToh14}).

\vskip 10 true pt 
\noindent
\underline{\bf The critical cone}. Given $(\overline{X},\overline{Y})\in {\rm gph}\,\partial\lambda_1$ and let $\overline{U}\in \mathbb{O}^n(\overline{X})\cap \mathbb{O}^n(\overline{Y})$. It follows from Lemma \ref{lemma:subdiff-spectral-function} and \cite[Lemma 2.2]{WDSToh14} (see also \cite[Lemma 3.1]{OWomersley93}) that 
\begin{equation}\label{eq:lambda-Y-max}
\left\{\begin{array}{l}
	0\le  \lambda_i(\overline{Y}) \le 1,\quad \forall\,i\in \alpha^1\quad {\rm and}\quad   \displaystyle\sum_{i\in\alpha^1}\lambda_i(\overline{Y})=1,   \\ [3pt]
	\lambda_i(\overline{Y})=0,\quad  \forall\,i\in \alpha^l,\quad  l=2,\ldots,r.
\end{array} \right.
\end{equation}
Denote
\begin{equation}\label{eq:def-mu-nv}
	\mu:=\{i\in\alpha^1\mid \lambda_i(\overline{Y})>0\} \quad {\rm and}\quad \nu:=\{i\in\alpha^1\mid \lambda_i(\overline{Y})=0\}.
\end{equation}
For each $l\in\{1,\ldots,r\}$, we further partition  the index set $\alpha^l$ by  $\{\beta_k^l\}_{k=1}^{s_l}$ as in \eqref{eq:def-beta-index} based on  $\lambda(\overline{Y})$. We then obtain from \eqref{eq:lambda-Y-max} that 
\[
\iota_1(\lambda(\overline{X}))=\alpha^1=\mu\cup\nu,\quad \mu=\bigcup_{k=1}^{s^1-1}\beta_k^1, \quad \nu=\beta_{s^1}^1  \quad {\rm and} \quad \alpha^l=\beta_1^l  \quad {\rm for}\quad l=2,\ldots,r. 
\]
Recall the index set $\eta_1(\lambda(\overline{X}),\lambda(\overline{Y}))\subseteq \iota_1(\lambda(\overline{X}))$ defined in \eqref{eq:def-index-eta-1}. Obviously $\eta_1(\lambda(\overline{X}),\lambda(\overline{Y}))=\mu$. It follows from Proposition \ref{prop:critical_cone-eq}  that
\[
H\in {\cal C}(\overline{X} + \overline{Y};\partial\lambda_1(\overline{X})) \;\Longleftrightarrow	\; \left[\, \left({\rm diag}(\overline{U}^\top H\overline{U}) \right)_i= \lambda_1(\overline{U}^\top_{\alpha^1} H\overline{U}_{\alpha^1}),  \quad \forall\, i \in \mu\,\right].
\] 
One can also derive from Proposition \ref{prop:critical_cone_aff} that 
\begin{equation}\label{eq:critical-aff-max1}
H\in {\rm aff}\,({\cal C}(A;\partial\lambda_1(\overline{X}))) \;\Longleftrightarrow \;   \overline{U}^\top_{\alpha^1} H\overline{U}_{\alpha^1}=\begin{bmatrix}
	\rho\, I_{|\mu|} & 0 \\ 
	0 & \overline{U}^\top_{\nu} H\overline{U}_{\nu}
\end{bmatrix}\;\; \mbox{for some $\rho \in \mathbb{R}$}.
\end{equation}

\noindent
\underline{\bf The sigma term}.  Also given $(\overline{X},\overline{Y})\in {\rm gph}\,\partial\lambda_1$ and let $\overline{U}\in \mathbb{O}^n(\overline{X})\cap \mathbb{O}^n(\overline{Y})$. Denote $\omega:=\displaystyle \bigcup_{l=2}^r\alpha^l$. By noting that $\lambda_i(\overline{Y})=0$ for any $i\in \nu\cup \omega$, we derive from  \eqref{eq:Upsilon1-eq} and \eqref{eq:lambda-Y-max} that 
\begin{equation}\label{eq:Upsilon-eig}
\Upsilon^1_{\overline{X}}(\overline{Y},H)=-2\sum_{i\in \mu}\sum_{j\in \omega}\frac{\lambda_i(\overline{Y})}{\lambda_i(\overline{X})-\lambda_j(\overline{X})}(\overline{U}_{\mu}^\top H\overline{U}_{\omega})_{ij}^2\le 0,\quad H\in\mathbb{S}^n.
\end{equation}

\gap

In the rest of this section, we show that the strong regularity of the generalized equation for the KKT system at a local optimal optimal of problem \eqref{opt-eq-eig} implies the strong second-order sufficient condition and the constraint nondegeneracy at the same point, i.e., the three statements in  Theorem \ref{prop:equivalent-1} are equivalent.

Given a feasible point $\overline{\bf x}$ of problem \eqref{opt-eq-eig}, we say Robinson's constraint qualification (CQ) \cite{Robinson76} at $\overline{\bf x}$ holds if
\begin{equation}\label{eq:RCQ}
h'(\overline{\bf x})\,\mathbb{X}=\mathbb{Y}.
\end{equation} 
It has been proved in \cite[Proposition 3.3]{CDZ2017} that the function $\lambda_1(\,\bullet\,)$ is ${\cal C}^2$-cone reducible at any point so that the  set ${\rm epi}\,\lambda_1$ is second-order regular  \cite[Proposition 3.136]{BShapiro00} (see \cite[Definitions 3.85 \& 3.135]{BShapiro00} for the definitions of ${\cal C}^2$-cone reducibility and second-order regularity). One can then obtain the following second-order necessary and sufficient conditions of \eqref{opt-eq-eig} by adapting the proof of \cite[Theorems 3.45 \& 3.86]{BShapiro00}. For brevity, we omit the detailed proof here.

\begin{proposition}\label{thm:second-order-opt}
	Suppose that $\overline{\bf x}\in \mathbb{X}$ is a local optimal solution of \eqref{opt-eq-eig} and Robinson's CQ \eqref{eq:RCQ} holds at $\overline{\bf x}$. Then the following second-order necessary condition holds at $\overline{\bf x}$:
\[
	\sup_{(\overline{\bf y},\overline{Y})\in{\cal M}(\overline{\bf x})}\left\{\langle {\bf d}, {\cal L}''_{{\bf x}{\bf x}}(\overline{\bf x},\overline{\bf y},\overline{Y}){\bf d}\rangle -\Upsilon_{g(\overline{\bf x})}^1\left(\overline{Y},g'(\overline{\bf x}){\bf d}\right) \right\} \ge 0, \quad \forall\, {\bf d}\in {\cal C}(\overline{\bf x}),
\]
where $\Upsilon_{g(\overline{\bf x})}^1$ is given by \eqref{eq:Upsilon-eig}.	Conversely, let $\overline{\bf x}$ be a feasible point of  \eqref{opt-eq-eig} and assume Robinson's CQ \eqref{eq:RCQ} holds at $\overline{\bf x}$. Then the following condition 
	\[
	\sup_{(\overline{\bf y},\overline{Y})\in{\cal M}(\overline{\bf x})}\left\{\langle {\bf d}, {\cal L}''_{{\bf x}{\bf x}}(\overline{\bf x},\overline{\bf y},\overline{Y}){\bf d}\rangle -\Upsilon_{g(\overline{\bf x})}^1\left(\overline{Y},g'(\overline{\bf x}){\bf d}\right) \right\} > 0, \quad \forall\, {\bf d}\in {\cal C}(\overline{\bf x})\setminus\{0\}
\]
	is necessary and sufficient for the existence of a positive scalar $\rho$ and a neighborhood $\mathcal{N}$ of $\overline{\bf x}$ such that
	\begin{equation}\label{eq:quadratic growth}
		f({\bf x})+\lambda_1(g({\bf x}))\ge f(\overline{\bf x})+\lambda_1(g(\overline{\bf x}))+\rho\|{\bf x}-\overline{\bf x}\|^2\quad \forall\,{\bf x}\in \mathcal{N} \ \mbox{such that $h({\bf x})=0$}.
\end{equation}
\end{proposition}

%By employing the similar argument with \cite[Lemma 4.1]{Sun06}, we obtain the following proposition, which shows that  the uniform second-order growth condition of the  CMatOP \eqref{opt-eq-eig} implies the strong second-order sufficient condition (Definition \ref{def:ssoc}).   
The inequality \eqref{eq:quadratic growth} is usually called the quadratic growth condition at $\bar{x}$ of problem \eqref{opt-eq-eig}. In the conventional nonlinear programming,  there is a stronger concept termed uniform quadratic growth condition  \cite[Definition 5.16]{BShapiro00}. Let $\mathbb{U}$ be a Banach space and consider functions $f:\mathbb{X}\times\mathbb{U}\to \mathbb{R}$, $g:\mathbb{X}\times\mathbb{U}\to\mathbb{S}^n$ and $h:\mathbb{X}\times\mathbb{U}\to \mathbb{Y}$. We say that $(f({\bf x},{\bf u}),g({\bf x},{\bf u}),h({\bf x},{\bf u}))$ 
%with ${\bf u}\in\mathbb{U}$, 
is a $C^2$-smooth parameterization of \eqref{opt-eq-eig} if $f(\,\bullet\,,\,\bullet\,)$, $g(\,\bullet\,,\,\bullet\,)$ and $h(\,\bullet\,,\,\bullet\,)$ are twice continuously differentiable and there exits  $\overline{\bf u}\in\mathbb{U}$ such that $f(\,\bullet\,,\overline{\bf u})\equiv f(\,\bullet\,)$, $g(\,\bullet\,,\overline{\bf u})\equiv g(\,\bullet\,)$ and $h(\,\bullet\,,\overline{\bf u})\equiv h(\,\bullet\,)$.
Let $\overline{\bf x}\in\mathbb{X}$ be a stationary point of problem \eqref{opt-eq-eig}. We say that the uniform quadratic growth condition holds at $\overline{\bf x}$ with respect to a $C^2$-smooth parameterization $(f({\bf x},{\bf u}),g({\bf x},{\bf u}),h({\bf x},{\bf u}))$ if there exist $\rho>0$ and neighborhoods ${\cal V}$ of $\overline{\bf x}$ and ${\cal U}$ of $\overline{\bf u}$ such that for any ${\bf u}\in {\cal U}$ and any stationary point ${\bf x}({\bf u})\in {\cal V}$ of the corresponding parameterized problem, the following holds:
	\[
	f({\bf x},{\bf u})+\lambda_1(g({\bf x},{\bf u}))\ge f({\bf x}({\bf u}),{\bf u})+\lambda_1(g({\bf x}({\bf u}),{\bf u}))+\rho\|{\bf x}-{\bf x}({\bf u})\|^2, \quad \forall\,{\bf x}\in {\cal V}\;\, \mbox{such that $h({\bf x},{\bf u})=0$}.
	\]
	We say that the uniform quadratic growth condition holds at $\overline{\bf x}$ if the above inequality holds for every $C^2$-smooth parameterization of \eqref{opt-eq-eig}.

While Proposition \ref{thm:second-order-opt} indicates the equivalence of the quadratic growth condition and the second-order sufficient condition, 
the following proposition shows that the uniform quadratic growth condition of problem \eqref{opt-eq-eig} at a stationary solution implies the strong second-order sufficient condition at that point. Its proof is similar to \cite[Lemma 4.1]{Sun06} for the nonlinear semidefinite programming problem.
\begin{proposition}\label{prop:uniform-sog}
	Let $\overline{\bf x}\in\mathbb{X}$ be a stationary point of problem \eqref{opt-eq-eig}. Suppose that Robinson's CQ \eqref{eq:RCQ}  holds at $\overline{\bf x}$. If the uniform quadratic growth condition holds at $\overline{\bf x}$, then the strong second-order sufficient condition \eqref{eq:ssoc} holds at $\overline{\bf x}$.
\end{proposition} 
\begin{proof}
	Denote $\overline{X}:=g(\overline{\bf x})$. Let $(\overline{\bf y},\overline{Y})\in {\cal M}(\overline{\bf x})$ and  $\overline{U}\in\mathbb{O}^n(\overline{X})\cap\mathbb{O}^n(\overline{Y})$. 
%	Let $\overline{\iota}_1:=\iota_1(\lambda(\overline{X}))$ and $\overline{\eta}_1:=\eta_1(\lambda(\overline{X}),\lambda(\overline{Y}))$ be the index sets defined in \eqref{index1} and \eqref{eq:def-index-eta-1}, respectively. 
	 Recall the index sets $\mu$ and $\nu$ defined by \eqref{eq:def-mu-nv}.
Given a positive scalar $\tau$, we consider the following  problem:
\[
\begin{array}{cl}
\displaystyle\operatornamewithlimits{minimize}_{{\bf x}\in \mathbb{X}} &\; f({\bf x})  + \lambda_1\left(g({\bf x})-\tau\overline{U}_{\nu}\overline{U}_{\nu}^\top \right)\\[0.1in] 
\mbox{subject to} &\; h({\bf x})=0.
\end{array}
\]
Let ${\cal M}_{\tau}(\overline{\bf x})$ be the set of all $({\bf y},Y)\in \mathbb{Y}\times \mathbb{S}^n$ such that $(\overline{\bf x}, {\bf y}, Y)$ satisfies the  KKT optimality condition of the above problem, and ${\cal C}_{\tau}(\overline{\bf x})$ be the the critical cone defined in \eqref{eq:def-critical-MOP}  with respect to  the above problem. 
%\begin{equation}\label{eq:KKT-param}
%	\left\{ 
%	\begin{array}{l}
%		{\cal L}'_{\bf x}({\bf x}, {\bf y}, {Y})=0,\quad h({\bf x})=0, \\ [0.1in]
%		Y\in \partial\lambda_1\left(g({\bf x})-\tau\overline{U}_{\nu}\overline{U}_{\nu}^\top \right).
%	\end{array}
%\right.
%\end{equation}
For all sufficiently small $\tau$, we have 
\[
\iota_1\left(\lambda\left(\overline{X}-\tau\overline{U}_{\nu}\overline{U}_{\nu}^\top \right)\right) \, \equiv \, \eta_1(\lambda(\overline{X}),\lambda(\overline{Y})) \, = \, \mu,
\]
where the set $\iota_1$ and $\eta_1$ are defined in \eqref{index1} and \eqref{eq:def-index-eta-1}.  It then follows from 
\eqref{eq:lambda-Y-max} and the above equality that for such sufficiently small $\tau$, $(\overline{\bf y}, \overline{Y})\in {\cal M}_{\tau}(\overline{\bf x})\subseteq {\cal M}(\overline{\bf x})$. 
%and
%\begin{equation}\label{eq:KKT-param}
%	\left\{ 
%	\begin{array}{l}
%		{\cal L}'_{\bf \bar{x}}({\bf \bar{x}}, {\bf \bar{y}}, \overline{Y})=0,\quad h({\bf \bar{x}})=0, \\ [0.1in]
%		\overline{Y}\in \partial\lambda_1\left(g({\bf \bar{x}})-\tau\overline{U}_{\nu}\overline{U}_{\nu}^\top \right).
%	\end{array}
%\right.
%\end{equation}
%Moreover, it is easy to see from \eqref{eq:para-iota=eta} and \eqref{eq:lambda-Y-max} that for all sufficiently small $\tau>0$, 
%\begin{equation}\label{eq:M-subset}
%	{\cal M}_{\tau}(\overline{\bf x})\subseteq {\cal M}(\overline{\bf x}),
%\end{equation}
%where ${\cal M}_{\tau}(\overline{\bf x})$ be the set of all $({\bf y},Y)\in \mathbb{Y}\times \mathbb{S}^n$, which together with $\overline{\bf x}$,  satisfy \eqref{eq:KKT-param}.
In addition, we know from Propositions \ref{prop:critical_cone-eq} and  \ref{prop:critical_cone_aff} that   
\[
	{\cal C}_{\tau}(\overline{\bf x})\supseteq {\rm app}(\overline{\bf y},\overline{Y}),
\]
where ${\rm app}(\overline{\bf y},\overline{Y})$ is defined in \eqref{eq:def-app}.
Therefore, Proposition \ref{thm:second-order-opt} implies that for all sufficiently small $\tau$,
\[
	\sup_{({\bf y},Y)\in{\cal M}_{\tau}(\overline{\bf x})}\left\{\langle {\bf d}, {\cal L}''_{{\bf x}{\bf x}}(\overline{\bf x},{\bf y},Y){\bf d}\rangle -\Upsilon_{\left(\overline{X}-\tau\overline{U}_{\nu}\overline{U}_{\nu}^\top\right)}^1\left(Y,g'(\overline{\bf x}){\bf d}\right) \right\} > 0, \quad \forall \; {\bf d}\in {\cal C}_{\tau}(\overline{\bf x})\setminus\{0\}.
\]
Since $\lambda_i(\overline{Y})=0$ for all $i\in \nu$, we obtain from \eqref{eq:critical-aff-max1} and \eqref{eq:Upsilon-eig} that
\[
		\Upsilon_{\left(\overline{X}-\tau\overline{U}_{\nu}\overline{U}_{\nu}^\top\right)}^1\left(\overline{Y},g'(\overline{\bf x}){\bf d}\right)=\Upsilon_{\overline{X}}^1\left(\overline{Y},g'(\overline{\bf x}){\bf d}\right).
\]
Combining the above derivations together, we know that the 
 strong second-order sufficient condition \eqref{eq:ssoc} holds at ${\bf \bar{x}}$. \qed
\end{proof}

By considering the epigraphical formulation of \eqref{opt-eq-eig}  and following the proof of  \cite[Theorem 5.20]{BShapiro00}, one can show that the  strong regularity of a KKT solution $(\overline{\bf x},\overline{\bf y},\overline{Y})$ implies the uniform quadratic growth condition at $\overline{\bf x}$. We again omit the proof here. 

Now we are ready to present the main result of this section pertaining to the necessary and sufficient conditions for the strong regularity of problem \eqref{opt-eq-eig}, which can be directly obtained by combining Theorem \ref{prop:equivalent-1} and Proposition \ref{prop:uniform-sog}.

\begin{theorem}
	Let $\overline{\bf x}\in\mathbb{X}$ be a locally optimal solution of problem \eqref{opt-eq-eig}. Suppose that Robinson's CQ \eqref{eq:RCQ}  holds at $\overline{\bf x}$. Let $(\overline{\bf y},\overline{Y})\in {\cal M}(\overline{\bf x})$. Then the following statements are equivalent:
\vskip 0.1in
\noindent
(i) the strong second-order sufficient condition \eqref{eq:ssoc} and constraint nondegeneracy \eqref{eq:nondegen}  hold at $\overline{\bf x}$;
\vskip 0.1in
\noindent
(ii) every element in $\partial F(\overline{\bf x},\overline{\bf y},\overline{Y})$ is nonsingular;
\vskip 0.1in
\noindent
(iii) $(\overline{\bf x},\overline{\bf y},\overline{Y})$ is a strongly regular solution of the generalized equation for \eqref{opt-eq-eig}.
\vskip 0.1in
\noindent
(iv) The uniform quadratic growth condition and constraint nondegeneracy \eqref{eq:nondegen}  hold at $\overline{\bf x}$.

\end{theorem}

%%%%%%%%%%%%%%%%%%%%%%%%%%%%%%%%%%%%%%%%%%%%%%%%%%%%%%%%%%%%%%%%%%%%%

\section{Conclusion}\label{sec:conclusion}

In this paper, we conduct an extensive study on the characterization of  strong regularity of the KKT solutions for a class of nonsmooth composite matrix optimization  problems (CMatOPs). Due to its non-polyhedrality, the classical perturbation analysis developed for the nonlinear programming has become inadequate for CMatOPs. We have systemically analyzed second-order variational properties of  spectral functions associated with piecewise affine symmetric functions, including the characterizations of their induced tangent sets, lineality spaces, critical cones and the sigma term. These variational results provide the necessary tools for the  characterization of the strong regularity for the general CMatOPs. The work done on CMatOPs in this paper is by no means complete. Due to the rapid advances in matrix optimization applications in emerging fields, we believe that the fundamental perturbation analysis of CMatOPs will become even more important and many other variational properties are waiting to be explored.

%%%%%%%%%%%%%%%%%%%%%%%%%%%%%%%%%%%%%%%%%%%%%%%%%%%%%%%%%%%%%%%%%%%%%

\end{document}